\documentclass[aap]{imsart}
   
\RequirePackage{amsthm,amsmath,amsfonts,amssymb}
\RequirePackage[numbers]{natbib}
\RequirePackage[colorlinks,citecolor=blue,urlcolor=blue]{hyperref}

\startlocaldefs

\usepackage[mathscr]{eucal}
\usepackage{latexsym,bm,amsfonts,amssymb,pifont,  bm,  stmaryrd,mathtools}
\usepackage{graphicx}

\usepackage{pdfsync}

\graphicspath{ {./pics/} }
 
\usepackage{enumerate}
\usepackage{pdfsync}

\newtheorem{theorem}{Theorem}[section]

\newtheorem{prop}[theorem]{Proposition}

\newtheorem{lemma}[theorem]{Lemma}

\newtheorem{Assumption}{Assumption}
\renewcommand{\theAssumption}{\Alph{Assumption}}
\makeatletter
\newcommand{\settheoremtag}[1]{
	\let\oldtheAssumption\theAssumption
	\renewcommand{\theAssumption}{#1}
	\g@addto@macro\endAssumption{
		\global\let\theAssumption\oldtheAssumption}
}
\makeatother
\theoremstyle{remark}
\newtheorem{remark}[theorem]{Remark}

\numberwithin{equation}{section}

\newcommand{\R}{\mathbb{R}}

\newcommand{\N}{\mathbb{N}}
\newcommand{\F}{\mathbb{F}}
\renewcommand{\H}{\mathbb{H}}
\renewcommand{\P}{\mathbb{P}}

\newcommand{\calc}{{\mathcal C}}
\newcommand{\calh}{{\mathcal H}}

\newcommand{\calf}{{\mathcal F}}
\newcommand{\cala}{{\mathcal A}}

\newcommand{\calr}{{\mathcal R}}

\newcommand{\calg}{{\mathcal G}}

\newcommand{\calv}{{\mathcal V}}
\newcommand{\calz}{{\mathcal Z}}

\newcommand{\ov}{\bar}
\newcommand{\un}{\underline}

\newcommand{\bta}{{\boldsymbol{\eta}}}
\newcommand{\bzeta}{{\boldsymbol{\zeta}}}

\newcommand{\bW}{{\boldsymbol{W}}}
\newcommand{\btheta}{{\boldsymbol{\theta}}}
\newcommand{\bvartheta}{{\boldsymbol{\vartheta}}}

\newcommand{\limst}{\stackrel{st}{\longrightarrow}}

\def\RR{\mathbb{R}}

\def\mE{\mathbb{E}}

\def\E{\mathbb{E}}

\def\PP {\mathbb{P}}

\def\H{\mathcal H}

\newcommand{\cf}{{\mathcal F}}

\def\la{{\lambda}}
\def\si{\sigma}

\def\eps{{\varepsilon}}

\def\la{{\lambda}}

\def\Om{{\Omega}}

\def\al{{\alpha}}

\def\Ga{{\Gamma}}
\def\ga{{\gamma}}

\def\del{\delta_n}

 \newcommand{\vp}{\varphi}
 \newcommand{\ka}{\kappa}
 
\newcommand{\lc}{\left[}
\newcommand{\rc}{\right]}

\newcommand{\bone}{\boldsymbol{1}}
\newcommand{\limL}{\stackrel{L^1}{\Longrightarrow}}
\newcommand{\lec}{\lesssim}

\def\wt{\tilde}
\def\wh{\hat}

\allowdisplaybreaks[3]

\endlocaldefs

\begin{document}

\begin{frontmatter}
	\title{Statistical inference for rough volatility:\\ Central limit theorems}
	\runtitle{Statistical inference for rough volatility}
	
	\begin{aug}
		\author[A]{\fnms{Carsten H.}~\snm{Chong}\ead[label=e1]{carstenchong@ust.hk}},
	\author[B]{\fnms{Marc}~\snm{Hoffmann}\ead[label=e2]{hoffmann@ceremade.dauphine.fr}},
	\author[C]{\fnms{Yanghui}~\snm{Liu}\ead[label=e3]{yanghui.liu@baruch.cuny.edu}},
	\author[D]{\fnms{Mathieu}~\snm{Rosenbaum}\ead[label=e4]{mathieu.rosenbaum@polytechnique.edu}}
	\and
	\author[D]{\fnms{Grégoire}~\snm{Szymanski}\ead[label=e5]{gregoire.szymanski@polytechnique.edu}}
	\address[A]{Department of Information Systems, Business Statistics and Operations Management, The Hong Kong University of Science and Technology\printead[presep={,\ }]{e1}}	
	\address[B]{CEREMADE, Université Paris Dauphine-PSL\printead[presep={,\ }]{e2}}
	\address[C]{Department of Mathematics, Baruch College CUNY\printead[presep={,\ }]{e3}}
	\address[D]{CMAP, \'Ecole Polytechnique\printead[presep={,\ }]{e4,e5}}
	\end{aug}
	
	\begin{abstract}
In recent years, there has been a substantive interest in rough volatility models. In this class of models, the local behavior of stochastic volatility is much more irregular than semimartingales and resembles that of a fractional Brownian motion with Hurst parameter $H<0.5$. In this paper, we derive a consistent and asymptotically mixed normal estimator of $H$ based on high-frequency price observations.  In contrast to previous works, we work in a semiparametric setting and do not assume any a priori relationship between volatility estimators and true volatility. Furthermore, our estimator attains a rate of convergence that is known to be optimal in a minimax sense in parametric rough volatility models.
	\end{abstract}
	
	
	
\end{frontmatter}

\section{Introduction}
For many years, continuous-time stochastic volatility models were predominantly based on stochastic differential equations driven by  Brownian motion or L\'evy processes. But more recently,   \cite{GJR} found empirical evidence that stochastic volatility is actually much rougher than semimartingales, in the sense that it locally resembles a fractional Brownian motion with   Hurst index $H< 0.5$, a statement that was further supported by other empirical work based on both return data \cite{bennedsen2016decoupling, FTW21, gatheral2020quadratic} and options data \cite{bayer2016pricing,F21,LMPR18}.

The data-driven approach of  \cite{GJR} to uncover rough volatility starts by considering high-frequency log-price data $\{x_{i\del}: i=0,\dots,[T/\del]\}$, where for example $\del = 5\,\text{min}$ and $T= 1\,\text{year}$. In a next step, daily spot variance estimates  are calculated from the formula
\begin{equation}\label{eq:RV} 
	\wh c_j = \frac{1}{k_n\delta_n} \sum_{i=1}^{k_n} (\delta^n_{(j-1)k_n+i} x)^2,\qquad j=1,\dots, [T/(k_n\del)], 
\end{equation} 
where $\delta^n_i x= x_{i\del}-x_{(i-1)\del}$ and $k_n=78 $ is the number of 5\,min increments during one trading day. 
Afterwards, realized power variations of $\log \wh c_j$, that is,
\[ m(q,\Delta)=\frac{1}{[T/\Delta]}\sum_{j=1}^{[T/\Delta]} \lvert \log \wh c_{j\Delta}-\log \wh c_{(j-1)\Delta}\rvert^q \]
are computed for different values of $q>0$ and $\Delta\in\{1\,\text{day}, 2\,\text{days}, \ldots\}$. If $\log \wh c_j$ were discrete observations of a continuous It\^o semimartingale, then one would expect that $m(q,\Delta)$ scales as $\Delta^{q/2}$, implying that the slope $\zeta_q$ in a regression of $\log m(q,\Delta)$ on $\log \Delta$ satisfies
\[ \zeta_q/q \approx \frac12. \]
However, for large set of high-frequency data, \cite{GJR} consistently found values of $\zeta_q/q < \frac12$, indicating that stochastic volatility locally behaves as a fractional Brownian motion with Hurst parameter $H<\frac12$. (To be  precise, \cite{GJR} actually computed the realized $q$-variations of daily \emph{realized} variance, that is, of $\log RV_j=k_n\delta_n \wh c_j$. But since $\log RV_{j\Delta}-\log RV_{(j-1)\Delta} = \log \wh c_{j\Delta}-\log \wh c_{(j-1)\Delta}$, this amounts to the same as computing $m(q,\Delta)$.)

As was pointed out by \cite{bennedsen2016decoupling, FTW21}, the above approach rests on the assumption that realized variances have the same scaling behavior as the true unobserved volatility. At the same time, it is well known (see e.g., \cite[Chapter 8]{AJ}) that in the absence of jumps and if volatility is a semimartingale, spot variance estimators of the type \eqref{eq:RV} converge to true spot variance plus a small  modulated white noise. In a first attempt to take estimation errors for spot variance into account, \cite{bennedsen2016decoupling, FTW21} assume that
\begin{equation}\label{eq:equal} 
	\log \wh c_j = \log \text{true spot variance}_j + \eps_j, 
\end{equation} 
where $\eps_j$ is a zero-mean iid sequence that is independent of everything else. Under assumption \eqref{eq:equal}, \cite{bennedsen2016decoupling, FTW21} derive consistent estimators of the roughness parameter $H$ in parametric rough volatility models and uphold the conclusion of \cite{GJR} that volatility is rough in a large set of financial time series. We also refer to \cite{bolko2020GMM}, where the authors assume \eqref{eq:equal} with slightly different assumptions on $(\log \wh c_j,\eps_j)$, and to \cite{Szymanski22}, where a central limit theorem (CLT) for $H$ is established under \eqref{eq:equal} (see also \cite{WANG2021}). 

This paper aims to substantially generalize the aforementioned results in two directions: first, we establish consistent  and asymptotically mixed normal estimators of $H$ in a semiparametric setting, where except for $H$ all other model ingredients are fully nonparametric; and second, we shall do so without assuming any relationship (such as \eqref{eq:equal}) between volatility proxies and true volatility. The rate of convergence of our best estimator is 
\begin{equation}\label{eq:opt} 
	\del^{-1/(4H+2)},
\end{equation} which as our companion paper \cite{Paris} shows is optimal in a minimax sense in parametric rough volatility models.


The remaining paper is structured as follows: in Section~\ref{sec:2}, after introducing the model assumptions, we state the main technical result of this paper, Theorem~\ref{thm:CLT-main}, a CLT for \emph{volatility of volatility (VoV)} estimators in a rough volatility framework.  Section~\ref{sec:4} discusses how we turn Theorem~\ref{thm:CLT-main} into  rate-optimal and feasible estimators of $H$. In addition to a usual application of the delta method, the rough volatility setting requires us overcome two distinct challenges: 
\begin{itemize}
	\item eliminating a nonnegligible asymptotic bias term for which we do not have a sufficiently fast estimator;
	\item constructing an optimal sequence $k_n$ for spot variance estimation that depends on the unknown parameter $H$ without losing a marginal bit of convergence rate.
\end{itemize}
Our final  estimator $H_n$ for $H$ is given in Equation~\eqref{eq:finalest}. As Theorems~\ref{thm:est} and \ref{thm:final} show, $H_n$ is a feasible and rate-optimal estimator of $H$ if $H\in(0,\frac12)$ and is equal to $\frac12$ with high probability if volatility is a continuous It\^o semimartingale. In Section~\ref{sec:sim}, we report the results of a short simulation study.  Section~\ref{sec:3} contains the main steps of the proof of Theorem~\ref{thm:CLT-main}, with certain technical details postponed to Appendices~\ref{sec:approx2}--\ref{sec:details}.

 In what follows, we write $A\lec B$ if there is a constant $C\in(0,\infty)$ that does not depend on any important parameter such that $A\leq CB$. Furthermore, if $A_n(t)$ and $B_n(t)$ are stochastic processes, we write $A_n\approx B_n$ if $\E[\sup_{t\in[0,T]} \lvert A_n(t)-B_n(t)\rvert] \to 0$ as $n\to\infty$.  For two sequences $a_{n}$ and $b_{n}$ we write $a_{n}\sim b_{n}$ if $a_{n}/b_{n}\to 1$ as $n\to\infty$.  If $x\in\R^n$, we denote its Euclidean norm by $\lvert x\rvert$. For any $\al\in \R$,  we write $x_+^\al=x^\al$ if $x>0$ and $x_+^\al =0$ otherwise. We also use the notation $\N=\{1,2,\dots\}$ and $\N_0=\{0,1,2,\dots\}$.

\section{Model and CLT for VoV estimators} \label{sec:2}
 
  On a filtered probability space $(\Om,\calf,\F=(\calf_t)_{t\geq0},\P)$ satisfying the usual conditions, we assume that the log-price $x$ of an asset is given by a continuous It\^o semimartingale of the form  
  \begin{equation}\label{eqn.s}
x_{t} = x_0+ \int_0^t b_{s}ds +\int_0^t\si_{s}dW_{s},\qquad t\geq0. 
\end{equation}
We assume that the squared volatility process $c=\si^2$ satisfies
\begin{equation} \label{eq:si} c_t = c_0 +\int_0^t a_s ds+   \int_0^t \wt g(t-s)\wt\eta_s ds + \int_0^t g(t-s)(\eta_s dW_s + \hat \eta_s d\hat W_s),  \end{equation}
where
\begin{equation}\label{eq:eta}\begin{split}
	\eta^2_t	&=\eta^2_0+\int_0^t a^\eta_s ds+ \int_0^t \wt g^\eta(t-s)\wt\theta_s ds+\int_0^t g^\eta(t-s)\btheta_s  d\bar \bW_s,\\
		\hat\eta^2_t	&=\hat\eta_0+\int_0^t a^{\hat \eta}_s ds+ \int_0^t \wt g^{\hat \eta}(t-s)\wt\vartheta_s ds+\int_0^t g^{\hat \eta}(t-s)\bvartheta_s d\bar \bW_s.
\end{split}\end{equation}
The ingredients of \eqref{eqn.s}--\eqref{eq:eta} are assumed to satisfy the following conditions.
\settheoremtag{CLT}
\begin{Assumption}\label{ass:CLT}
	Suppose that the log-price process $x$ is given by \eqref{eqn.s} with the following specifications:
	\begin{enumerate}
		\item There is $H\in(0,\frac12]$ such that the squared volatility process $c_t=\si_t^2$ satisfies \eqref{eq:si} with  $\eta$ and $\wh \eta$ given by \eqref{eq:eta}. The variables $x_0$, $c_0$, $\eta_0^2$ and $\wh \eta_0^2$ are $\calf_0$-measurable.
		\item The processes $a$, $b$, $a^\eta$ and $a^{\hat \eta}$ (resp., $\btheta$ and $\bvartheta$) are adapted and locally bounded real-valued (resp., $\R^{1\times 4}$-dimensional) processes. Moreover, for all $T>0$, we assume that 
		\begin{equation}\label{eq:b} 
			\lim_{h\to0}	\sup_{s,t\in[0,T], \lvert s-t\rvert\leq h} \bigl\{ \E[1\wedge \lvert b_t-b_s\rvert]+\E[1\wedge \lvert a_t-a_s\rvert]\bigr\} = 0.
		\end{equation}
		\item The processes $\wt \eta$, $\wt \theta$ and $\wt\vartheta$ are adapted, locally bounded and for all $T>0$, there is $K_T\in(0,\infty)$ such that  
		\begin{equation}\label{eq:Holder} 
			\sup_{s,t\in[0,T]}\bigl\{\E[1\wedge \lvert \wt\eta_t-\wt\eta_s\rvert]+\E[1\wedge \lvert \wt\theta_t-\wt\theta_s\rvert]+\E[1\wedge \lvert \wt\vartheta_t-\wt\vartheta_s\rvert]\bigr\}\leq K_T\lvert t-s\rvert^H.
		\end{equation}
		\item The processes $W$ and $\hat W$ are independent standard $\F$-Brownian motions and $\bar \bW$ is a four-dimensional $\F$-Brownian motion that is jointly Gaussian with $(W,\hat W)$. The components of $\bar \bW$ may depend on each other and on $(W,\hat W)$.
		\item We have
	\begin{equation}\label{eq:g} 
			\begin{aligned}
			g(t)	&=g_H(t)+g_0(t), &g^\eta(t)	&=g_{H_\eta}(t)+g^\eta_0(t), &g^{\hat \eta}(t)	&=g_{H_{\hat \eta}}(t)+g^{\hat \eta}_0(t),\\
			\wt g(t)	&=g_{\wt H}(t)+ \wt g_0(t), &\wt g^\eta(t)	&=g_{\wt H_\eta}(t)+\wt g^\eta_0(t),&
			\wt g^{\hat \eta}(t)	&=g_{\wt H_{\hat \eta}}(t)+\wt g^{\hat \eta}_0(t),
		\end{aligned}
	\end{equation} 
		where
		\begin{equation}\label{eq:gH} 
			g_H(t)=K_H^{-1} t_+^{H-1/2},\qquad K_H=\frac{\Ga(H+\frac12)}{\sqrt{\sin(\pi H)\Ga(2H+1)}},
		\end{equation} 
		and  $H_\eta,H_{\hat \eta}\in(0,\frac12]$,  $\wt H,\wt H_\eta,\wt H_{\hat \eta}\in [ H,\frac12]$
		and $g_0,g^\eta_0,g^{\hat \eta}_0, \wt g_0, \wt g^\eta_0, \wt g^{\hat \eta}_0 \in C^1([0,\infty))$ are functions vanishing at $t=0$. 
	\end{enumerate}
\end{Assumption}

Let us comment on the conditions imposed in Assumption~\ref{ass:CLT}. Except for the parameter $H$, the assumptions on  $x$, $c$, $\eta$ and $\wh \eta$ are fully nonparametric and designed in such a way that it contains the rough Heston model \cite{el2019roughening, ER19} as an example, which is a particular important one as it is founded in the microstructure of financial markets \cite{el2018microstructural,jusselin2020noarbitrage}. Note that we allow $c$, $\eta$ and $\wh\eta$ to have   both a usual (differentiable) and a rough (non-differentiable) drift. Moreover, by considering $W$, $\wh W$ and $\bar\bW$, we allow for the most general dependence between the Brownian motions driving $x$, $c$, $\eta$, $\wh \eta$. Also note that $H_\eta$ and $H_{\wh \eta}$ are not coupled with $H$, so the VoV processes $\eta$ and $\wh\eta$ can be much rougher than the volatility process $c$ itself.

The structural assumptions in  \eqref{eq:eta} for $\eta^{2}$ and $\hat{\eta}^{2}$ are important and cannot be replaced by just merely assuming that $\eta^2$ and $\hat{\eta}^2$ are $H$-H\"older regular (see the paragraph after \eqref{eq:Q-diff} for the reason). At the same time,  \eqref{eq:eta} is a mild assumption: it is not only satisfied by the rough Heston model but by any model in which VoV is assumed to be continuous and stationary, thanks to the Wold--Karhunen representation theorem. (Due to the coefficients $\btheta_s$ and $\bvartheta_s$, the processes $\eta^2$ and $\hat{\eta}^2$ do \emph{not} have to be stationary.) Also, in the estimation of VoV in a semimartingale context,  \eqref{eq:eta} with $H_\eta=H_{\hat \eta}=\frac12$ is typically assumed; see \cite{LLZ22, Vetter15}.

\begin{remark}[Rough volatility vs.\ long memory vs.\ jumps]~
\begin{enumerate}
\item There is long-standing debate in the literature whether volatility has long memory or not (see, e.g.,  \cite{ABDL03,CR98,Corsi09,Diebold01}).
Because we include various $g_0$-functions in \eqref{eq:g}, the kernels in \eqref{eq:si} and \eqref{eq:eta} are only specified around $t=0$. In particular, $H$, $H_\eta$ and $H_{\wh\eta}$ are parameters of roughness and are not related to long-range dependence / long-memory / persistence. In particular, in our model, volatility can be rough and have long memory at the same time, which is important as \cite{bennedsen2016decoupling, SLY21, SY22} point out.
\item In some parts of the previous literature, the notion of ``roughness'' is used as a synonym for jumps or discontinuities (see, e.g., \cite{ABD07,BOLLERSLEV2016464}). Since our model neither includes price nor volatility jumps, one may ask whether rough volatility can be explained through jumps, which are a well-documented feature of high-frequency price series \cite{AitSahalia09}. In this respect, we first note that   \cite{chong2022short} shows that rough volatility (in the sense that $H<\frac12$ in \eqref{eq:si}) can be statistically distinguished from both price and volatility jumps. In order to make the estimators developed in this paper robust to jumps, a natural idea is to include truncation in \eqref{eqn.chat.dx} and \eqref{eq:Vtilde} below. We leave it to future work to study the details of such an extension.
\item While    this paper focuses  on   the rough case $H\leq 1/2$, we conjecture that the results of this paper can be extended to $H<\frac34$ without major obstacles. Since noncentral limit theorems are known to appear for variation statistics of directly observable processes if $H\geq \frac34$ (see, e.g., \cite{NNT10}), we do not know whether our results extend to $H\geq \frac34$. 
\end{enumerate}
\end{remark}

If $c$ was directly observable, a classical way to feasibly estimate $H$ would be to prove a joint CLT for \emph{realized autocovariances} $\del^{1-2H}\sum_{i=1}^{[T/\del]-\ell} \delta^n_i c\, \delta^{n}_{i+\ell}c$ with different values of $\ell\in\N_0$ and then to obtain an estimator of $H$ from the ratio of two such functionals; see \cite{BN11, CDL22, CDM22, CHPP13, GH07, IL97, LT20}. Since we do not observe $c$, we first consider spot variance estimators 
\begin{eqnarray}\label{eqn.chat.dx}
\hat c^n_{t,s}=\frac{1}{k_n\del}\hat{C}^{n}_{t, s} ,\quad	\hat{C}^{n}_{t, s} = \sum_{i=[ t/\delta_{n}]}^{[ (t+s)/\delta_{n} ]-1}
	(\delta^{n}_{i} x)^{2}, \quad 
	\delta^{n}_{i} x = x_{i\delta_{n}} - x_{(i-1)\delta_{n}} ,
\end{eqnarray}
where $k_n \in\N$ and $k_n\sim \theta\delta_n^{-\kappa}$ for some $\kappa,\theta>0$. Then we
form realized autocovariances of these spot variance estimators by defining
\begin{equation}\label{eq:Vtilde} 
	\begin{split}
		\tilde{V}^{n,\ell,k_n}_{t}  &= (k_n\delta_n)^{1-2H} \frac{1}{k_{n}} \sum_{i=1}^{[ t/\delta_{n} ]-(\ell+2)k_n+1 } \bigl(   \hat{c}^{n}_{(i+k_n)  \delta_{n},k_{n}\delta_{n}}   -      \hat{c}^{n}_{i \delta_{n}, k_{n}\delta_{n}}  \bigr)\\
		&\quad\times\bigl(  \hat{c}^{n}_{(i+(\ell+1)k_n)  \delta_{n},k_{n}\delta_{n}}   -      \hat{c}^{n}_{(i+\ell k_n) \delta_{n}, k_{n}\delta_{n}}  \bigr)
	\end{split}
\end{equation}
for $\ell\geq0$.
Note that we write $[x]$ and $\{x\}$ for the integer and fractional part of  $x$, respectively. The normalization in the last line is chosen in such a way that $\wt V^{n,\ell,k_n}_t$ converges in probability. In the semimartingale context (with $H=\frac12$ and $\ell=0)$, the functional $\wt V^{n,0,k_n}_t$ was used in \cite{Vetter15} (see also \cite{Gloter00, LLZ22}) to estimate the integrated VoV process $\int_0^t (\eta_s^2+\wh\eta_s^2) ds$ (to be very precise, this is actually the integrated \emph{variance of variance}). Still in the semimartingale framework, functionals similar to \eqref{eq:Vtilde} have also been investigated in the literature to estimate the leverage effect; see \cite{AFLWY17,AFL13,BR12,KX17,Vetter12,WM14}. 

 To state a CLT for $\wt V^{n,\ell,k_n}$ for $H<\frac12$, we have to introduce some additional notation: for $n\in\N$,  $h>0$ and a function $f:\R\to\R$, we define the forward and central difference operators by $$\Delta^n_h f(t)=\sum_{i=0}^n (-1)^{n-i}\binom{n}{i}f(t+ih),\qquad \delta^n_h f(t)= \sum_{i=0}^n(-1)^i \binom{n}{i} f(x+ (\tfrac n2-i) h),$$
 respectively. For $n=1$, we simply write $\Delta_h f(t)=\Delta^1_h f(t)=f(t+h)-f(t)$ and $\delta_h f(t)=\delta^1_h f(t)=f(t+\frac h2)-f(t-\frac h2)$. Moreover, given $\al\in\R$, we use the shorthand notation $\Delta^n_h t^\al_+$ or $\Delta^n_h \lvert t\rvert^\al$ for $\Delta^n_h f(t)$ where   $f(t)=t_+^\al$ or $f(t)=\lvert t\rvert^\al$ ($\delta^n_h t^\al_+$ and $\delta^n_h \lvert t\rvert^\al$  are used similarly). Finally, for any $d\in\N$, we use $\stackrel{st}{\Longrightarrow}$ to denote functional stable convergence in law in the space of c\`adl\`ag functions $[0,\infty)\to\R^d$ equipped with the local uniform topology. The following CLT is the main technical result of this paper.

\begin{theorem}\label{thm:CLT-main}
Let $d\in\N$ and $\ell_1,\dots,\ell_d\geq2$ be integers. Furthermore, consider deterministic integer sequences $(k_n^{(1)})_{n\in\N}, \dots, (k_n^{(d)})_{n\in\N}$ such that for some $\kappa\in[\frac{2H}{2H+1},\frac12]$ and $\theta_1,\dots,\theta_d\in(0,\infty)$ we have $k_n^{(j)}\sim \theta_j\del^{-\kappa}$ for all $j=1,\dots,d$. For each $j=1,\dots, d$, let
\begin{equation}\label{eq:calz} 
	\calz^{n,j}_t = \del^{-(1-\kappa)/2}(\wt V^{n,\ell_j, k_n^{(j)}}_t - V^{\ell_j}_t - \cala^{n,\ell_j,k_n^{(j)}}_t),
\end{equation}
where for $\ell\geq2$, we define
\begin{equation}\label{eq:Vt} 
	V^\ell_t=	\Phi^H_\ell\int_0^t (\eta_s^2+\wh\eta_s^2) ds
\end{equation}
with
	\begin{equation}\label{eq:Phi} \begin{split}
	\Phi^H_\ell &=\frac{\delta_1^4 \lvert \ell\rvert^{2H+2}}{2(2H+1)(2H+2)}\\
	&=\frac{(\ell+2)^{2H+2}-4(\ell+1)^{2H+2}+6\ell^{2H+2} -4( \ell-1)^{2H+2} + (\ell-2)^{2H+2}}{2(2H+1)(2H+2)}\end{split}	\end{equation}
and for a general integer sequence $k_n$,
	\begin{equation}\label{eq:bias}\begin{split}
		\cala^{n,\ell,k_n}_t	& =-\frac{2K_H^{-1}}{H+\frac12}(k_{n}\delta_{n})^{-1/2-H} 
		\int_{0}^t \frac{1}{k_{n}}\sum_{i=0}^{k_n-1} \Delta^3_1 (\ell-1-\tfrac{i+\{u/\del\}}{k_n})_+^{H+1/2}\\
		&\quad\times\int_{[u/\del]\del}^u \si_v dW_v(\si_{u}	\eta_{u} -\si_{[u/\del]\del}\eta_{[u/\del]\del})du.
\end{split}\end{equation}
Under Assumption~\ref{ass:CLT}, the process $\calz^n_t=(\calz^{n,1}_t,\dots,\calz^{n,d}_t)^T$ satisfies the joint CLT
\begin{equation}\label{eq:CLT} 
\calz^n \stackrel{st}{\Longrightarrow} \calz,
\end{equation}
where $\calz=((\calz^1_t,\dots, \calz^d_t)^T)_{t\geq0}$ is a continuous $\R^d$-valued process that is defined on a very good filtered extension $(\ov \Om, \ov\calf,\ov\F=(\ov\calf_t)_{t\geq0},\ov\P)$ of the original probability space (see e.g. \cite[Chapter 2.1.4]{JP}) and conditionally on $\calf$ is a centered Gaussian process with independent increments and $\calf$-conditional covariance function
\begin{equation}\label{eq:var} 
\calc_t^{jj'}=	\ov\E[\calz^{j}_t\calz^{j'}_t\mid \calf]= \sum_{\nu=1}^3 \ga_\nu^{\ell_j,\theta_j,\ell_{j'},\theta_{j'}}(H)  \Ga_\nu(t).
\end{equation}
In the last line, 
\begin{equation}\label{eq:terms} 
	\Ga_1(t)= \int_0^t \si_s^8 ds,\quad \Ga_2(t)= \int_0^t (\eta^2_s+\wh\eta^2_s)^2 ds,\quad \Ga_3(t)=\int_0^t \si_s^4 (\eta^2_s+\wh\eta^2_s) ds
\end{equation}  
and for arbitrary $\ell,\ell'\geq2$ and $\theta,\theta'\in(0,\infty)$,
\begin{equation}\label{eq:ga}\begin{split}
	\ga_1^{\ell,\theta,\ell',\theta'}(H)	&=    \frac{\delta_{\theta}^{4}\delta_{  {\theta'}}^{4}|
    \ell \theta- {\ell'}  {\theta'}    |^{3}}{3(\theta  {\theta'})^{    2H+2}}    
    \mathbf{1}\biggl\{\ka =\frac{2H}{2H+1}  \biggr\}
    , \\
	\ga_2^{\ell,\theta,\ell',\theta'}(H)	&= \begin{multlined}[t][0.725\textwidth]
		\frac{\Ga(1+2H)^2(1-1/\cos(2\pi H))}{4\Ga(6+4H)(\theta\theta')^{2H+2}}\\
		\times\delta^4_{\theta}\delta^4_{\theta'}\bigl[\lvert \ell\theta-\ell'\theta'\rvert^{4H+5}+\lvert \ell\theta+\ell'\theta'\rvert^{4H+5}\bigr],\end{multlined}\\
	\ga_3^{\ell,\theta,\ell',\theta'}(H)&=- \frac{\delta^4_{\theta}\delta^4_{\theta'} [\lvert \ell\theta+\ell'\theta'\rvert^{2H+4}+ \lvert \ell\theta-\ell'\theta' \rvert^{2H+4}]}{8(H+\frac12)(H+1)(H+\frac32)(H+2)(\theta\theta')^{2H+2}}\mathbf{1}\biggl\{\ka =\frac{2H}{2H+1}  \biggr\}.
\end{split}\end{equation}
If $H=\frac14$, $\ga_2^{\ell,\theta,\ell',\theta'}(H)$ is defined via continuous extension by 
\begin{equation}
	\ga_2^{\ell,\theta,\ell',\theta'}(\tfrac14)=	\frac{\delta^4_{\theta}\delta^4_{\theta'}\bigl[\lvert\ell\theta- \ell'\theta'\rvert^{6}\log\lvert\ell\theta- \ell'\theta'\rvert+\lvert  \ell\theta+\ell'\theta'\rvert^{6}\log\lvert \ell\theta+\ell'\theta'\rvert\bigr]}{5760(\theta'\theta')^{5/2}}.
\end{equation}
\end{theorem}
 
The main steps of the proof of Theorem~\ref{thm:CLT-main} will be given in Section~\ref{sec:3}. 
\begin{remark}
The  distribution of the limit $\calz$ is mixed normal, with a fully explicit covariance function $\calc_t$. To make this CLT feasible, we exhibit consistent estimators of $\calc_t$ in Proposition~\ref{prop:asym}.
\end{remark}

\begin{remark} 
Taking $\kappa =\frac{2H}{2H+1}$ gives the optimal window size for spot variance estimation, as it  balances squared bias (of order $(k_{n}\delta_{n})^{2H}$)  and variance (of order $1/k_{n}$). The reader is referred to Remark 8.3 in \cite{AJ} and to the discussion after  \eqref{eq:Zt} below for more details. A formal proof of the optimality statement is  the subject of  our companion paper \cite{Paris}. 
\end{remark}

\begin{remark}\label{rem:bias}
In Section \ref{sec:4} we will see that   our estimator of $H$ is based on  the convergence of \eqref{eq:CLT}, in which  $\cala^{n,\ell,k_n}_t $ plays the role of an asymptotic bias that, in general, needs to be eliminated from  $\tilde{V}^{n,\ell,k_{n}}_{t}$. As Lemma~\ref{lem:Q3} below shows, this bias term comes from the cross-covariance term $Z^{n,\ell}_3$ defined in \eqref{eq:Zt} and is of  order  $O_\P(k_n^{-1/2-H})$. Thus, it does not converge to $0$ sufficiently fast to be negligible in \eqref{eq:calz}   unless    $\kappa>\frac{1}{2H+2}$. For the optimal $\kappa=\frac{2H}{2H+1}$, this means that no   debiasing is needed if and only if  $\frac{2H}{2H+1} >\frac{1}{2H+2}$, or $H> \frac14(\sqrt{5}-1)\approx 0.3090$. 
Unfortunately, we were not able to find a sufficiently fast debiasing statistic for $\cala^{n,\ell,k_n}_t$, which is why we  resort to a nonstandard  procedure in Section~\ref{sec:4}.
\end{remark}

\begin{remark} Notice that the optimal rate of convergence of $\wt V^{n,\ell,k_n}_t$ (after removing the bias) is $\del^{-1/(4H+2)}$, which  approaches $\del^{-1/2}$ as $H\downarrow 0$ and $\del^{-1/4}$ as $H\uparrow \frac12$ and is monotone in between. The fact that the convergence rate is faster for small $H$ compared to the semimartingale case $H=\frac12$ (see \cite{Vetter15}) seems counter-intuitive since the spot variance estimator $\wh c^n_{t,k_n\delta_n}$ should be less precise if $c$ is rough. This is true, and the rate of convergence of $\wh c^n_{t,k_n\delta_n}$ to $c_t$ indeed decreases with smaller $H$. But our goal is not to estimate $c_t$ but its roughness, and for that purpose, it is sufficient to have access to  local averages of $c_t$. As we shall see in the discussion after \eqref{eq:Zt} below, the (suitably normalized) realized autocovariances of these local averages (which is the term $Z^{n,\ell}_2$ from \eqref{eq:Zt}) converges with a rate that is \emph{independent} of $H$. At the same time, the size of the pre-estimation error (which is the term $Z^{n,\ell}_1$ from \eqref{eq:Zt}) decreases with $H$. As a consequence, the signal-to-noise ratio improves with smaller $H$, which is why we obtain a better rate of convergence. In conclusion, while it is more difficult to pin down the value of a rough path at a given time point, it is easier to recognize a rough than a smooth path in the presence of noise.
\end{remark}

\begin{remark}
Theorem~\ref{thm:CLT-main} is a joint functional stable CLT for $d$ realized autocovariances of the form \eqref{eq:Vtilde} with potentially different lags $\ell_j$ and sequences $k_n^{(j)}$ that are of the same asymptotic order $\del^{-\kappa}$ but potentially with different constants $\theta_j$. We include this multivariate CLT not  for the sole purpose of pursing utmost generality but really because we need it in Section~\ref{sec:4}, when we construct a rate-optimal estimator of $H$.  Similarly, since we need Theorem~\ref{thm:CLT-main} with a general $\kappa$ in Section~\ref{sec:4}, we prove   Theorem~\ref{thm:CLT-main} for all $\kappa\in[\frac {2H}{2H+1},\frac12]$ and not just for the optimal one. 
\end{remark}

\begin{remark}
For technical reasons, we need $\ell\geq3$ in Section~\ref{sec:4}, which is why we only consider $\ell\geq2$ in Theorem~\ref{thm:CLT-main}. If $\ell=0,1$, an additional dominating bias term appears (similarly to the situation considered in \cite{Vetter15} for semimartingales). In fact, following similar ideas to \cite{Vetter15} to remove this extra bias, one can show that Theorem~\ref{thm:CLT-main} extends to $\ell=0,1$ if one replaces $\wt V^{n,\ell,k_n}_t$ by
\begin{equation}\label{eq:01}
\begin{dcases}  \wt {\mathbb{V}}^{n,0,k_n}_t =\wt V^{n,0,k_n}_t - \frac{4}{3}(k_n\delta_n)^{-2H-1}\sum_{i=1}^{[t/\delta_n]-2k_n+1} \frac{1}{2k_n} \sum_{j=0}^{2k_n-1} (\delta^n_{i+j} x)^4 &\text{if } \ell=0,\\
\wt {\mathbb{V}}^{n,1,k_n}_t =\wt V^{n,1,k_n}_t +\frac{2}{3}(k_n\delta_n)^{-2H-1}\sum_{i=1}^{[t/\delta_n]-3k_n+1} \frac{1}{k_n} \sum_{j=0}^{k_n} (\delta^n_{i+j} x)^4&\text{if } \ell=1.
\end{dcases}
\end{equation}
\end{remark}

\section{Debiasing and rate-optimal inference for $H$}\label{sec:4}

There are two main challenges in deriving a rate-optimal estimator of $H$ on the basis of Theorem~\ref{thm:CLT-main}: first, if $H$ is small, $\wt V^{n,k_n,\ell}_t$ has a nonnegligible bias that dominates the CLT fluctuations; and second,   the optimal window size $k_{n}$ depends on the unknown roughness parameter $H$ itself.

In order to account for the asymptotic bias, our strategy is to consider multiple window sizes $k_n$ and combine the resulting $\wt V^{n,\ell,k_n}_t$'s in a very specific way that cancels the bias terms up to a negligible contribution. 
For $M\in\N$, let us   introduce 
the Vandermonde matrix
\begin{equation}\label{eq:vm} 
	V_M = \begin{pmatrix} 1 & 1 & 1 & \cdots& 1\\ 1 & 2^{-1} &  3^{-1} & \cdots & M^{-1}\\ 1 & 2^{-2} & 3^{-2} & \cdots & M^{-2}\\
		\vdots& \vdots& \vdots& \ddots & \vdots\\
		1 & 2^{-(M-1)} & 3^{-(M-1)} & \cdots & M^{-(M-1)} \end{pmatrix},
\end{equation}  
which has an inverse $V_M^{-1}$ by a standard result from linear algebra. Thus, we can define
\begin{equation}\label{eq:w}
	\wt w(M)=V_M^{-1}e_M,\quad e_M = (0,\ldots,0,1)^T,\quad w(M)=\frac{\wt w(M)}{  \lvert \wt w(M)\rvert},
\end{equation}
so that $w(M)$ is the normalized  last column of $V_M^{-1}$. The following proposition shows that a very specific linear combination of $\wt V^{n,\ell,mk_n}_t$ for different $m$'s  removes the dominating part of the bias. While Theorem~\ref{thm:CLT-main} only requires $\ell\geq2$, we have to impose $\ell\geq3$ from now on.
\begin{prop}\label{prop:debias}
	Suppose that the conditions of Theorem~\ref{thm:CLT-main} are satisfied with $H\in(0,\frac12]$ and that $k_n\sim \theta \del^{-\kappa}$ for some $\theta>0$ and $\kappa\in[\frac{2H}{2H+1},\frac12]$. Furthermore, assume that $\ell\geq3$. Defining
	\begin{equation}\label{eq:M} 
		M= M(H)=[\tfrac12-H + \tfrac1{4H} ]+1,
	\end{equation}
	we have that 
	\begin{equation}\label{eq:nobias} 
		\sum_{m=1}^M w(M)_m m^{1/2+H} \wt V^{n,\ell,mk_n}_t - V^{\ell}_{t}
		\sum_{m=1}^M w(M)_m m^{1/2+H} = O_\P((k_n\del)^{1/2}).
	\end{equation}
\end{prop}  

Of course, the left-hand side of \eqref{eq:nobias} multiplied by $(k_n\del)^{-1/2}$ satisfies a CLT, but since we do not need this in the following, we only prove the simpler version \eqref{eq:nobias}.
\begin{proof}[Proof of Proposition~\ref{prop:debias}]
	By Theorem~\ref{thm:CLT-main}, it suffices to show that
	\begin{equation}\label{eq:toshow-2} 
		\sum_{m=1}^M w(M)_m m^{1/2+H} \cala^{n,\ell,mk_n}_t =o_\P((k_n\del)^{1/2} ),
	\end{equation}
	where $\cala^{n,\ell,k_n}_t$ is defined in \eqref{eq:bias}.
	For
	$\ell\geq3$, the function $v\mapsto \Delta^3_1 G_H(\ell-1-\frac{i}{k_n} + v)$ is smooth on $[ -\frac12,0]$,
	with derivatives $(\Delta^3_1 G_H)^{(j)}(\ell-1-\frac{i}{k_n} + v)$ that are uniformly bounded in $v\in[ -\frac12,0]$, 
	$i$, $k_n$ and $j=0,\dots, M$.  Thus, by \eqref{eq:bias} and Taylor's theorem, 
	\begin{equation}\label{eq:A}\begin{split} 
			\cala^{n,\ell,k_n}_t	& =-  2
			(k_{n}\delta_{n})^{-1/2-H} \sum_{j=0}^{M-1} \frac1{j!} 
			\int_{0}^t \frac{1}{k_{n}}\sum_{i=0}^{k_n-1} ( \Delta^3_1 G_H)^{(j)}(\ell-1-\tfrac{i}{k_n})\bigl(-\tfrac{\{u/\del\}}{k_n}\bigr)^j\\
			&\quad\times (y_{u}-y_{[u/\del]\del})(\si_{u}	\eta_{u}-\si_{[u/\del]\del}\eta_{[u/\del]\del}) du+ O_\P(k_n^{-1/2-H-M}).
	\end{split}\end{equation}
	As the reader can  verify, by our definition of $M$ in \eqref{eq:M}, we have that $k_n^{-1/2-H-M}=o((k_n\del)^{1/2})$. 
	
	Next, we recognize that the sum over $i$ is a Riemann sum approximation of the integral $\int_0^1 ( \Delta^3_1 G_H)^{(j)}(\ell-1-v) dv$. By the Euler--Maclaurin formula (see e.g., \cite[Theorem~1]{Lampret01}), there are finite numbers $\xi^\ell_{j,j'}$ such that 
	\begin{align*} 
		\frac{1}{k_{n}}\sum_{i=0}^{k_n-1} ( \Delta^3_1 G_H)^{(j)}(\ell-1-\tfrac{i}{k_n}) = \sum_{j'=0}^{M-1}   \xi^\ell_{j,j'} k_n^{-j'} + O(k_n^{-M}).
	\end{align*}
	Inserting this back into \eqref{eq:A}, we can ignore the $O(k_n^{-M})$-term as before. In fact, we only have to keep those terms for which $j+j'\leq M-1$. Thus, letting
	\begin{align*}
		\Xi^{n,\ell}_{p}(t) &=  -2\delta_{n}^{-1/2-H} \sum_{j,j'=0}^{M-1} \bone_{\{j+j'=p\}} \frac{\xi^\ell_{j,j'}}{j!} 	 \\
		&\quad\times\int_0^t (-\{u/\del\})^j (y_{u}-y_{[u/\del]\del})(\si_{u}	\eta_{u}-\si_{[u/\del]\del}\eta_{[u/\del]\del}) du, 
	\end{align*} 
	we have that 
	\[ \cala^{n,\ell,k_n}_t= k_n^{-1/2-H} \sum_{p=0}^{M-1} \Xi^{n,\ell}_p(t) k_n^{-p} + o_\P((k_n\del)^{1/2}). \]
	Note that $\Xi^{n,\ell}_p(t)$ depends on $\del$ but not on  $k_n$. Therefore, applying the previous identity to $mk_n$ for $m=1,\dots, M$, we arrive at the following systems of equations:
	\begin{equation}\label{eq:system} 
		m^{1/2+H}\cala^{n,\ell,mk_n}_t=    \sum_{p=0}^{M-1} m^{-p} \Xi^{n,\ell}_p(t) k_n^{-1/2-H-p} + o_\P((k_n\del)^{1/2}),\quad m=1,\dots, M.
	\end{equation}  
	Thus, introducing 
	\begin{align*}
		\un\cala^{n,\ell,k_n}_t&=(1^{1/2+H}\cala^{n,\ell,k_n}_t,\dots, M^{1/2+H}\cala^{n,\ell,Mk_n}_t)^T, \\
		\un \Xi^{n,\ell}(t)&=(\Xi^{n,\ell}_0(t) k_n^{-1/2-H-0},\dots,\Xi^{n,\ell}_{M-1}(t) k_n^{-1/2-H-(M-1)})^T,
	\end{align*} 
	we can rewrite \eqref{eq:system} as
	\[ \un\cala^{n,\ell,k_n}_t = V^T_M\un \Xi^{n,\ell} (t)+ o_\P((k_n\del)^{1/2}), \]
	where $V_M$ is the Vandermonde matrix \eqref{eq:vm}. 
	Thus, by the definition of $w(M)$ (see \eqref{eq:w}),
	\begin{align*}
		\sum_{m=1}^M w(M)_m m^{1/2+H}\cala^{n,\ell, m k_n}_t&= w(M)^T\un \cala^{n,\ell,k_n}_t\\
		& = \lvert \wt w(M)\rvert^{-1} e_M^T (V_M^{-1})^T V^T_M\un \Xi^{n,\ell}(t) + o_\P((k_n\del)^{1/2}) \\
		&= \lvert \wt w(M)\rvert^{-1}  \Xi^{n,\ell}_{M-1}(t)k_n^{-1/2-H-(M-1)}+ o_\P((k_n\del)^{1/2}).
	\end{align*}  
	Since $\Xi^{n,\ell}_{M-1}(t)=O_\P(1)$, \eqref{eq:toshow-2} follows from our choice of $M$.
\end{proof}

We now explain how to implement this debiasing procedure in practice.  For the remaining part of this section, we assume that 
\begin{equation}\label{eq:notzero} 
	\int_0^t (\eta_s^2+\wh\eta_s^2) ds >0\quad \text{a.s.}
\end{equation}
(or, equivalently, all forthcoming statements are valid without \eqref{eq:notzero} but in restriction to the set $\{\int_0^t (\eta_s^2+\wh\eta_s^2) ds >0\}$).
Define
\begin{equation}\label{eq:Vhat}   \begin{multlined}[][0.9\textwidth]
		\wh{V}^{n,\ell,k_n}_{t}  
		=  \del\sum_{i=1}^{[ t/\delta_{n} ]-(\ell+2)k_n+1 }  (   \hat{c}^{n}_{(i+k_n)  \delta_{n},k_n\delta_{n}}   -      \hat{c}^{n}_{i \delta_{n}, k_n\delta_{n}}   )\\  \times  (   \hat{c}^{n}_{(i+(\ell+1)k_n)  \delta_{n},k_n\delta_{n}}   -     \hat{c}^{n}_{(i+\ell k_n) \delta_{n}, k_n\delta_{n}}   )\end{multlined}
\end{equation}
for $\ell\geq3$ and $k_n\in\N$, which clearly satisfies $\wh{V}^{n,\ell,k_n}_{t}=(k_n\del)^{2H}\wt{V}^{n,\ell,k_n}_{t}$ but in contrast to $\wt{V}^{n,\ell,k_n}_{t}$ is actually a statistic since it does not depend on the unknown $H$. 
We  construct a first pilot estimator of $H$ by fixing two lags
$\ell_1,\ell_2\geq3$ and then defining
\begin{equation}\label{eq:prelim} 
	\wt H_n = \vp^{-1}\biggl(\frac{\wh V^{n,\ell_1,\wt k_n}_t}{\wh V^{n,\ell_2,\wt k_n}_t}\biggr),
\end{equation}
where $\vp: H\mapsto \Phi^H_{\ell_1}/\Phi^H_{\ell_2}$ is assumed to be a diffeomorphism and 
\begin{equation}\label{eq:wtkn} 
	\wt k_n = [\del^{-1/2}].
\end{equation}
This choice of $\wt k_n$ 
has the advantage that it makes $\wt H_n$ a consistent estimator of $H$, which furthermore satisfies a bias-free central limit theorem regardless of the value of $H\in(0,\frac12)$. On the downside, its rate is poor if $H$ is small. In the following, we therefore propose an iterative approach to improve the rate, which at the same time retains the bias-free property of the resulting estimators. To this end, let
\begin{equation}\label{eq:Hj} \begin{split}
		\H&=\{	H^{(j)}=\tfrac14(\sqrt{4j^2-4j+5}-2j+1): j\in\N\}\\
		&=\{0.3090,0.1514,0.0963,0.0700,\dots\},\end{split}
\end{equation}
which is precisely the set of values of $H$ for which $\frac12-H+\frac1{4H}$ (as it appears in \eqref{eq:M}) is an integer. Therefore, if $H^{(j)}<H\leq H^{(j-1)}$ for $j\in\N$ (where $H^{(0)}=\frac12$), then $M=j$. Using the pilot estimator $\wt H_n$, we now define 
\begin{equation}\label{eq:whM} 
	\wh M_n = \Bigl[\tfrac12-\wt H_n + \tfrac1{4\wt H_n} + \del^{1/4}\log\del^{-1}\Bigr]+1
\end{equation}
as an estimator of the number $M$ from \eqref{eq:M}.  Since $\wt H_n$ is a consistent estimator of $H$ and $\del^{1/4}\log\del^{-1}\to0$, if $H\notin \calh$, we have 
\begin{equation}\label{eq:M-conv} 
	\lim_{n\to\infty}\P(\wh M_n = M)=1.
\end{equation}
If $H\in\calh$, then we still have $\tfrac12-\wt H_n + \frac1{4\wt H_n} \to \tfrac12-  H + \frac1{4  H}$ in probability, but since the limit is an integer, after rounding, $[\tfrac12-\wt H_n + \frac1{4\wt H_n}]$ will typically jump between two consecutive integers as $n$ increases. To avoid that, we have included  $\del^{1/4}\log \del^{-1}$ in the definition of $\wh M_n$, which is asymptotically bigger than the $\del^{1/4}$-fluctuations of  $\tfrac12-\wt H_n + \frac1{4\wt H_n}$ and therefore guarantees that we  have $\P(\wh M_n = M)\to1$ for $H\in\calh$ as well.

Having defined $\wh M_n$, we now set $\bar H^{(0)}_n = \wt H_n$ and define consecutively
\begin{equation}\label{eq:barH} 
	\bar H^{(j)}_n = \varphi^{-1}\biggl(\frac{\sum_{m=1}^{j} w(j)_m m^{1/2-\ov H^{(j-1)}_n} \wh V^{n,\ell_1,m\bar k^{(j)}_n}_t}{\sum_{m=1}^{j} w(j)_m m^{1/2-\ov H^{(j-1)}_n} \wh V^{n,\ell_2,m\bar k_n^{(j)}}_t}\biggr),\quad \ov k^{(j)}_n = [\del^{-2H^{(j)}/(2H^{(j)}+1)}],
\end{equation}
for $j=1,\dots, \wh M_n-1$ and let
\begin{equation}\label{eq:ovH} 
	\ov H_n=\ov H^{(\wh M_n-1)}_n.
\end{equation}
\begin{prop}\label{prop:debias-2}
	Suppose that the conditions of Theorem~\ref{thm:CLT-main} are satisfied with $H\in(0,\frac12)$ and assume \eqref{eq:notzero}. Further fix two  lags $\ell_1,\ell_2\geq3$ such that the function $\vp:H\mapsto \Phi^H_{\ell_1}/\Phi^H_{\ell_2}$, where $\Phi^H_\ell$ is defined in \eqref{eq:Phi}, is a diffeomorphism on $(0,\frac12)$. For any $j\in\N_0$, if $H\leq H^{(j)}$, then
	\begin{equation}\label{eq:nobias-2} 
		\bar H^{(j)}_n - H = O_\P((\bar k^{(j)}_n\del)^{1/2}).
	\end{equation}
\end{prop}  
\begin{proof} 
	We prove the claim by induction, and since the base case $j=0$ corresponds to the CLT of $\wt H_n$, we can consider $j\geq1$ and assume that \eqref{eq:nobias-2} is true for $j-1$.
	We  rewrite
	\begin{equation}\label{eq:help3} 
		\begin{split} \bar H^{(j)}_n &= \varphi^{-1}\biggl(\frac{\sum_{m=1}^{j} w(j)_m m^{1/2-\bar H^{(j-1)}_n+2H} (m\ov k_n^{(j)}\del)^{-2H} \wh V^{n,\ell_1,m  \ov k^{(j)}_n}_t}{\sum_{m=1}^{j} w(j)_m m^{1/2-\ov H^{(j-1)}_n+2H} (m\ov k_n^{(j)}\del)^{-2H} \wh V^{n,\ell_2,m  \ov k^{(j)}_n}_t}\biggr)\\
			&= \varphi^{-1}\biggl(\frac{\sum_{m=1}^{j} w(j)_m m^{1/2-\ov H^{(j-1)}_n+2H}  \wt V^{n,\ell_1,m \ov k^{(j)}_n}_t}{\sum_{m=1}^{j} w(j)_m m^{1/2-\ov H^{(j-1)}_n+2H}  \wt V^{n,\ell_2,m \ov k^{(j)}_n}_t}\biggr)\end{split}
	\end{equation}	
	and recall \eqref{eq:Vt}, \eqref{eq:bias} and that $\varphi$ is a diffeomorphism. Therefore, defining $\psi(x_1,x_2)=\vp^{-1}(x_1/x_2)$, we can use the mean-value theorem to find $(\xi^{n,1}_t,\xi^{n,2}_t)$ satisfying $$\xi^{n,\iota}_t \stackrel{\P}{\longrightarrow} \sum_{m=1}^j w(j)_m m^{1/2+H} \Phi^H_{\ell_\iota}$$ for $\iota=1,2$ such that 
	\begin{equation}\label{eq:dec-2} 
		\begin{split}
			\ov H^{(j)}_n-H	&= \sum_{\iota=1,2} \partial_{x_\iota} \psi(\xi^{n,1}_t,\xi^{n,2}_t)\sum_{m=1}^j w(j)_m m^{1/2-\ov H^{(j-1)}_n+2H} \cala^{n,\ell_\iota,m\ov k_n^{(j)}}_t \\
			& \quad+\begin{multlined}[t][0.75\textwidth]\sum_{\iota=1,2} \partial_{x_\iota} \psi(\xi^{n,1}_t,\xi^{n,2}_t)\sum_{m=1}^j w(j)_m m^{1/2-\ov H^{(j-1)}_n+2H} \\
				\times(\wt V^{n,\ell_\iota,m\ov k_n^{(j)}}_t- \cala^{n,\ell_\iota,m\ov k_n^{(j)}}_t-V^{\ell_\iota}_t).\end{multlined}
		\end{split} \!\!\!
	\end{equation}
	By Theorem~\ref{thm:CLT-main}, $\wt V^{n,\ell_\iota,m\ov k_n^{(j)}}_t- \cala^{n,\ell_\iota,m\ov k_n^{(j)}}_t-V^{\ell_\iota}_t=O_\P((\ov k_n^{(j)} \del)^{1/2})$. 
	It remains to show that the first term on the right-hand side of \eqref{eq:dec-2} is $O_\P((\ov k_n^{(j)} \del)^{1/2})$. Let us fix $\iota$. Since $\partial_{x_\iota} \psi(\xi^{n,1}_t,\xi^{n,2}_t)$ converges in probability, we only have to show that for any $\ell\geq3$,
	\begin{equation}\label{eq:toshow} 
		\sum_{m=1}^j w(j)_m m^{1/2-\ov H^{(j-1)}_n+2H} \cala^{n,\ell,m\ov k_n^{(j)}}_t =O_\P((\ov k_n^{(j)} \del)^{1/2}).
	\end{equation}
	By Lemma~\ref{lem:Q3},    $\cala^{n,\ell,m \ov k_n^{(j)}}_t =O_\P((\ov k_n^{(j)})^{-1/2-H})=o_\P(\del^{H^{(j)}/(1+2H^{(j)})})$. At the same time, for any $m=1,\dots, j$, we have that $m^{1/2-\ov H^{(j-1)}_n+2H}-m^{1/2+H} = O_\P((\ov k_n^{(j-1)}\del)^{1/2})=O_\P(\del^{1/(4H^{(j-1)}+2)})$ by the induction hypothesis. So if we replace $m^{1/2-\ov H^{(j-1)}_n+2H}$ by $m^{1/2+H}$ in \eqref{eq:toshow}, the overall error is
	$o_\P(\del^{H^{(j)}/(1+2H^{(j)})+1/(4H^{(j-1)}+2)})$, which can be shown to be  $o_\P(\del^{1/(4H^{(j)}+2)})$ by using the explicit formula for $H^{(j)}$ from \eqref{eq:Hj}. Now once we have replaced $m^{1/2-\ov H^{(j-1)}_n+2H}$ by $m^{1/2+H}$, \eqref{eq:toshow} follows from Proposition~\ref{prop:debias} (or, more directly, from \eqref{eq:toshow-2}).
\end{proof}

By \eqref{eq:M-conv} and the previous proposition, $\bar H_n$ is our best estimator so far: it is bias-free and satisfies a CLT with rate $\del^{1/(4H^{(j-1)}+2)}$, where $j\in\N$ is such that $H^{(j)}<H\leq H^{(j-1)}$. Unless $H\in\calh$, this rate is close  but still not equal to the optimal one, which is $\del^{1/(4H+2)}$.
As alluded to before, the remaining obstacle to rate efficiency is the fact that the optimal window size $k_n$ should be of order $\del^{-2H/(2H+1)}$, which depends on the parameter $H$ to be estimated.
While    $\bar H_n$ is not rate-optimal in general, it is nevertheless consistent for $H$, so one might be tempted to use $\check k_n=[\del^{-2\bar H_n/(2\bar H_n+1)}]$ as a new window size and to construct a new estimator similarly to \eqref{eq:barH} with $\check k_n$ substituted for $\wt k_n$ and $\wh M_n$ substituted for $j$. While this is a natural approach, there is a pitfall inherent in any such plug-in estimator: the sequence $\check k_n$ is \emph{random} as it depends on the data through $\ov H_n$. As Theorem~\ref{thm:CLT-main} was shown with a deterministic window size, it cannot be applied with $\check k_n$. 

In order to tackle this problem, we use the randomization approach of \cite{Szymanski22} that relies on the following---seemingly paradoxical---idea: Add more randomness to  $\check k_n$ in order to reduce its randomness!
To see what this means and why it works, consider an auxiliary probability space $(\Om',\calf',\P')$ equipped with a uniform random variable $U$. As usual, we form the product space
$$ \wh\Om= \Om\times\Om',\quad \wh\calf= \calf\otimes\calf',\quad \wh \P=  \P\otimes \P'$$
and extend all random variables on $(  \Om, \calf, \P)$ to the new space in the canonical fashion. To simplify the notation, we keep writing $\P$ in the following, but whenever $U$ appears, of course, it stands for $\wh\P$. In addition, we choose two sequences $q_n\sim q/\log \del^{-1}$ for some $q>0$ and $r_n\to\infty$ such that $\del^{-1/4}/r_n\to\infty$ and $\log \del^{-1}/r_n\to0$. We then define the \emph{oracle sequence}
\begin{equation}\label{eq:whkn} 
	\wh k_n=[\del^{-{2\bar H^U_n}/(2\bar H^U_n+1)}],
\end{equation}
where 
\begin{equation}\label{eq:barHUn} 
	\bar H^U_n = \frac{[r_n(\bar H_n+q_n)+U]+1}{r_n}
\end{equation}
is a randomized version of $\bar H_n$. Note that $\bar H^U_n$ depends both on the data (through $\ov H_n$) and on $U$, which is what we mean by ``adding randomness.'' 
The success of the randomization approach pivots on the  following \emph{oracle property}, proved in \cite[Lemma~9]{Szymanski22}:
\begin{equation}\label{eq:magic} 
	\lim_{n\to\infty}	\P(\wh k_n = k_n^U)= 1,
\end{equation}
where
\begin{equation}\label{eq:knU} 
	k_n^U= [\del^{-2H_n^U/(2H_n^U+1)}], \qquad H^U_n=\frac{[r_n( H+q_n)+U]+1}{r_n}.
\end{equation}
Note that $k_n^U$ only depends on $U$ but \emph{no longer on $\calf$}, in particular, no longer on the data. This is what we mean by ``reducing randomness.'' In conclusion, what the randomization approach really does {\color{blue}is} to \emph{exchange} data-dependent randomness for data-independent randomness in the sequence $\wh k_n$. 
And this clearly pays off: conditionally on $U$, the sequence $k_n^U$ is deterministic, to which we can apply all limit theorems obtained so far. Thus, our rate-optimal estimator of $H$ is
\begin{equation}\label{eq:whH} 
	\wh H_n = \varphi^{-1}\biggl(\frac{\sum_{m=1}^{\wh M_n} w(\wh M_n)_m m^{1/2-\ov H_n} \wh V^{n,\ell_1,m\wh k_n}_t}{\sum_{m=1}^{\wh M_n} w(\wh M_n)_m m^{1/2-\ov H_n} \wh V^{n,\ell_2,m\wh k_n}_t}\biggr),
\end{equation}
whose asymptotic behavior is given in the following theorem, our main result.
\begin{theorem}\label{thm:est}
	Grant Assumption~\ref{ass:CLT} and suppose that $q_n\sim q/\log\del^{-1}$ for some $q>0$ and $r_n$ is an increasing sequence such that $\del^{-1/4}/r_n\to\infty$ and $\log \del^{-1}/r_n\to0$.  Moreover, fix two  lags $\ell_1,\ell_2\geq3$ such that the function $\vp:H\mapsto \Phi^H_{\ell_1}/\Phi^H_{\ell_2}$, where $\Phi^H_\ell$ is defined in \eqref{eq:Phi}, is a diffeomorphism on $(0,\frac12)$. Assuming  \eqref{eq:notzero} and using the notations
	\begin{equation}\label{eq:betaH} 
		\beta(H)=e^{2q/(2H+1)^2}
	\end{equation}
	and 
	\begin{equation}\label{eq:wMH} \begin{split}
			w(M,H)&=(w(M,H)_1,\dots,w(M,H)_M)^T,\\
			w(M,H)_m&= \frac{w(M)_mm^{1/2+H}}{\sum_{m=1}^M w(M)_mm^{1/2+H}} \end{split}
	\end{equation}
	and 
	\begin{equation}\label{eq:ga-matrix} 
		\ga_\nu^{\ell,\ell'}(H)=(\ga_\nu^{\ell,\beta(H)m,\ell',\beta(H)m'}(H))_{m,m'=1}^M \in \R^{M\times M},\quad \nu=1,2,3,\quad \ell,\ell'\geq3,
	\end{equation}
	we have for any $H\in(0,\frac12)$ that
		\begin{equation}\label{eq:CLT-est} \begin{split}
			&	\del^{-1/(4H+2)}(\wh H_n - H) \\
			&~\limst N\biggl(0, \biggl(\frac{\vp(H)}{\vp'(H)}\biggr)^2\sum_{\iota,\iota'=1,2} \frac{(-1)^{\iota+\iota'}}{V^{\ell_\iota}_tV^{\ell_{\iota'}}_t} \sum_{\nu=1}^3  [w(M,H)^T\ga^{\ell_\iota,\ell_{\iota'}}_\nu(H)w(M,H)  ] \Ga_\nu(t) \biggr),
		\end{split}\!\!\!
	\end{equation}
	where $M=M(H)$ is the number from \eqref{eq:M} and the limit in \eqref{eq:CLT-est}
	is independent of $\calf'$.
\end{theorem}
\begin{proof} 
	By \eqref{eq:M-conv} and \eqref{eq:magic}, 
	it suffices to prove \eqref{eq:CLT-est} for
	\begin{equation}\label{eq:Hprime} 
		H'_n = \varphi^{-1}\biggl(\frac{\sum_{m=1}^{M} w(M)_m m^{1/2-\ov H_n} \wh V^{n,\ell_1,m  k^U_n}_t}{\sum_{m=1}^{M} w(M)_m m^{1/2-\ov H_n} \wh V^{n,\ell_2,m  k^U_n}_t}\biggr)
	\end{equation}
	instead of $\wh H_n$. And by the definition of stable convergence in law, it suffices to do so conditionally on $U$ as $U$ does not appear in the limit. 
Similarly to \eqref{eq:help3} and \eqref{eq:dec-2}, we have
	\begin{equation*} H'_n 
		= \varphi^{-1}\biggl(\frac{\sum_{m=1}^{M} w(M)_m m^{1/2-\ov H_n+2H}  \wt V^{n,\ell_1,m  k^U_n}_t}{\sum_{m=1}^{M} w(M)_m m^{1/2-\ov H_n+2H}  \wt V^{n,\ell_2,m  k^U_n}_t}\biggr)\end{equation*}
	and  we can find $(\zeta^{n,1}_t,\zeta^{n,2}_t)$ 
	such that 
	\begin{equation}\label{eq:dec} 
		\begin{split}
			&\del^{-1/(4H+2)}(H'_n-H)\\
			&\qquad	=\del^{-1/(4H+2)}\sum_{\iota=1,2} \partial_{x_\iota} \psi(\zeta^{n,1}_t,\zeta^{n,2}_t)\sum_{m=1}^M w(M)_m m^{1/2-\ov H_n+2H} \cala^{n,\ell_\iota,mk_n^U}_t \\
			&\quad\qquad+\begin{multlined}[t][0.75\textwidth]\sum_{\iota=1,2} \partial_{x_\iota} \psi(\zeta^{n,1}_t,\zeta^{n,2}_t)\sum_{m=1}^M w(M)_m m^{1/2-\ov H_n+2H}\\
				\times\del^{-1/(4H+2)}(\wt V^{n,\ell_\iota,mk_n^U}_t- \cala^{n,\ell_\iota,mk_n^U}_t-V^{\ell_\iota}_t).\end{multlined}
		\end{split}
	\end{equation}
	Conditionally on $U$, the sequence $k_n^U$ is  deterministic. Furthermore,
	since $q_n\sim q\log \del^{-1}$ and $\log \del^{-1}/r_n\to0$, we have   $k_n^U/\del^{-2H/(2H+1)} \to e^{2q/(2H+1)^{2}}$.
	By Theorem~\ref{thm:CLT-main}, we know that $(\del^{-1/(4H+2)}(\wt V^{n,\ell_\iota,mk_n^U}_t- \cala^{n,\ell_\iota,mk_n^U}_t-V^{\ell_\iota}_t))_{\iota=1,2,m=1,\dots,M}$ satisfies a joint CLT, so a tedious but  straightforward   computation shows that the second term on the right-hand side of \eqref{eq:dec} converges stably  to the right-hand side of \eqref{eq:CLT-est}. Analogously to how we proved \eqref{eq:toshow}, we can first use Proposition~\ref{prop:debias-2} to replace $m^{1/2-\ov H_n +2H}$ by $m^{1/2+H}$ and then apply Proposition~\ref{prop:debias} to show that  the first term on the right-hand side of \eqref{eq:dec} is $o_\P(\del^{1/(4H+2)})$, completing the proof.
\end{proof}

In order to make Theorems \ref{thm:CLT-main} and \ref{thm:est} feasible,
we need consistent estimators of $\Ga_1(t)$, $\Ga_2(t)$ and $\Ga_3(t)$ from \eqref{eq:terms}.
The following estimators are adapted from \cite[Theorem~8.12]{AJ}.
\begin{prop}\label{prop:asym}
	Let $\wh K_n= \wh k_n [\del^{-\la}]$, where $\wh k_n$ is defined in \eqref{eq:whkn} and $\la\in(0,\frac12)$. Moreover, define 
	\begin{equation}\label{eq:Xchat} \begin{split}
			\delta^{\prime n}_i X& = X_{i\wh K_n\del}-X_{(i-1)\wh K_n\del},\\ \delta^{\prime n}_i \wh c &= \frac{1}{\wh k_n\del}(\wh {\color{blue}c}^n_{(1+i\wh K_n)\del,\wh k_n\del}-\wh {\color{blue}c}^n_{(1+(i-1)\wh K_n)\del,\wh k_n\del})\end{split}
	\end{equation}
	and
	\begin{equation}\label{eq:var-est}\begin{split}
			\wh \Ga^n_1(t)	&= \frac{1}{9\del^3} \sum_{i=1}^{[t/(\wh K_n\del)]-1} ( \delta^{\prime n}_i X)^4(\delta^{\prime n}_{i+1} X)^4,\\
			\wh \Ga^n_2(t)	&= \frac{(\wh K_n\del)^{1-4\wh H_n}}{3} \sum_{i=1}^{[t/(\wh K_n\del)]-2} ( \delta^{\prime n}_i \wh c)^4,\\
			\wh \Ga^n_3(t)	&=  \frac{(\wh K_n\del)^{-1-2\wh H_n}}{3} \sum_{i=1}^{[t/(\wh K_n\del)]-2}  (\delta^{\prime n}_{i} \wh c)^2(\delta^{\prime n}_{i+1} X)^4.
	\end{split}\end{equation}
	Then under the assumptions of Theorem~\ref{thm:est}, we have 
	$ \wh \Ga^n_\nu  \stackrel{\P}{\Longrightarrow} \Ga_\nu$ for each $\nu=1,2,3$.
\end{prop}
\begin{proof} 
	Let $\delta^{\prime n}_i \wt c$ and $\wt \Ga^n_\nu(t)$, $\nu=1,2,3$, be defined in the same way as the corresponding quantities in \eqref{eq:Xchat} and \eqref{eq:var-est} except that $\wh k_n$ and $\wh K_n$ are replaced by some deterministic sequences $k_n\sim\theta\del^{-2H/(2H+1)}$ and $K_n\sim \Theta \del^{-2H/(2H+1)-\la}$ with $\theta,\Theta>0$. 
	Similarly to \eqref{eq:magic}, we have $\P(\wh K_n = K_n^U)=1$, where $K_n^U \sim \Theta' \del^{-2H/(2H+1)-\la}$ for some $\Theta'>0$ (and almost all realizations of $U$). Thus, it suffices to show
	$$ \wt \Ga_1^n  \limL \Ga_1,  \quad 	\wt \Ga_2^n\limL \Ga_2,\quad	\wt \Ga^n_3\limL \Ga_3, $$
	assuming Assumption~\ref{ass:CLT'}. 
	The first convergence is a consequence of  \cite[Theorem~8.4.1]{JP}. 
	
	For the remaining two, 
	we make the following observation: by \eqref{eq:J},  we have that 
	\begin{align*}
		\delta^{\prime n}_i \wt c-\delta^{\prime n}_i  c &= J^n_{1,1+iK_n}-J^n_{1,1+(i-1)K_n} +\frac{1}{k_n\del} \int_{iK_n\del}^{(iK_n+k_n)\del} (c_s-c_{iK_n\del}) ds  \\
		&\quad +\frac{1}{k_n\del} \int_{(i-1)K_n\del}^{((i-1)K_n+k_n)\del} (c_s-c_{(i-1)K_n\del}) ds. 
	\end{align*}
	It is not hard to see from the definition that $J^n_i$ is of size $k_n^{-1/2}$, uniformly in $i$. Moreover, the last two terms on the right-hand side of the previous display are of size $(k_n\del)^H$, uniformly in $i$. Therefore, if we define $\Ga^n_2(t)$ and $\Ga^n_3(t)$ in the same way as $\wt \Ga^n_2(t)$ and $\wt \Ga^n_3(t)$ but with $\delta^{\prime n}_i \wt c$ replaced by $\delta^{\prime n}_i  c$, then 
	$$ \E\biggl[\sup_{t\in[0,T]} \bigl\{\lvert \wt \Ga^n_2(t)-\Ga^n_2(t)\rvert + \lvert \wt \Ga^n_3(t) - \Ga^n_3(t)\rvert\bigr\}\biggr] \lec (\wh K_n\del)^{-H}(k_n^{-1/2}+(k_n\del)^H)\to0 $$
	as $n\to\infty$. Consequently, it remains to show
	$\Ga^n_2\limL \Ga_2$ and $\Ga^n_3\limL \Ga_3$. The first convergence
	was shown in \cite[Theorem 3]{BN11}, while the second is easily obtained from standard techniques of high-frequency statistics (involving drift removal, localization of $\si$, $\eta$ and $\wh\eta$, and a LLN in the case where $X$ is a Brownian motion and $c$ is a fractional Brownian motion) and the fact that $\E[(W^H_1)^2(W_2-W_1)^4]=\E[(\wh W^H_1)^2(W_2-W_1)^4]=3$.
\end{proof}

This could have been the end of our construction of a rate-optimal and feasible estimator of $H$ if it was not for a crucial detail that we have overlooked so far. It turns out that \emph{all} estimators considered in this section (including $\wh H_n$) break down if $H=\frac12$, that is, if volatility is not rough but just a semimartingale. This is because $\Phi^{1/2}_\ell=0$ for any $\ell\geq3$, which implies by Theorem~\ref{thm:CLT-main} that  $(\wt V^{n,\ell_1,\wt k_n}_t,\wt V^{n,\ell_2,\wt k_n}_t)$  for $\ell_1,\ell_2\geq3$ converges in law to a bivariate mixed normal distribution. In particular, the ratio $\wt V^{n,\ell_1,\wt k_n}_t/\wt V^{n,\ell_2,\wt k_n}_t$ and thus the estimator $\wt H_n$ converges in distribution (not in probability) to a random variable with a density. In other words, na\"ively applying $\wt H_n$ if $H=\frac12$ can output any value in the interval $(0,\frac12)$  just by chance! 

There are at least two ways of remedying this problem. One possibility is to choose $\ell_2\in\{0,1\}$ in \eqref{eq:prelim}, which ensures that $\Phi^{1/2}_{\ell_2}\neq0$.  But in this case, the latent bias term $\cala^{n,\ell,k_n}_t$ from \eqref{eq:bias} (which will  have a slightly different form) involves the function $v\mapsto \Delta^3_1(\ell_2-1-v)$, which is no longer smooth in $v\in[0, \frac32]$. This has the consequence that the debiasing procedure from Proposition~\ref{prop:debias} has to be modified. We propose a different, quicker, solution. Loosely speaking, we first use the limit theory of $\wt V^{n,\ell_2,\wt k_n}_t$ to test whether $H=\frac12$ and only use $\wh H_n$ if $H=\frac12$ is rejected. More precisely, we define
\begin{equation}\label{eq:finalest} 
	H_n=\wh H_n \bone_{\calr_n}(\wh V^{n,\ell_2,\wt k_n}_t)+\frac12 \bone_{\R\setminus\calr_n}(\wh V^{n,\ell_2,\wt k_n}_t),
\end{equation}
where
\begin{equation}\label{eq:Rn} 
	\calr_n=\biggl\{ x\in\R: \lvert x\rvert> \del^{3/4}\log\del^{-1}\biggl(\sum_{\nu=1}^3 \ga_\nu^{\ell_2,1,\ell_2,1}(\tfrac12)\wh \Ga^n_\nu(t)\biggr)^{1/2} \biggr\}
\end{equation}
and $\wh \Ga^n_\nu(t)$ is defined in \eqref{eq:var-est}.
\begin{theorem}\label{thm:final}
	Under the assumptions of Theorem~\ref{thm:est} and Proposition~\ref{prop:asym}, we have  
	\[\begin{dcases}
		\lim_{n\to\infty} \P(H_n=\wh H_n)=1	& \text{if } H\in(0,\tfrac12),\\
		\lim_{n\to\infty} \P(H_n=\tfrac12)=1	& \text{if } H=\tfrac12.
	\end{dcases}  \]
	In particular, if $H\in(0,\frac12)$, \eqref{eq:CLT-est} continues to hold with $H_n$  instead of $\wh H_n$.
\end{theorem}
\begin{proof}
	If $H=\frac12$, note that by Theorem~\ref{thm:CLT-main}, $\del^{-3/4}\wh V^{n,\ell_2,\wt k_n}_t$ converges stably in law to a centered normal with conditional variance $\sum_{\nu=1}^3 \ga_\nu^{\ell_2,1,\ell_2,1}(\frac12)\Ga_\nu(t)$. Thus, by Proposition~\ref{prop:asym}, $\calv_n=\del^{-3/4}\wh V^{n,\ell_2,\wt k_n}_t/(\sum_{\nu=1}^3 \ga_\nu^{\ell_2,1,\ell_2,1}(\frac12)\wh \Ga_\nu(t))^{1/2}\stackrel{d}{\longrightarrow} N(0,1)$, so $\P(H_n=\frac12)=\P(\lvert\calv_n\rvert \leq \log \del^{-1})\to1$. Similarly, if $H\in(0,\frac12)$, we know from Theorem~\ref{thm:CLT-main} that $$\calv'_n=\del^{1/2-H}\calv_n\frac{\bigl(\sum_{\nu=1}^3 \ga_\nu^{\ell_2,1,\ell_2,1}(\frac12)\wh \Ga_\nu(t)\bigr)^{1/2}}{(\ga_2^{\ell_2,1,\ell_2,1}(H)\wh \Ga_2(t))^{1/2}}\stackrel{d}{\longrightarrow} N(0,1),$$ which shows that 
	\begin{equation*}
		\P(H_n=\wh H_n)=\P\biggl(\lvert \calv'_n\rvert>\del^{1/2-H}\log\del^{-1}\frac{\bigl(\sum_{\nu=1}^3 \ga_\nu^{\ell_2,1,\ell_2,1}(\frac12)\wh \Ga_\nu(t)\bigr)^{1/2}}{(\ga_2^{\ell_2,1,\ell_2,1}(H)\wh \Ga_2(t))^{1/2}}\biggr)\to1.\qedhere
	\end{equation*}
\end{proof}

\section{Simulation study}\label{sec:sim}

We evaluate the performance of the estimator $\wt H_n$ from \eqref{eq:prelim} in a Monte Carlo simulation. While $\wt H_n$ only achieves a rate of convergence of $\delta_n^{-1/4}$, it does not need any debiasing or fine tuning. We postpone the numerical analysis and calibration of our rate-optimal estimator $H_n$ from \eqref{eq:finalest} to future work. For our simulation study, we consider the rough Heston model
\begin{equation}\label{eq:Heston}\begin{split}
x_t&= x_0+\int_0^t\sqrt{c_s} dW_s, \\
c_t&=c_0+\frac{1}{\Ga(H+1/2)}\int_0^t (t-s)^{H-1/2}(\theta-\la c_s)ds \\
&\quad+ \frac{\nu}{\Ga(H+1/2)}\int_0^t (t-s)^{H-1/2}\sqrt{c_s}(\rho dW_s+\sqrt{1-\rho^2}dW'_s),
\end{split}\end{equation}
where $W$ and $W'$ are two independent Brownian motions,
\begin{equation}\label{eq:param}
x_0=1, \quad V_0=\theta=0.02/252, \quad \la=0.3,\quad \nu=0.3/\sqrt{252}, \quad\rho=-0.7
\end{equation}
are the same parameters as in Section~4.2 of \cite{AE19} and 
\begin{equation}\label{eq:H}
H\in\{0.1,0.2,0.3,0.4,0.49\}.
\end{equation}
For each value of $H$, 
we simulate 5{,}000 samples of $\{x_{i/n}: i=1,\dots, n\}$ for $n=23{,}400$ (which corresponds to sampling every second within one trading day) through the following procedure: we first use the multifactor approximation of \cite{AE19} (with 500 factors) to simulate $\{c_{j/m}:j=1,\dots,m\}$ at $m=10n = 234{,}000$ points within $[0,1]$. In order to obtain the correct roughness, we then use the hybrid scheme of \cite{BLP17} to simulate $\{c_{j/m}:j=1,\dots,m\}$ again, using the previously simulated values of $c$ as input on the right-hand side of the $c$-equation in \eqref{eq:Heston}. From those values of $c$, we obtain prices from a Riemann sum approximation and finally $\{x_{i/n}: i=1,\dots, n\}$ by subsampling at every 10th point of this price series.

Table~\ref{tab:1} reports the performance of a version of the estimator $\wt H_n$ from \eqref{eq:prelim} in which $\ell_1=1$ and $\ell_2=0$ and the $V$-statistics are replaced by the bias-corrected ones from \eqref{eq:01}.  Each estimate is obtained from using a rolling window of 20 simulated paths, which in practice roughly corresponds to taking monthly averages. We can see that the estimator $H$ is essentially unbiased for all considered values of $H$, except when $H=0.1$, where the estimator shows a small upward bias. This is in accordance with the fact that the asymptotic bias term $\cala^{n,\ell,k_n}_t$ from \eqref{eq:bias} is $O_\P(k_n^{-1/2-H})=O_\P(\delta_n^{1/4+H/2})$ by Lemma~\ref{lem:Q3} and therefore only marginally smaller than the $\delta_n^{1/4}$-fluctuations of the estimator $\wt H_n$, giving rise to a small finite-sample bias. Also, for all values of $H$, the variance of $\wt H_n$ is acceptable but nontrivial, suggesting that further fine-tuning (e.g., by considering more than one lag and by using an adaptive window size to achieve optimal rate of convergence) can improve the performance of $\wt H_n$.

\begin{table*}
\caption{Realizations of $\wt H_n$ based on a rolling window of 20 out of 5{,}000 simulated paths.}
\label{tab:1}
\begin{tabular}{@{}llllll@{}}
 & \multicolumn{5}{c}{Quantile}\\
\hline 
$H$ & \multicolumn{1}{c}{2.5\%}&\multicolumn{1}{c}{25\%}
& \multicolumn{1}{c}{50\%} & \multicolumn{1}{c}{75\%}&\multicolumn{1}{c}{97.5\%}
  \\
\hline
0.1  & 0.0452 &0.0963& 0.1270& 0.1568&0.2120  \\
0.2    & 0.1216 &0.1734 &0.1976 &0.2263& 0.2865  \\
0.3  & 0.1756 &0.2668 &0.2981 &0.3289 &0.4017 \\
0.4   & 0.2523 &0.3572 &0.4033 &0.4579& 0.6115\\
0.49    & 0.1469 &0.3964 &0.4910 &0.5889 &0.8148 \\
\hline
\end{tabular}
\end{table*}

\section{Proof of Theorem~\ref{thm:CLT-main}}\label{sec:3}

The proof of Theorem~\ref{thm:CLT-main} essentially consists of two parts: an approximation step (see Section~\ref{sec:CLT1}), where we isolate terms that contribute to the limit $\calz$ in \eqref{eq:CLT}, and a CLT step (see Section~\ref{sec:CLT2}), where we actually prove their stable convergence in law to $\calz$. 

Let us start with a remark about drifts: by the stochastic  and ordinary Fubini theorem,
\begin{align*}
	\int_0^t g_0(t-s)\eta_s dW_s &= \int_0^t\int_s^t g'_0(r-s)dr\eta_sdW_s= \int_0^t \int_0^r g'_0(r-s)\eta_s dW_s dr, \\
	\int_0^t \wt g_0(t-s)\wt\eta_s ds&= \int_0^t \int_s^t \wt g'_0(r-s)dr \wt\eta_s ds = \int_0^t\int_0^r\wt g'_0(r-s)\wt\eta_s ds dr,
\end{align*}
and similarly for other integrals,
so we can rewrite \eqref{eq:si} and \eqref{eq:eta} as 
\begin{equation}\label{eq:ci}\begin{split}
		c_t&=c_0+A_t + \int_0^t g_H(t-s) \bta_sd\bW_s,\qquad	\eta^2_t=\eta^2_0+A^\eta_t + \int_0^t g_{H_\eta}(t-s)\btheta_sd\bar \bW_s,\\
		\wh\eta_t^2&=\wh\eta_0^2+A^{\wh\eta}_t + \int_0^t g_{H_{\wh\eta}}(t-s)\bvartheta_sd\bar \bW_s,
	\end{split}
\end{equation}
where  $\bta_t=(\eta_t,\wh \eta_t)$, $\bW_t= (W_t,\wh W_t)^T$ and
\begin{equation}\label{eq:Ai}\begin{split}
		A_t&=\int_0^t \biggl(a_s+\int_0^s g_0'(s-r)\bta_rd\bW_r+\int_0^s\wt g'_0(s-r)\wt\eta_r dr\biggr) ds + \int_0^t g_{\wt H}(t-s)\wt \eta_s ds,\\
		A^\eta_t&=\int_0^t \biggl(a^\eta_s+\int_0^s (g^\eta_0)'(s-r)\btheta_rd\bar \bW_r+\int_0^s(\wt g^\eta_0)'(s-r)\wt\theta_r dr\biggr) ds + \int_0^t g_{\wt H_\eta}(t-s)\wt \theta_s ds,\\
		A^{\wh\eta}_t&=\int_0^t \biggl(a^{\wh\eta}_s+\int_0^s (g^{\wh\eta}_0)'(s-r)\bvartheta_rd\bar \bW_r+\int_0^s(\wt g^{\wh\eta}_0)'(s-r)\wt\vartheta_r dr\biggr) ds + \int_0^t g_{\wt H_{\wh\eta}}(t-s)\wt \vartheta_s ds.
	\end{split}
\end{equation}
In the last display, the processes in parentheses are all locally bounded, and so are $\wt \eta$, $\wt \theta$ and $\wt \vartheta$. Therefore, there is no loss of generality to assume
\begin{equation}\label{eq:assump} 
	g_0\equiv \wt g_0\equiv g_{0}^\eta\equiv \wt g^\eta_{0}\equiv g_{0}^{\wh \eta} \equiv \wt g_{0}^{\wh \eta}\equiv0,\qquad  \begin{cases} \wt H,\wt H_\eta, \wt H_{\wh \eta}\in(0,\frac12) &\text{if } H\in(0,\frac12),\\ \wt \eta\equiv\wt\theta\equiv\wt\vartheta\equiv0&\text{if } H=\frac12.\end{cases}
\end{equation}
In addition, as it is usual when infill asymptotics are considered, Assumption~\ref{ass:CLT} can be localized (cf.\ \cite[Lemma~4.4.9]{JP}). Therefore, there is no loss of generality if we assume the following strengthened hypotheses.
\settheoremtag{CLT'}
\begin{Assumption}\label{ass:CLT'}
	In addition to Assumption~\ref{ass:CLT}, we have \eqref{eq:assump} and there is a deterministic constant $K\in(0,\infty)$ such that 
	\begin{equation}\label{eq:asbound} 
		\sup_{t\in  [0,\infty)} \bigl\{\lvert a_t\rvert+ \lvert a^\eta_t\rvert+\lvert a^{\wh\eta}_t\rvert+  \lvert b_t\rvert + \lvert \wt\eta_t\rvert + \lvert \wt \theta_t\rvert+ \lvert \wt\vartheta_t\rvert +\lvert \btheta_t\rvert + \lvert \bvartheta_t\rvert \bigr\}<K\quad \text{a.s.}
	\end{equation}  
	In particular, all processes appearing in \eqref{eqn.s}, \eqref{eq:si} and \eqref{eq:eta} have   uniformly bounded moments of all orders.  In addition, for all $p>0$, there is a constant $K_p\in(0,\infty)$ such that 
	\begin{equation}\label{eq:b-2} 
		\lim_{h\to0}	\sup_{s,t\in[0,\infty):\lvert s-t\rvert\leq h} \bigl\{  \E[  \lvert a_t-a_s\rvert^p] +\E[  \lvert b_t-b_s\rvert^p] \bigr\}= 0 
	\end{equation}
	and
	\begin{equation}\label{eq:Holder-2} 
		\sup_{s,t\in[0,\infty)}\bigl\{\E[ \lvert \wt\eta_t-\wt\eta_s\rvert^p]^{1/p}+\E[ \lvert \wt\theta_t-\wt\theta_s\rvert^p]^{1/p}+\E[ \lvert \wt\vartheta_t-\wt\vartheta_s\rvert^p]^{1/p}\bigr\}\leq K_p\lvert t-s\rvert^H.
	\end{equation}
\end{Assumption}

\subsection{Main decomposition and approximations}\label{sec:CLT1}

Since the arguments can be applied component by component, there is no loss of generality to assume $d=1$ in this subsection. For brevity, we also write $\ell=\ell_1$, $k_n=k_n^{(1)}$,  $\theta=\theta^{(1)}$ and $\wt V^{n,\ell}_t=\wt V^{n,\ell,k_n}_t$. 
In a first step, write
\begin{equation}\label{eq:J} 
	\hat{c}^{n}_{i  \delta_{n}, k_n\delta_{n}} 
	=  J_{1,i}^{n}+J_{2,i}^{n} 
	, 
\end{equation}
where
\begin{align*}
	J_{1, i}^n  = & \frac{1}{k_{n}\delta_{n}} 
	\sum_{j=0}^{ k_{n}-1 } 
	\biggl(
	(\delta^{n}_{i+j} x )^{2}  
	- 
	\int_{(i+j-1)\delta_{n}}^{(i+j)\delta_{n}} c_{s}ds  \biggr),
	\\
	J_{2, i}^n  = & \frac{1}{k_{n} \delta_{n} } 
	\sum_{j=0}^{ k_{n} -1} 
	\int_{(i+j-1)\delta_{n}}^{(i+j)\delta_{n}} c_{s}ds = \frac{1}{k_{n}\delta_{n}} \int_{(i-1)\delta_{n}}^{(i-1+k_{n})\delta_{n}} c_{s}ds 
	=\frac{C_{(i-1+k_{n})\delta_{n}} - C_{(i-1)\delta_{n}}}{k_{n}\delta_{n}} , 
\end{align*}
and $C_{t} = \int_{0}^{t}c_{s}ds=\int_0^t \si_s^2ds$ is the integrated volatility. The decomposition \eqref{eq:J} shows that the spot variance estimator $\wh c^n_{i\del,k_n\del}$ is first and foremost an estimator of $J^n_{2,i}$, a local average of spot variance (with $J^n_{1,i}$ being the estimation error). 
With this decomposition, we    have
\begin{equation}\label{eq:Z} \begin{split}
		&\bigl(    \hat{c}^{n}_{(i+k_{n})  \delta_{n}, {k_n\delta_n}}   -   \hat{c}^{n}_{i \delta_{n}, {k_n\delta_n}} \bigr ) \bigl(  \hat{c}^{n}_{(i+(\ell+1)k_{n})  \delta_{n}, {k_n\delta_n}}   -      \hat{c}^{n}_{(i+\ell k_n) \delta_{n}, {k_n\delta_n}} \bigr ) \\
		&\qquad=	Z^{n,\ell} _{1, i} + Z^{n,\ell} _{2, i}+ Z^{n,\ell}_{3, i}+Z^{\prime n, \ell}_{3,i},\end{split}
\end{equation}
where
\begin{align*}
	Z^{n,\ell}_{1,i}&=(J^{n} _{1, i+k_{n}}- J^{n} _{1, i})(J^{n} _{1, i+(\ell+1)k_{n}}-J^{n} _{1, i+\ell k_n}),
	\\ Z^{n,\ell}_{2,i}
	&=  (J^{n} _{2, i+k_{n}}- J^{n} _{2, i})(J^{n} _{2, i+(\ell+1)k_{n}}-J^{n} _{2, i+\ell k_n}), \\
	Z^{n,\ell}_{3, i}&= (J^{n} _{1, i+k_{n}}- J^{n} _{1, i}) (J^{n} _{2, i+(\ell+1)k_{n}}-J^{n} _{2, i+\ell k_n}),
	\\ Z^{\prime n, \ell}_{3,i}&=(J^{n} _{2, i+k_{n}}- J^{n} _{2, i})(J^{n} _{1, i+(\ell+1)k_{n}}-J^{n} _{1, i+\ell k_n}).
\end{align*}
Correspondingly, we obtain the decomposition
\begin{equation}\label{eq:V-dec} 
	\wt V^{n,\ell}_t=Z_1^{n,\ell}(t)+Z_2^{n,\ell}(t)+Z_3^{n,\ell}(t)+Z_3^{\prime n,\ell}(t),
\end{equation}
where 
\begin{equation}\label{eq:Zt} \begin{split}
		Z^{n,\ell}_{1\mid 2}(t)&= \frac{(k_n\del)^{1-2H}}{k_n} \sum_{i=1}^{[t/\del]-(\ell+2)k_n+1} (J^n_{1\mid 2,i+k_n}-J^n_{1\mid 2,i})(J^n_{1\mid 2,i+(\ell+1)k_n}-J^n_{1\mid 2,i+\ell k_n}),\\
		Z^{n,\ell}_3(t)&= \frac{(k_n\del)^{1-2H}}{k_n} \sum_{i=1}^{[t/\del]-(\ell+2)k_n+1} (J^n_{1,i+k_n}-J^n_{1,i})(J^n_{2,i+(\ell+1)k_n}-J^n_{2,i+\ell k_n}),\\
		Z^{\prime n,\ell}_3(t)&= \frac{(k_n\del)^{1-2H}}{k_n} \sum_{i=1}^{[t/\del]-(\ell+2)k_n+1} (J^n_{2,i+k_n}-J^n_{2,i})(J^n_{1,i+(\ell+1)k_n}-J^n_{1,i+\ell k_n}),\end{split}  
\end{equation}
and $1\mid 2$ means that we can take either $1$ or $2$ (consistently for the whole line).

We can interpret $Z^{n,\ell}_1$ and $Z^{n,\ell}_2$ as  (a suitably normalized version of) the realized autocovariance of increments of the error term $J^n_{1,i}$ and the signal term $J^n_{2,i}$, respectively, while $Z^{n,\ell}_3$ and $Z^{\prime n,\ell}_3$ are cross-covariances of signal and noise.
With   decomposition \eqref{eq:Zt} in hand, we can now give an informal argument why $\kappa=\frac{2H}{2H+1}$ is the optimal window size in our estimation procedure. To this end, we consider $Z^{n,\ell}_1$ and $Z^{n,\ell}_2$ in more detail. 
By the integration by parts formula for semimartingales,
\begin{equation}\label{eq:Jn1} \begin{split}
		J^n_{1,i}&=\frac{2}{k_n\del} \sum_{j=0}^{k_n-1} \int_{(i+j-1)\del}^{(i+j)\del} (x_s-x_{(i+j-1)\del}) dx_s\\ &= \frac{2}{k_n\del}  \int_{(i-1)\del}^{(i+k_n-1)\del} (x_s-x_{[s/\del]\del}) dx_s, \end{split}
\end{equation} 
which shows that $J^n_{1,i}$ is a term of order $k_n^{-1/2}$, uniformly in $i$. Moreover, if we neglect the drift in $x$, then $J^n_{1,i}$ is a martingale increment with step size $k_n\del$, and so is  $J^n_{1,i+k_n}-J^n_{1,i}$, but with step size $2k_n\del$. Because $\ell\geq2$, it follows that 
the sum over $i$ in $Z^{n,\ell}_1$ can be written as $O(k_n)$ many martingale sums of $O((k_n\delta_n)^{-1})$ many terms, where each of these terms is of size $O_\P(k_n^{-1})$. Thus, 
$Z^{n,\ell}_1$ is of order $O_\P((k_n\del)^{1/2-2H}k_n^{-1})$,  
which is small if $k_n$ is large. 

At the same time,  $(k_n\del)^{-H}(J^n_{2,i+k_n}-J^n_{2,i})$ is a normalized second-order increment of $C$ over an interval of length $k_n\del$. Therefore, $Z^{n,\ell}_2$
is nothing else but the normalized second-order quadratic variation of $C$ (computed with a lag $\ell$). By definition, $C$ is the integral of a fractional process. It is well known from \cite{BN11,BN13} that the normalized higher-order quadratic variation of a fractional process, computed over a step size of $k_n\del$, converges to a limit at rate $(k_n\del)^{-1/2}$, for all  $H\in(0,\frac12]$. Of course, $C$ is not a fractional process but rather its integral. Our analysis of $Z^{n,\ell}_2$ below shows that the rate of convergence remains unchanged. In other words, the fluctuations of the signal term $Z^{n,\ell}_2$ are small if $k_n$ is small. 

Therefore, nonchalantly ignoring the mixed terms $Z^{n,\ell}_3$ and $Z^{\prime n,\ell}_3$, we obtain the optimal convergence rate if $\kappa$ is chosen such that    $\max\{  (k_n\del)^{1/2-2H}k_n^{-1},(k_n\del)^{1/2} \} $ is minimized. This precisely gives $\kappa=\frac{2H}{2H+1}$ and an optimal rate of convergence of $\del^{-1/(4H+2)}$.


The following three propositions determine the main parts of $Z^{n,\ell}_1(t)$, $Z^{n,\ell}_2(t)$, $Z^{n,\ell}_3(t)$ and $Z^{\prime n,\ell}_3(t)$ that contribute to the CLT. 
\begin{prop}\label{prop:approx1}
	Let  the assumptions be as in  Theorem~\ref{thm:CLT-main}. 
	Then for all $\kappa\in[\frac{2H}{2H+1},\frac12]$ and integer sequences $k_n\sim\theta \del^{-\kappa}$ with $\theta>0$, the following convergence holds: 
	\begin{eqnarray}\label{eqn.zm}
		(k_n\del)^{-1/2}(Z^{n,\ell}_1(t) -M^{n,\ell}_1(t) ) \limL 0,
	\end{eqnarray}
	where using the notations  $y_t=\int_0^t \si_s dW_s$ and  $\chi(t)=-1$ for $t\in[0,\frac12]$ and $\chi(t)=1$ for $t\in(\frac12,1]$, we define
	\begin{equation}\label{eq:M1} \begin{split}
			M^{n,\ell}_1(t)&=\frac{4(k_n\del)^{-1-2H}}{k_n}\sum_{i=1}^{[t/\delta_{n} ]-(\ell+2)k_{n}+1}				\int_{(i-1)\del}^{(i-1+2k_n)\del} \chi(\tfrac{[s/\del]-i+1}{2k_n-1})(y_{s}-y_{[s/\del]\del}) dy_s\\
			&\quad\times \int_{(i+\ell k_n-1)\del}^{(i-1+(\ell+2)k_n)\del} \chi(\tfrac{[s/\del]-i-\ell k_n+1}{2k_n-1})(y_{s}-y_{[s/\del]\del}) dy_s.
		\end{split}\raisetag{-3\baselineskip}
	\end{equation}
	%
	If $\kappa\in(\frac{2H}{2H+1},\frac12]$, we further have 
	\begin{eqnarray}\label{eqn.M1}
		(k_n\del)^{-1/2}M^{n,\ell}_1(t)\limL 0.
	\end{eqnarray}
\end{prop}

\begin{prop}\label{prop:approx2} Under the assumptions of Theorem~\ref{thm:CLT-main}, we have for all $\kappa\in[\frac{2H}{2H+1},\frac12]$ and integer sequences $k_n\sim\theta \del^{-\kappa}$ with $\theta>0$ that
	\[ (k_n\del)^{-1/2}(Z^{n,\ell}_2(t)-V^\ell_t -M^{n,\ell}_2(t)) 	\limL0, \]
	where 
	\begin{align} 
		M^{n,\ell}_2(t) 
		&=	\frac{(k_n\del)^{-1-2H}}{k_n}\sum_{i= 1}^{[t/\del]-(\ell+2)k_n+1} \int_0^{(i-1+(\ell+2)k_n)\delta_n}\int_0^r \Bigl\{ \Delta^{2}_{k_n\del}G_H((i-1)\del-r) \nonumber\\
		&\quad\times \Delta^{2}_{k_n\del}G_H((i+\ell k_n-1)\del-u)  +\Delta^{2}_{k_n\del}G_H((i+\ell k_n-1)\del-r) \label{eq:M2} \\
		&\quad\times  \Delta^{2}_{k_n\del}G_H((i-1)\del-u)\Bigr\} \bta_ud\bW_u \bta_rd\bW_r\nonumber
	\end{align}
	and 
	\begin{equation}\label{eq:GH}
		G_H(t)=\frac{K_H^{-1}}{H+\frac12}t_+^{H+1/2}.
	\end{equation} 
\end{prop}

\begin{prop}\label{prop:approx3} Under the assumptions of Theorem~\ref{thm:CLT-main}, we have for all $\kappa\in[\frac{2H}{2H+1},\frac12]$ and integer sequences $k_n\sim\theta \del^{-\kappa}$ with $\theta>0$ that
	\begin{align*}
		(k_n\del)^{-1/2}(Z^{n,\ell}_3(t)-\cala^{n,\ell}_t-M^{n,\ell}_{31}(t)-M^{n,\ell}_{32}(t))&\limL0,\\
		(k_n\del)^{-1/2}(Z^{\prime n,\ell}_3(t) -M^{\prime n,\ell}_{3}(t))&\limL0,
	\end{align*}  
	where $\cala^{n,\ell}_t=\cala^{n,\ell,k_n}_t$ and 
	\begin{align*}  
			M^{n,\ell}_{31}(t)	&=  \frac{2	(k_n\delta_n)^{-1-2H}}{k_{n}}  \sum_{i=1}^{[ t/\delta_{n} ]-(\ell+2)k_n+1 }  
			\int_{(i-1)\delta_{n}}^{(i-1+2k_{n})\delta_{n}} \chi(\tfrac{[s/\del]-i+1}{2k_n-1})(y_{s}-y_{[s/\del]\del})\\
			&\quad\times \int_0^s \Delta^{2}_{k_n\del} G_H((i-1+\ell k_n)\del-r) \bta_rd\bW_r dy_s,\\
			M^{n,\ell}_{32}(t)	&= 	\frac{2(k_n\delta_n)^{-1-2H} }{k_{n}}  \sum_{i=1}^{[ t/\delta_{n} ]-(\ell+2)k_n+1 }  \int_{0}^{(i-1+(\ell+2)k_{n})\delta_{n}} \! \Delta^{2}_{k_n\del} G_H((i-1+\ell k_n)\del-r)\\
			&\quad\times\int_{(i-1)\delta_{n}}^{(i-1+2k_n)\del\wedge r} \chi(\tfrac{[s/\del]-i+1}{2k_n-1})(y_{s}-y_{[s/\del]\del}) dy_s\bta_rd\bW_r
 \end{align*}
	and 
	\begin{equation}\label{eq:wtZn3-prime} \begin{split}
			M^{\prime n,\ell}_3(t)&=	\frac{2(k_n\delta_n)^{-1-2H} }{k_{n}}  \sum_{i=1}^{[ t/\delta_{n} ]-(\ell+2)k_n+1 }    
			\int_{(i+\ell k_n-1)\delta_{n}}^{(i-1+(\ell+2)k_{n})\delta_{n}} \chi(\tfrac{[s/\del]-i-\ell k_n+1}{2k_n-1})
			\\
			&\quad\times(y_{s}-y_{[s/\del]\del})dy_s \int_0^{(i-1+2k_n)\delta_n} \Delta^{2}_{k_n\del}G_H((i-1)\del-r) \bta_rd\bW_r.\end{split}
	\end{equation}
	If $\kappa\in(\frac{2H}{2H+1},\frac12]$, we also have $(k_n\del)^{-1/2}(M^{n,\ell}_{31}(t)+M^{n,\ell}_{32}(t)+M^{\prime n,\ell}_{3}(t))\limL 0$.
\end{prop}

We now give an overview of the proof of Proposition~\ref{prop:approx2}, with details delegated to Section~\ref{sec:approx2}. Note that $Z^{n,\ell}_2(t)$ is the only term that contributes to the LLN and, furthermore, is the only term that contributes to the CLT for any $\kappa\in [\frac{2H}{2H+1},\frac12]$.  The other three terms $Z^{n,\ell}_1(t)$, $Z^{n,\ell}_3(t)$ and $Z^{\prime n,\ell}_3(t)$ never contribute to the LLN and, unless $\kappa=\frac{2H}{2H+1}$, do not contribute to the CLT, either. Also, the approximations we need to make for them are mostly similar to those for $Z^{n,\ell}_2(t)$. This is why we postpone the whole proof of Propositions~\ref{prop:approx1} and \ref{prop:approx3} to Section~\ref{sec:approx13}.

By \eqref{eq:ci},  we have for any $t\geq0$ that 
\begin{equation*}
	\int_{t}^{t+k_n\delta_n}  c_s ds 	= c_0k_n\delta_n+\int_{t}^{t+k_n\delta_n} A_s ds +\int_0^{t+k_n\delta_n} \Delta_{k_n\del}G_H(t-r)\bta_r d\bW_r.
\end{equation*}
Consequently,
\begin{equation}\label{eq:int-c}
	J^n_{2,i+k_{n}\delta_{n}}-J^n_{2,i}=\frac{1}{k_n\del }	\int_{(i-1)\del}^{(i-1+k_n)\delta_n}  (c_{s+k_n\del}-c_s) ds 	= \frac1{k_n\del }( D^n_{1,i}+D^n_{2,i}),
\end{equation}
where
\begin{equation}\label{eq:5terms}\begin{split}
		D^n_{1,i}	& =\int_0^{(i-1+2k_n)\delta_n} \Delta^{2}_{k_n\del}G_H((i-1)\del-r) \bta_rd\bW_r, \\
		D^n_{2,i}	&=  \int_{(i-1)\del}^{(i-1+k_n)\delta_n}  (A_{s+k_n\del}-A_s) ds.
\end{split}\end{equation}
We can safely remove the drift part $D^n_{2,i}$:
\begin{lemma}\label{lem:Cs} Under Assumption~\ref{ass:CLT'}, we have $(k_n\del)^{-1/2}(Z^{n,\ell}_2(t)-\wt Z^{n,\ell}_2(t))\limL 0$, where
	\begin{equation}\label{eq:wtZn2} 
		\wt Z^{n,\ell}_2(t)=(k_n\del)^{-1-2H}\frac{1}{k_n}\sum_{i=1}^{[t/\del]-(\ell+2)k_n+1} D^n_{1,i} D^n_{1,i+\ell k_n}.
	\end{equation}
\end{lemma}

Next, an application of the integration by parts formula shows that 
\begin{equation}\label{eq:main-dec}
	\wt Z^{n,\ell}_2(t) = M^{n,\ell}_{21}(t)+M^{n,\ell}_{22}(t)+Q^{n,\ell}_2(t) =M^{n,\ell}_2(t)+Q^{n,\ell}_2(t),
\end{equation}
where 
\begin{align*}
	M^{n,\ell}_{21}(t)& =\frac{(k_n\del)^{-1-2H}}{k_n}\sum_{i=1}^{[t/\del]-(\ell+2)k_n+1} \int_0^{(i-1+2k_n)\delta_n} \Delta^{2}_{k_n\del}G_H((i-1)\del-r)\\
	&\quad\times\int_0^r \Delta^{2}_{k_n\del}G_H((i+\ell k_n-1)\del-u)  \bta_ud\bW_u \bta_rd\bW_r,\\
	M^{n,\ell}_{22}(t)& =\frac{(k_n\del)^{-1-2H}}{k_n}\sum_{i=1}^{[t/\del]-(\ell+2)k_n+1} \int_0^{(i-1+(\ell+2)k_n)\delta_n} \Delta^{2}_{k_n\del}G_H((i+\ell k_n-1)\del-r) \\
	&\quad\times\int_0^r \Delta^{2}_{k_n\del}G_H((i-1)\del-u) \bta_ud\bW_u \bta_rd\bW_r,\\
	Q^{n,\ell}_2(t)&=\frac{(k_n\del)^{-1-2H}}{k_n}\sum_{i=1}^{[t/\del]-(\ell+2)k_n+1} \int_0^{(i-1+2k_n)\delta_n} \Delta^{2}_{k_n\del}G_H((i-1)\del-r) \\
	&\quad\times \Delta^{2}_{k_n\del}G_H((i+\ell k_n-1)\del-r)\lvert \bta_r\rvert^2  dr. 
\end{align*}
For the proof of Proposition~\ref{prop:approx2}, we only have to further consider $Q^{n,\ell}_2(t)$.

Interchanging summation and integration and factoring $k_n\delta_n$ out of $\Delta^2_{k_n\del}G_H$, we obtain
\begin{equation}\label{eq:step1}\begin{split}
		Q^{n,\ell}_2(t)&= \frac{1}{k_n} \int_0^{([t/\del]-\ell k_n)\delta_n} \sum_{i=([r/\del]-2k_n+2)\vee1}^{[t/\del]-(\ell+2)k_n+1} \Delta^{2}_{1}G_H(\tfrac{i-1-r/\del}{k_n}) \\
		&\quad\times \Delta^{2}_{1}G_H(\tfrac{i-1+\ell k_n-r/\del}{k_n})\lvert \bta_r\rvert^2 dr.\end{split}
\end{equation}
Writing $r/\del=[r/\del]+\{r/\del\}$ as the sum of its integer and fractional part and changing the index $i-1-[r/\del]$ to $i$ result in
\begin{equation}\label{eq:step2} \begin{split} Q^{n,\ell}_2(t)&= \frac{1}{k_n} \int_0^{([t/\del]-\ell k_n)\delta_n} \sum_{i=(1-2k_n)\vee (-[r/\del])}^{[t/\del]-[r/\del]-(\ell+2)k_n} \Delta^{2}_{1}G_H(\tfrac{i-\{r/\del\}}{k_n})\\
		&\quad\times  \Delta^{2}_{1}G_H(\tfrac{i+\ell k_n-\{r/\del\}}{k_n}) \lvert \bta_r\rvert^2dr.\end{split}
\end{equation}
The next lemma shows that we can replace the lower bound in the summation by $1-2k_n$ and the upper bound by $+\infty$.
\begin{lemma}\label{lem:QV}
	Under Assumption~\ref{ass:CLT'}, we have $(k_n\del)^{-1/2}(Q^{n,\ell}_2(t)-\hat Q^{n,\ell}_2(t))\limL 0$, where
	\begin{equation}\label{eq:Qhat} 
		\hat Q^{n,\ell}_2(t)= \int_0^{([t/\del]-\ell k_n)\delta_n}  \frac{1}{k_n}\sum_{i=1-2k_n}^{\infty} \Delta^{2}_{1}G_H(\tfrac{i-\{r/\del\}}{k_n})  \Delta^{2}_{1}G_H(\tfrac{i+\ell k_n-\{r/\del\}}{k_n})\lvert \bta_r\rvert^2 dr.\!\!\!
	\end{equation}
\end{lemma}

Since $\{r/\del\}\in[0,1)$, the sum over $i$ is a Riemann sum that converges as $k_n\to\infty$ to the limit $\int_{-2}^\infty \Delta^2_1 G_H(v)\Delta^2_1G_H(v+\ell) dv$.  This integral is nothing else but $\Phi^H_\ell$ defined in \eqref{eq:Phi}.
\begin{lemma}\label{lem:Phi}
	For any $H\in(0,\frac12)$ and $\ell\geq2$,
	\begin{equation}\label{eq:Phi-2} 
		\Phi^H_\ell =\int_{-2}^\infty \Delta^2_1 G_H(v)\Delta^2_1 G_H(v+\ell) dv.	\end{equation}
\end{lemma}
As an immediate consequence, we obtain $\wh Q^{n,\ell}_2(t)\limL \Phi^H_\ell \int_0^t \lvert \bta_r\rvert^2 dr= V^\ell_t$, the desired LLN limit. There is only one problem: the convergence rate. Even for  a smooth function (which $\Delta^2_1 G_H(v)$ is not), a Riemann sum converges to its limit only with rate $k_n$, which for small $H$ and small $\kappa$ (including the optimal $\kappa=\frac{2H}{2H+1}$) is much slower than the needed $(k_n\del)^{-1/2}$.  Nevertheless, we shall prove
\begin{lemma}\label{lem:a}
	Under Assumption~\ref{ass:CLT'},  we have $(k_n\del)^{-1/2}(\hat Q^{n,\ell}_2(t)- V^\ell_t)\limL0$.
\end{lemma}
This unexpected gain in convergence rate is only possible because we have a very   special  Riemann sum \emph{and} a very special process $\bta$ in \eqref{eq:Qhat}. To understand what is so particular about the former, let us exploit the periodicity of the mapping $u\mapsto \{u\}$ and change variables a few times to rewrite 
\begin{equation}\label{eq:al}\begin{split}
		\Phi^H_\ell	&=\sum_{i=1-2k_n}^\infty \int_{\frac{i-1}{k_n}}^{\frac i{k_n}} \Delta^2_1 G_H(v)\Delta^2_1 G_H(v+\ell)dv \\
		&=\frac{1}{k_n}\sum_{i=1-2k_n}^\infty \int_0^1 \Delta^2_1 G_H(\tfrac{i-v}{k_n})\Delta^2_1 G_H(\tfrac{i-v}{k_n}+\ell)dv\\
		&=\frac{1}{k_n}\sum_{i=1-2k_n}^\infty \delta_n^{-1}\int_{([r/\del])\del}^{([r/\del]+1)\del}  \Delta^2_1 G_H(\tfrac{i-\{u/\del\}}{k_n})\Delta^2_1 G_H(\tfrac{i-\{u/\del\}}{k_n}+\ell)du,
	\end{split}
\end{equation}
which is valid for any $r>0$ and $n\in\N$. 
Comparing the last line of the previous display with \eqref{eq:Qhat}, we realize that there is no need to study how fast the sum over $i$ approaches its limit since
\begin{align*}
	&\hat Q^{n,\ell}_2(t)-V^\ell_t =\hat Q^{n,\ell}(t)- \Phi^H_\ell \int_0^{([t/\del]-\ell k_n)\del} \lvert \bta_u\rvert^2du+O_\PP(k_n\del)	\\
	&\quad= \frac{1}{k_n}\sum_{i=1-2k_n}^\infty \int_0^{([t/\del]-\ell k_n)\del} \biggl( \Delta^{2}_{1}G_H(\tfrac{i-\{r/\del\}}{k_n})  \Delta^{2}_{1}G_H(\tfrac{i+\ell k_n-\{r/\del\}}{k_n})\\
	&\quad\quad-\delta_n^{-1}\int_{([r/\del])\del}^{([r/\del]+1)\del}  \Delta^2_1 G_H(\tfrac{i-\{u/\del\}}{k_n})\Delta^2_1 G_H(\tfrac{i-\{u/\del\}}{k_n}+\ell)du\biggr)\lvert \bta_r\rvert^2 dr +O_\PP(k_n\del).
\end{align*}
What matters is therefore how fast the difference in parentheses goes to $0$ (as long as we obtain a bound that is an integrable function of $\frac{i}{k_n}$). With this in mind, we rewrite the last line in the previous display as
\begin{multline*}
	\sum_{j=1}^{[t/\del]-\ell k_n} \frac{1}{k_n\del} \sum_{i=1-2k_n}^{\infty} \int_{(j-1)\del}^{j\del}\int_{(j-1)\del}^{j\del}\biggl\{ \Delta^{2}_{1}G_H(\tfrac{i-\{r/\del\}}{k_n})  \Delta^{2}_{1}G_H(\tfrac{i+\ell k_n-\{r/\del\}}{k_n})\\
	-\Delta^{2}_{1}G_H(\tfrac{i-\{u/\del\}}{k_n})  \Delta^{2}_{1}G_H(\tfrac{i+\ell k_n-\{u/\del\}}{k_n})\biggr\}du\lvert \bta_r\rvert^2dr +O_\PP(k_n\del).
\end{multline*}
The $dudr$-double integral on the right-hand side can be split into an $\int_{(j-1)}^{j\del} \int_{(j-1)\del}^r$-part and an $\int_{(j-1)}^{j\del} \int_r^{j\del}$-part. By symmetry, the latter is equal to 
\begin{multline*} -\int_{(j-1)\del}^{j\del}\int_{(j-1)\del}^{r}\biggl\{ \Delta^{2}_{1}G_H(\tfrac{i-\{r/\del\}}{k_n})  \Delta^{2}_{1}G_H(\tfrac{i+\ell k_n-\{r/\del\}}{k_n})\\-\Delta^{2}_{1}G_H(\tfrac{i-\{u/\del\}}{k_n})  \Delta^{2}_{1}G_H(\tfrac{i+\ell k_n-\{u/\del\}}{k_n})\biggr\}\lvert \bta_u\rvert^2du  dr, \end{multline*}
which implies that
\begin{equation}\label{eq:Q-diff}\begin{split}
		&\hat Q^{n,\ell}_2(t)- V^\ell_t \\
		&\qquad= \sum_{j=1}^{[t/\del]-\ell k_n} \frac{1}{k_n\del} \sum_{i=1-2k_n}^{\infty} \int_{(j-1)\del}^{j\del}\int_{(j-1)\del}^{r}\biggl\{ \Delta^{2}_{1}G_H(\tfrac{i-\{r/\del\}}{k_n})  \Delta^{2}_{1}G_H(\tfrac{i+\ell k_n-\{r/\del\}}{k_n})\\
		& \quad\qquad-\Delta^{2}_{1}G_H(\tfrac{i-\{u/\del\}}{k_n})  \Delta^{2}_{1}G_H(\tfrac{i+\ell k_n-\{u/\del\}}{k_n})\biggr\}(\eta_r^2-\eta_u^2 + \wh\eta_r^2-\wh\eta_u^2)du  dr +O_\PP(k_n\del). 
	\end{split}\raisetag{-3.5\baselineskip}\end{equation}

In the last line,  the regularity of $(\eta^2, \wh \eta^2)$ starts to play a role. If $(\eta^2, \wh \eta^2)$ were just any $H$-H\"older regular function, the best bound we can hope for is $\del^H$, which is clearly not enough when $H$ is small.  However, this bound can be significantly improved  if $\eta^2$ and $\wh\eta^2$ have additional structure, as in   \eqref{eq:eta}. This is  therefore the first (and only) place in this paper where the assumption \eqref{eq:eta} is used. Leveraging \eqref{eq:eta} into a sufficiently good bound in \eqref{eq:Q-diff} is still nontrivial, so we complete the proof of Lemma~\ref{lem:a} in Section~\ref{sec:approx2}.

\begin{proof}[Proof of Proposition~\ref{prop:approx2}]
	Proposition~\ref{prop:approx2} follows by combining Lemma~\ref{lem:Cs}, Equation~\eqref{eq:main-dec}, Lemma~\ref{lem:QV} and Lemma~\ref{lem:a}.
\end{proof} 

\subsection{Multivariate stable convergence in law}\label{sec:CLT2}

In a first step, we carry out a few approximations of $M^{n,\ell,k_n}_1(t)$, $M^{n,\ell,k_n}_2(t)$, $M^{n,\ell,k_n}_{31}(t)$, $M^{n,\ell,k_n}_{32}(t)$ and $M^{\prime n,\ell,k_n}_3(t)$. The proof can be found in Section~\ref{sec:details}.
\begin{prop}\label{prop:Mapprox}
	Under the conditions of Theorem~\ref{thm:CLT-main}, we have 
	\begin{equation}\label{eq:approx} 
		\begin{split}
			&\del^{-(1-\kappa)/2} M^{n,\ell,k_n}_1(t) \approx \sum_{j=1}^{[t/\del]}  \zeta^{n,j,\ell,k_n}_1, \qquad
			\del^{-(1-\kappa)/2} M^{n,\ell,k_n}_2(t)	\approx \sum_{j=1}^{[t/\del]}  \zeta^{n,j,\ell,k_n}_2, \\
			&	\del^{-(1-\kappa)/2} (M^{n,\ell,k_n}_{31}(t)+M^{n,\ell,k_n}_{32}(t)+M^{\prime n,\ell, k_n}_3(t))	\approx \sum_{j=1}^{[t/\del]}  \zeta^{n,j,\ell,k_n}_{3},
		\end{split}
	\end{equation}
	where, with the notation $\xi(t)=((1-3\lvert t\rvert)\vee (\lvert t\rvert -1))\bone_{[-1,1](t)}$, 
	\begin{equation}\label{eq:wtzeta1}\begin{multlined}[][0.9\textwidth]
			\zeta^{n,j,\ell,k_n}_1=8 \del^{-(1-\kappa)/2}(k_n\del)^{-1-2H} \int_{(j-1)\del}^{j\del} \si_{(j-1)\del}^4 \int_{([s/\del]-(\ell+2)k_n+1)\del}^{([s/\del]-(\ell-2)k_n)\del}\\
			\xi(\tfrac{[r/\del]-[s/\del]+\ell k_n}{2k_n}) (W_{r}-W_{[r/\del]\del}) dW_r (W_{s}-W_{[s/\del]\del})dW_s
	\end{multlined}\end{equation}
	and
	\begin{equation}\label{eq:wtzeta2}\begin{split}
			\zeta^{n,j,\ell,k_n}_2	
			&=\del^{-(1-\kappa)/2}\int_{(j-1)\del}^{j\delta_n}\int_{r-k_n\del^{1-\eps}}^r\int_{-2}^\infty \Bigl\{\Delta^{2}_{1}G_H(v)\Delta^{2}_{1}G_H(v+\tfrac{r-u}{k_n\del}+\ell)  \\
			&\quad		+ \Delta^2_1 G_H(v)\Delta^2_1G_H(v+\tfrac{r-u}{k_n\del}-\ell)\Bigr\} dv \bta_{(j-1)\del}d\bW_u \bta_{(j-1)\del}d\bW_r
	\end{split}\end{equation}
	and $\zeta^{n,j,\ell,k_n}_{3}=\zeta^{n,j,\ell,k_n}_{31}+ \zeta^{n,j,\ell,k_n}_{32}$ with
	\begin{equation}\label{eq:wtzeta3} \begin{split}
			\zeta^{n,j,\ell,k_n}_{31}&=-2(k_n\del)^{-1/2-H}\del^{-(1-\kappa)/2}\int_{(j-1)\del}^{j\del} \si_{(j-1)\del}^2  \\
			&\quad\times\begin{multlined}[t][0.75\textwidth]  \int_{s-k_n\del^{1-\eps}}^{(j-1)\del} \int_0^1 \Bigl\{ \Delta^3_1 G_H(\tfrac{s-r}{k_n\del} + \ell -u-1) \\
				+\Delta^3_1 G_H(\tfrac{s-r}{k_n\del} - \ell -u-1)\Bigr\}du \bta_{(j-1)\del} d\bW_r  (W_s-W_{(j-1)\del})dW_s,\end{multlined}\\
			\zeta^{n,j,\ell,k_n}_{32}&=\begin{multlined}[t][0.8\textwidth]-2(k_n\del)^{-1/2-H}\del^{-(1-\kappa)/2}\int_{(j-1)\del}^{j\del} \si_{(j-1)\del}^2 \int_{r-(\ell+2)k_n\del}^{(j-1)\del}  \\
				\int_0^1  \Delta^3_1 G_H( \ell-\tfrac{r-s}{k_n\del} -u-1) du (W_s-W_{[s/\del]\del})dW_s\bta_{(j-1)\del} d\bW_r.\end{multlined}
		\end{split}\!\!\!\!
	\end{equation}
\end{prop}

Now let us define 
\begin{equation}\label{eq:triarr} 
	\bzeta^n_t = \begin{pmatrix} \bzeta^{n,\ell_1,k_n^{(1)}}_t \\ \vdots\\ \bzeta^{n,\ell_d,k_n^{(d)}}_t   \end{pmatrix} = \sum_{j=1}^{[t/\del]} \begin{pmatrix}  \zeta^{n,j,\ell_1,k_n^{(1)}}_1 & \zeta^{n,j,\ell_1,k_n^{(1)}}_2 &\zeta^{n,j,\ell_1,k_n^{(1)}}_3  \\ \vdots & \vdots&\vdots\\  \zeta^{n,j,\ell_d,k_n^{(d)}}_1 & \zeta^{n,j,\ell_d,k_n^{(d)}}_2 &\zeta^{n,j,\ell_d,k_n^{(d)}}_3 \end{pmatrix}.
\end{equation}
By \eqref{eq:wtzeta1}, \eqref{eq:wtzeta2} and \eqref{eq:wtzeta3}, we see that the $j$th matrix on the right-hand side of \eqref{eq:triarr} is $\calf_{j\del}$-measurable with a zero $\calf_{(j-1)\del}$-conditional expectation. In conjunction with the fact that 
$$\calz^n \approx \bzeta^n \boldsymbol{1},\qquad \boldsymbol{1}=(1,1,1)^T,$$
which follows from Propositions~\ref{prop:approx1}, \ref{prop:approx2}, \ref{prop:approx3} and \ref{prop:Mapprox}, we can complete the proof of Theorem~\ref{thm:CLT-main} using a stable CLT for martingale arrays (see \cite[Theorem~2.2.15]{JP}) upon showing the following:
\begin{enumerate}
	\item For any $t>0$, $m,m'\in \{1,\dots,d\}$ and $\nu,\nu'\in\{1,2,3\}$ such that $\nu\neq\nu'$, we have
	\begin{align}\label{eq:cov}
		\sum_{j=1}^{[t/\del]} \E[\zeta^{n,j,\ell_m,k_n^{(m)}}_\nu\zeta^{n,j,\ell_{m'},k_n^{(m')}}_\nu \mid \calf_{(j-1)\del}]	&\stackrel{\P}{\longrightarrow} \ga_\nu^{\ell_m,\theta_m,\ell_{m'},\theta_{m'}}(H)\Ga_\nu(t), \\
		\sum_{j=1}^{[t/\del]} \E[\zeta^{n,j,\ell_m,k_n^{(m)}}_\nu\zeta^{n,j,\ell_{m'},k_n^{(m')}}_{\nu'} \mid \calf_{(j-1)\del}]	&\stackrel{\P}{\longrightarrow} 0.\label{eq:cov2}
	\end{align}
	\item For any $m\in \{1,\dots,d\}$, $\nu\in\{1,2,3\}$ and $t>0$, we have
	\begin{equation}\label{eq:tight} 
		\sum_{j=1}^{[t/\del]} \E[(\zeta^{n,j,\ell_m,k_n^{(m)}}_\nu)^4 \mid \calf_{(j-1)\del}]\stackrel{\P}{\longrightarrow} 0.
	\end{equation}
	\item If $N\in\{W,\wh W\}$ or $N$ is  a bounded martingale on $(\Om,\calf,\F,\P)$ that is orthogonal in the martingale sense to both $W$ and $\wh W$, then
	\begin{equation}\label{eq:stable} 
		\sum_{j=1}^{[t/\del]} \E[\zeta^{n,j,\ell_m,k_n^{(m)}}_\nu(N_{j\del}-N_{(j-1)\del}) \mid \calf_{(j-1)\del}]\stackrel{\P}{\longrightarrow} 0.
	\end{equation}
\end{enumerate}

The proof of these three properties will be given in Section~\ref{sec:details}. This completes the proof of Theorem~\ref{thm:CLT-main}.\qed

\begin{appendix}
	
	\section{Details of the proof of Proposition~\ref{prop:approx2}}\label{sec:approx2}
	
	We start with a lemma on the regularity of the process $A$ from \eqref{eq:Ai}.
	\begin{lemma}\label{lem:Ai}
		For any $T>0$, we have that 
		\begin{align*}
			&	\E[(A_{t+h}-A_t)^2]^{1/2}+\E[(A^\eta_{t+h}-A^\eta_t)^2]^{1/2}+\E[(A^{\wh\eta}_{t+h}-A^{\wh\eta}_t)^2]^{1/2}\\
			&\qquad\lec  \lc(1+t^{-H})h^{(2H+1/2)\wedge1}\rc \wedge h^{H},
		\end{align*} 
		with a constant that is uniform for $t\in[0,T]$, $h>0$ and $i=0,\dots,L$.  
		If $H=\frac12$, the previous bound can be improved to 
		\[ \E[(A_{t+h}-A_t)^2]^{1/2}+\E[(A^\eta_{t+h}-A^\eta_t)^2]^{1/2}+\E[(A^{\wh\eta}_{t+h}-A^{\wh\eta}_t)^2]^{1/2}\lec   h.
		\] 
	\end{lemma}
	\begin{proof} The statement is obvious for $H=\frac12$, so we assume $H\in(0,\frac12)$ in the following. We only consider increments of $A_t$; the bounds for $A^\eta$ and $A^{\wh\eta}$ can be derived in the same way.
		Since the first term in the definition of $A_t$ is differentiable almost surely with $L^2$-bounded derivative, we only need to consider the second term. To avoid introducing additional notation, we  assume that $g_0'\equiv \wt g'_0\equiv0$  such that $A_t=(g_{\wt H}\ast \wt\eta)(t)$, where $(f\ast g)(t)=\int_\RR f(t-s)g(s)ds$ denotes the convolution of two integrable functions. Note that we used the convention $\wt\eta^{(i)}_s=0$ for $s<0$. 
		Since 
		\begin{align*}
			\E[(A_{t+h}-A_t)^2]^{1/2}&\leq \int_0^t \E[(\wt\eta_{t+h-s}-\wt\eta_{t-s})^2]^{1/2} g_{\wt H}(s) ds + \int_t^{t+h} \! \E[(\wt\eta_{t+h-s})^2]^{1/2} g_{\wt H}(s) ds \\
			&\lec h^H + [(t+h)^{\wt H +1/2}-t^{\wt H+1/2}]\lec h^H + h^{\wt H +1/2}\lec h^H,
		\end{align*}
		we have shown the second upper bound.
		To get the first one, observe that 
		\begin{align*}
			\E[(\Delta^2_h A_t)^2]^{1/2} &= \E[(\Delta^2_h (g_{\wt H}\ast \wt\eta)(t))^2]^{1/2}=\E[(\Delta_h g_{\wt H}\ast \Delta_h\wt\eta)(t))^2]^{1/2} \\
			&\leq \int_{-h}^{t+h} \lvert  g_{\wt H}(s+h)-g_{\wt H}(s)\rvert \rvert \E[(  \wt \eta_{t+h-s}-  \wt \eta_{t-s})^2]^{1/2} ds.
		\end{align*}
		The last integral splits into three parts, according to whether $s\in(-h,0)$, $s\in(t,t+h)$ or $s\in(0,t)$. Bounding them by
		\begin{align*}
			\int_{-h}^0 g_{\wt H}(s+h)\E[(  \wt \eta_{t+h-s}-  \wt \eta_{t-s})^2]^{1/2}ds &\lec  h^H\int_{-h}^0 g_{\wt H}(s+h) ds\lec h^{H+\wt H +1/2}, \\
			\int_t^{t+h}  \lvert g_{\wt H}(s+h)-g_{\wt H}(s)\rvert   \E[(  \wt \eta_{t+h-s} )^2]^{1/2} ds&\lec \int_t^{t+h}  \lvert  g_{\wt H}(s+h)-g_{\wt H}(s)\rvert ds \\
			&\lec h^{\wt H+1/2}\wedge
			t^{\wt H-3/2} h^2\leq t^{-H} h^{H+\wt H+1/2},\\
			\int_0^t \lvert  g_{\wt H}(s+h)-g_{\wt H}(s)\rvert   \E[(  \wt \eta_{t+h-s}-  \wt \eta_{t-s})^2]^{1/2} ds & \leq h^H \int_0^t \lvert  g_{\wt H}(s+h)-g_{\wt H}(s)\rvert   ds \\
			&\lec h^{H+\wt H+1/2},
		\end{align*}
		we obtain the assertion of the lemma  from \cite[Proposition 2]{RS04}.
	\end{proof}
	
	\begin{proof}[Proof of Lemma~\ref{lem:Cs}] We start with  $H\in(0,\frac12)$.
		Because 
		\begin{equation}\label{eq:Delta-G-2} 
			\Delta^{2}_{1}G_H(u)\lec u^{H-3/2}\wedge 1,
		\end{equation} we have that 
		\begin{equation}\label{eq:C1}\begin{split}
				\E[\lvert D^n_{1,i}\rvert^2]^{1/2}	&\lec (k_n\del)^{1/2+H}\biggl( \int_0^{(i-1+2k_n)\delta_n} (\Delta^{2}_{1}G_H(\tfrac{i-1-r/\del}{k_n}))^2 dr\biggr)^{1/2}   \\
				&\lec (k_n\del)^{1+H} \biggl( \int_{-2}^{\frac{i-1}{k_n}} (\Delta^2_1 G_H(u))^2 du  \biggr)^{1/2} \lec (k_n\del)^{1+H},  \end{split}
		\end{equation} 
		uniformly in $i$. 
		Furthermore, by Lemma~\ref{lem:Ai},
		\begin{equation}\label{eq:C2}  \begin{split}
				\E[\lvert D^n_{2,i}\rvert^2]^{1/2}&\lec  (k_n\del)^{(2H+1/2)\wedge 1} \int_{(i-1)\del}^{(i-1+k_n)\del} s^{-H} d s\\
			\end{split}
		\end{equation}
	so it follows from the mean-value theorem and the Cauchy--Schwarz inequality  that the difference $Z^{n,\ell}_2(t)-\wt Z^{n,\ell}_2(t)$ 
is of order $(k_n\del)^{-1-2H}k_n^{-1}(k_n\del)^{1+H}(k_n\del)^{(2H+1/2)\wedge 1} \times\sum_{i=1}^{[t/\delta_n]} \int_{(i-1)\del}^{(i-1+k_n)\del} s^{-H} d s$. Since $\sum_{i=1}^{[t/\delta_n]} \int_{(i-1)\del}^{(i-1+k_n)\del} s^{-H} d s = O(k_n \int_0^ts^{-H} ds) = O(k_n)$, it follows that    $Z^{n,\ell}_2(t)-\wt Z^{n,\ell}_2(t)=o_\P((k_n\del)^{1/2})$ for all $H\in(0,\frac12)$.
		
		If $H=\frac12$, we note that $$D^n_{1,i}=\int_{(i-1)\del}^{(i-1+2k_n)\del} \Delta^2_{k_n\del} G_{1/2}((i-1)\del-r)\bta_r d\bW_r = O_\P((k_n\del)^{3/2})$$ and $D^n_{2,i}=O_\P((k_n\del)^2)$. Thus, decomposing
		\begin{align*}
			(k_n\del)^{-1/2}(Z^{n,\ell}_2(t)-\wt Z^{n,\ell}_2(t))	&= (k_n\del)^{-2}\frac1{k_n}  \sum_{i=1}^{[t/\del]-(\ell+2)k_n+1} D^n_{1,i} D^n_{2,i+\ell k_n}  \\
			&\quad+(k_n\del)^{-2}\frac1{k_n}  \sum_{i=1}^{[t/\del]-(\ell+2)k_n+1} D^n_{2,i} D^n_{1,i+\ell k_n}\\
			&\quad+(k_n\del)^{-2}\frac1{k_n}  \sum_{i=1}^{[t/\del]-(\ell+2)k_n+1} D^n_{2,i} D^n_{2,i+\ell k_n},
		\end{align*}
		we easily notice that the last term is $O_\P(k_n\del)$ and therefore negligible. Let us consider the first expression on the right-hand side; the second one can be treated similarly. Bounding term by term, we notice that it is of order $O_\P(1)$. This means two things: to show convergence to zero, we need to find a better way of bounding this expression. But at the same time, we are allowed to make any modification that leads to an $o_\P(1)$ error. In particular, thanks to \eqref{eq:b-2}, we may replace $D^n_{2,i+\ell k_n}$ by (recall that we may assume $A_t=\int_0^t a_s ds$)
		$$ \wt D^n_{2,i+\ell k_n}= \int_{(i-1+\ell k_n)\del}^{(i-1+(\ell+1)k_n)\del} \int_s^{s+k_n\del} a_{(i-1)\del} dr ds, $$
		which has the advantage that it is $\calf_{(i-1)\del}$-measurable. Therefore, the product $D^n_{1,i} D^n_{2,i+\ell k_n}$ is $\calf_{(i-1+2k_n)\del}$-measurable with zero $\calf_{(i-1)\del}$-conditional expectation. By a martingale size estimate (see \cite[Appendix A]{CDL22}), it follows that 
		\begin{equation*} \begin{aligned}[b]
				&\E\biggl[\sup_{t\in[0,T]} \biggl\lvert (k_n\del)^{-2}\frac1{k_n}  \sum_{i=1}^{[t/\del]-(\ell+2)k_n+1} D^n_{1,i} \wt D^n_{2,i+\ell k_n} \biggr\rvert\biggr] \\
				&\qquad\lec (k_n\del)^{-2}\frac1{k_n} (k_n/\del)^{1/2} (k_n\del)^{3/2}(k_n\del)^2= k_n\del\to0. \end{aligned}\qedhere
		\end{equation*}
	\end{proof}
	
	\begin{proof}[Proof of Lemma~\ref{lem:QV}]
		We first remove $\vee (-[r/\del])$ from the lower bound of $i$. Since this is only relevant for $r\leq 2k_n\del$ and the two $\Delta^2_1G_H$-terms are uniformly bounded for $i\in\{1-2k_n,\dots, 0\}$, this removal only incurs an error of order $k_n^{-1}(k_n\del)k_n=k_n\del$, which is smaller than the desired convergence rate of $(k_n\del)^{1/2}$.
		
		It remains to replace the upper bound of the sum by $+\infty$. In order to justify this, observe that 
		\begin{equation}\label{eq:Delta-G} 
			\lvert \Delta^2_1 G_H(\tfrac{i-\{r/\del\}}{k_n})\rvert\lec \bigl(\tfrac{i}{k_n}\bigr)^{H-3/2}\wedge 1\qquad (\text{and } \lec \bone_{\{i\leq 2k_n+1\}} \text{ if } H=\tfrac12),
		\end{equation}
		uniformly in $n$, $i$ and $r$. If $H\in(0,\frac12)$, we now choose
		some $p>1+(1-\kappa)/(4\kappa (1-H))$. For any  $\kappa\in[\frac{2H}{2H+1},\frac12]$, if $p$ is sufficiently close to the lower bound, we still have $k_n^p\del\to0$. So if we consider the two cases $t-r\geq k_n^p\del$ and $t-r\leq k_n^p\del$ separately, we observe in the former case that   
		\begin{equation}\label{eq:p1} 
			\begin{split}
				&\E	\biggl[\biggl\lvert\frac{1}{k_n} \int_0^{t-k_n^p\del} \sum_{i=[t/\del]-[r/\del]-(\ell+2)k_n+1}^\infty \Delta^{2}_{1}G_H(\tfrac{i-\{r/\del\}}{k_n})  \Delta^{2}_{1}G_H(\tfrac{i+\ell k_n-\{r/\del\}}{k_n}) \lvert \bta_r\rvert^2dr \biggr\rvert\biggr]\\
				&\quad \lec \frac{1}{k_n}\sum_{i=k_n^p/2}^\infty \biggl(\frac{i}{k_n}\biggr)^{2H-3}\int_0^{t-k_n^p\del} \E[\lvert \bta_r\rvert^2]dr\lec k_n^{(2-2H)(1-p)}=o((k_n\del)^{1/2})
			\end{split}
		\end{equation}
		by our choice of $p$. 
		If $t-r\leq k_n^p\del$, we pick some $\eps>0$ to be specified later and, for the moment, small enough such that we have the bound 
		$$	\lvert \Delta^2_1 G_H(\tfrac{i-\{r/\del\}}{k_n})\Delta^2_1 G_H(\tfrac{i+\ell k_n-\{r/\del\}}{k_n})\rvert \leq \bigl(\tfrac{i}{k_n}\bigr)^{2H-3}\wedge 1 \leq \bigl(\tfrac{i}{k_n}\bigr)^{-1-\eps}\wedge 1.$$
		Then
		\begin{equation}\label{eq:bound}\begin{split}
				&\E	\biggl[\biggl\lvert\frac{1}{k_n} \int_{t-k_n^p\del}^{([t/\del]-\ell k_n)\delta_n} \sum_{i=[t/\del]-[r/\del]-(\ell+2)k_n+1}^\infty \Delta^{2}_{1}G_H(\tfrac{i-\{r/\del\}}{k_n}) \\
				&\quad\times \Delta^{2}_{1}G_H(\tfrac{i+\ell k_n-\{r/\del\}}{k_n})\lvert \bta_r\rvert^2dr \biggr\rvert\biggr] \lec  \frac1{k_n} k_n^{1+\eps}  \int_{t-k_n^p\del}^{([t/\del]-\ell k_n)\delta_n}\E[\lvert \bta_r\rvert^2]dr\lec k_n^{\eps+p}\del.
			\end{split}\!\!\!
		\end{equation}
		The reader can   verify that for any $H\in(0,\frac12)$, if $p$ is close enough to $1+(1-\kappa)/(4\kappa (1-H))$ and $\eps$ is small enough, then $k_n^{\eps+p}\del=o((k_n\del)^{1/2})$. 
		
		If $H=\frac12$, we choose $p\in(1,\frac32)$. By \eqref{eq:Delta-G}, the left-hand side of \eqref{eq:p1} is simply zero because $p>1$. Similarly, the summation in \eqref{eq:bound} only involves $O(k_n)$ many terms, so the left-hand side of \eqref{eq:bound} is $O(k_n^p\del)$, which is $(k_n\del)^{1/2}$ since $p<\frac32$.
	\end{proof}
	
	\begin{proof}[Proof of Lemma~\ref{lem:Phi}] If $H=\frac12$, we have $\Delta^2_1 G_{1/2}(v)= 0$ for $v\notin(-2,0)$. Thus, $\Phi^{1/2}_\ell=0$ for all $\ell\geq2$. For $H\in(0,\frac12)$,
		it is possible to compute $\Phi^H_\ell$ using properties of fractional Brownian motion and integration by parts. But in order to prepare for upcoming (and more involved) calculations, we show how to obtain \eqref{eq:Phi-2} using Fourier methods. An advantage of this approach is that it yields a formula for arbitrary $\ell\in\R$ (not just $\ell\in\{2,3,\dots\}$), without the need to differentiate between multiple cases. First notice that there is no harm to extend the integral in \eqref{eq:Phi-2} to $-\infty$, because $\Delta^2_1G_H(v)=0$ for all $v<-2$. Therefore, 
		by Parseval's formula, 
		\[ \int_{-2}^\infty \Delta^2_1 G_H(v)\Delta^2_1 G_H(v+\ell) dv = \frac{1}{2\pi} \int_\R \calf[\Delta^2_1G_H](\xi)\overline{\calf[\Delta^2_1G_H](\xi)} e^{-i\ell \xi} d\xi, \]
		where $\calf[\vp](\xi)=\int_\R\vp(x)e^{-ix\xi}d\xi$ denotes the Fourier transform of an $L^2$-function (which can be extended to the space of tempered distributions) and $\overline{z}$ denotes the complex conjugate of $z\in\mathbb{C}$. We need a few definitions and facts regarding Fourier transforms, which can be found in \cite[Section 3.2 and Example~7.1.17]{H90}:	for $\al\in\mathbb{C}\setminus\{0,-1,-2,\dots\}$,
		\begin{equation}\label{eq:fml} 
			\begin{aligned}
				x^\al_{\pm}	&=(\pm x)^\al\bone_{\{\pm x >0\}}, \quad \calf[x^\al_{\pm}](\xi)	=\Ga(\al+1)e^{\mp i\pi(\al+1)/2}(\xi\mp i0)^{-\al-1},  \\
				(x\pm i0)^\al &=x^\al_+ + e^{\pm i\pi \al}x^\al_-, \quad\calf[(x\pm i0)^\al](\xi)=2\pi e^{\pm i\pi\al/2}\Ga(-\al)^{-1}\xi^{-\al-1}_\pm.
			\end{aligned}
		\end{equation}
		In particular, still for $\al\in\mathbb{C}\setminus\{0,-1,-2,\dots\}$,
		\begin{equation}\label{eq:abs} \begin{split}
				\calf[\lvert x\rvert^\al](\xi)&=\Ga(\al+1)(e^{- i\pi(\al+1)/2}(\xi- i0)^{-\al-1}+ e^{ i\pi(\al+1)/2}(\xi+ i0)^{-\al-1})\\
				&=2\Ga(\al+1)\cos(\tfrac{\pi(\al+1)}{2})\lvert\xi\rvert^{-\al-1}.
			\end{split}
		\end{equation}
		Moreover, by the fact that $\calf[\vp(\cdot + h)](\xi)=e^{ih\xi}\calf[\vp](\xi)$, the operator $\Delta^2_1$ in the time domain corresponds to multiplication with 
		$e^{2i\xi}-2e^{i\xi}+1$ in the Fourier domain. Therefore, recalling \eqref{eq:GH}, we have the right-hand side of \eqref{eq:Phi-2} equals
		\begin{align*}
			&\frac{K_H^{-2}\Ga(H+\frac32)^2}{2\pi(H+\frac12)^2}\int_\R e^{-i\ell\xi}e^{-\frac12i\pi(H+\frac32)}(\xi-i0)^{-H-3/2}e^{\frac12i\pi(H+\frac32)}(\xi+i0)^{-H-3/2} \\
			&\quad\times(e^{2i\xi}-2e^{i\xi}+1)(e^{-2i\xi}-2e^{-i\xi}+1) d\xi.
		\end{align*}
		Observe that
		\[ (e^{2i\xi}-2e^{i\xi}+1)(e^{-2i\xi}-2e^{-i\xi}+1)= e^{2i\xi}-4e^{i\xi}+6-4e^{-i\xi}+e^{-2i\xi}=(e^{\frac12i x}-e^{-\frac12i x})^4, \]
		which corresponds to $\delta^4_1$ in the time domain. 
		Moreover, by \eqref{eq:fml},
		\begin{equation}\label{eq:fml-prod}\begin{split}
				(\xi-i0)^{-\al} (\xi+i0)^{-\al}&= \xi_+^{-2\al}+e^{i\pi \al}\xi^{-\al}_-\xi^{-\al}_+ +\xi^{-\al}_+e^{-i\pi\al}\xi^{-\al}_- + \xi_-^{-2\al} \\ &=\xi_+^{-2\al}+ \xi_-^{-2\al}=\lvert \xi\rvert^{-2\al}.\end{split}
		\end{equation}
		Therefore, using the last formula in \eqref{eq:fml} and with the convention that $\delta^4_1$ acts on the variable $\ell$, we obtain
		\begin{align*}
			\Phi^H_\ell&= \frac{K_H^{-2}\Ga(H+\frac32)^2}{2\pi(H+\frac12)^2}\int_\R e^{-i\ell\xi} (\xi-i0)^{-H-3/2} (\xi+i0)^{-H-3/2}(e^{\frac12i x}-e^{-\frac12i x})^4d\xi\\
			& =\frac{\Ga(H+\frac12)^2\Ga(-2H-2)}{2\pi K_H^2}\delta^4_1(e^{i\pi (H+1)}(\ell-i0)^{2H+2} + e^{-i\pi(H+1)}(\ell +i0)^{2H+2}) \\
			&  =\frac{\Ga(H+\frac12)^2\Ga(-2H-2)}{2\pi K_H^2}(e^{i\pi (H+1)}+e^{-i\pi (H+1)})\delta^4_1(\ell_+^{2H+2}+\ell_-^{2H+2})\\
			& =\frac{2\cos(\pi (H+1))\Ga(H+\frac12)^2\Ga(-2H-2)}{2\pi K_H^2} \delta^4_1\lvert \ell\rvert^{2H+2}.
		\end{align*}
		Using \eqref{eq:gH} and properties of the Gamma function, one can show that the factor in front of $\delta^4_1\lvert\ell\rvert^{2H+2}$ is equal to $1/(2(2H+1)(2H+2))$, proving \eqref{eq:Phi-2}.
	\end{proof}
	
	\begin{proof}[Proof of Lemma~\ref{lem:a}]
		Recall  \eqref{eq:ci}. In a first step, we show that the contributions of $A^\eta$ and $A^{\wh\eta}$ to \eqref{eq:Q-diff} are negligible at a rate of $(k_n\del)^{1/2}$. We only consider $A^\eta$, as our arguments apply to   $A^{\wh\eta}$  analogously.  If $j\leq (k_n/\del)^{1/2}$, then   $\sum_{j=1}^{[(k_n/\del)^{1/2}]} \frac{1}{k_n\del} \int_{(j-1)\del}^{j\del}\int_{(j-1)\del}^{r}$ is of size $(\del/k_n)^{1/2}$, the sum over $i$ of the terms in $\{ \cdots\}$ can be bounded by a multiple of $k_n^{1+\eps}$, where $\eps>0$ can be as small as we want (cf.\ \eqref{eq:bound}), and $A^\eta_r-A^\eta_u$ is of size $\del^H$ by Lemma~\ref{lem:Ai}. So in total, the contribution of terms with $j\leq (k_n/\del)^{1/2}$ is of size $ (k_n\del)^{1/2} k_n^\eps\del^{H}$, which is $o((k_n\del)^{1/2})$ if $\eps$ is sufficiently small. If $j>(k_n/\del)^{1/2}$, then $u\geq (k_n\del)^{1/2}$ and therefore, by similar arguments, the contribution of the terms with $j>(k_n/\del)^{1/2}$ is of size $k_n^{-1}k_n^{1+\eps} (k_n\del)^{-H/2} \del^{(2H+1/2)\wedge1}$, which is $o((k_n\del)^{1/2})$ if  $\eps$ is small.
		
		It remains to analyze the expressions $\sum_{j=1}^{[t/\del]-\ell k_n} \xi^n_j$ and $\sum_{j=1}^{[t/\del]-\ell k_n} \wh\xi^n_j$, where
		\begin{equation}\label{eq:Q-mart}\begin{split}
				\xi^n_j&=  \frac{1}{k_n\del} \sum_{i=1-2k_n}^{\infty} \int_{(j-1)\del}^{j\del}\int_{(j-1)\del}^{r}\biggl\{ \Delta^{2}_{1}G_H(\tfrac{i-\{r/\del\}}{k_n})  \Delta^{2}_{1}G_H(\tfrac{i+\ell k_n-\{r/\del\}}{k_n})\\
				&\quad -\Delta^{2}_{1}G_H(\tfrac{i-\{u/\del\}}{k_n})  \Delta^{2}_{1}G_H(\tfrac{i+\ell k_n-\{u/\del\}}{k_n})\biggr\}\int_0^r \Delta_{r-u} g_{H_\eta}(u-s) \btheta_s d\bar \bW_s du  dr
			\end{split}\!\end{equation}
		and $\hat \xi^n_j$ is defined in the same way but with $\vartheta$ instead of $\theta$. Clearly, it suffices to consider $\xi^n_j$.
		To this end, if $H_\eta\in(0,\frac12)$, we consider a sequence of numbers $0=\la_0<\la_1<\dots< \la_Q < \la_{Q+1}=\infty$, whose values shall be determined at a later stage, and define $\la_n^{(q)}=[\del^{-\la_q}]$ for all $q=0,\dots, Q+1$. In particular, $1=\la_n^{(0)} \ll \la_n^{(1)}\ll \cdots\ll \la_n^{(Q)} < \la_n^{(Q+1)}=\infty$. Accordingly, we can define $\xi^{n,q}_j$ by the same formula as in \eqref{eq:Q-mart}, except we replace $\int_0^r \cdots d\ov \bW_s$ by $\int_{(j+1-\la_n^{(q+1)})\del}^{(j+1-\la_n^{(q)})\del\wedge r} \cdots d\ov \bW_s$. Then clearly
		\[ \sum_{j=1}^{[t/\del]-\ell k_n} \xi^n_j = \sum_{q=0}^Q\sum_{j=1}^{[t/\del]-\ell k_n} \xi^{n,q}_j.  \]
		Since $u,r\in((j-1)\del,j\del)$, we have by the mean-value theorem (for $q=1,\dots,Q$) and a change of variables (for $q=0$) that
		\begin{equation}\label{eq:xi-q}\begin{split}
				\int_{(j+1-\la_n^{(q+1)})\del}^{(j+1-\la_n^{(q)})\del\wedge r} (\Delta_{r-u} g_{H_\eta}(u-s))^2 ds&\lec \begin{dcases}
					\del^2	\int_{(j+1-\la_n^{(q+1)})\del}^{(j+1-\la_n^{(q)})\del} (u-s)^{2{H_\eta}-3} ds	&\text{if } q\geq1, \\
					\del^{2{H_\eta}} \int_{j+1-\la_n^{(1)}}^{r/\del} (\Delta_{\frac{r-u}{k_n\delta_{n}}}g_{H_\eta}(\tfrac{u}{\del}-s))^{2} ds	& \text{if } q=0.
				\end{dcases}\\
				& \lec \del^{2{H_\eta}}(\la_n^{(q)})^{2{H_\eta}-2}, 
			\end{split}\raisetag{-4\baselineskip}\end{equation}
		which, in combination with previous arguments for the contribution of $A^{\eta}$, shows that
		\begin{equation}\label{eq:size-xi} 
			\E[(\xi^{n,q}_j)^2]^{1/2}\lec k_n^{\eps}\del^{1+{H_\eta}}(\la_n^{(q)})^{{H_\eta}-1}
		\end{equation}
		uniformly in $n$ and $j$, with arbitrarily small $\eps>0$. Next, observe that $\xi^n_j$ is $\calf_{j\del}$-measurable with $\E[\xi^n_j \mid \calf_{(j+1-\la_n^{(q+1)})\del}]=0$. Therefore, using a martingale size estimate for $q=0,\dots, Q-1$ and a standard size estimate for $q=Q$ (see \cite[Appendix A]{CDL22}), we obtain 
		\begin{equation}\label{eq:sum-xi-q} \begin{split}
				&\E\biggl[\sup_{t\in[0,T]}\biggl(\sum_{j=1}^{[t/\del]-\ell k_n} \xi^{n,q}_j\biggr)^2\biggr]^{1/2}\\
				&\qquad\lec \begin{cases}
					(\la_n^{(q+1)}/\del)^{1/2}k_n^{\eps}\del^{1+{H_\eta}}(\la_n^{(q)})^{{H_\eta}-1}	&\text{if } q\leq Q-1, \\
					k_n^\eps \del^{H_\eta} (\la_n^{(Q)})^{{H_\eta}-1}	&\text{if } q=Q. 
			\end{cases} \end{split}
		\end{equation}
		We want this to go to zero faster than $(k_n\del)^{1/2}$ for all $q=0,\dots,Q$. Because we can replace $\eps$ by $\frac12\eps$ in the last display, it suffices to start with $\la_0=0$ and then define $\la_1,\la_2,\dots$ iteratively using the relation
		\begin{equation}\label{eq:iteration} \begin{split}
				&-\tfrac12 \la_{q+1} +(\tfrac12-\eps)\kappa  +{H_\eta}+(1-{H_\eta})\la_q=0\\
				&\qquad \iff \la_{q+1}=(1-2\eps)\kappa+2{H_\eta}+2(1-{H_\eta})\la_q.\end{split}
		\end{equation}
		The solution to this recurrence equation is 
		\begin{equation}\label{eq:la-q} 
			\la_q= \frac{((1-2\eps)\kappa+2{H_\eta})((2-2{H_\eta})^q-1)}{1-2{H_\eta}},
		\end{equation} 
		from which we see that $\la_q\to\infty$ if we keep iterating. Let $Q$ be the smallest $Q$ such that $\la_Q$, computed from the formula \eqref{eq:la-q}, is bigger than $(\frac{1-\kappa}2+\kappa \eps-{H_\eta})/(1-{H_\eta})$, which is smaller than $1$ if $\eps$ is small. Replacing $\la_Q$ by a number between this threshold and $1$ (if $\la_Q$ from \eqref{eq:la-q} exceeds $1$), we obtain $\E [ \sup_{t\in[0,T]}(\sum_{j=1}^{[t/\del]-\ell k_n} \xi^{n,q}_j )^2 ]^{1/2}=o((k_n\del)^{1/2})$ for all $q=0,\dots, Q$, proving the lemma for $H_\eta\in(0,\frac12)$.
		
		If $H_\eta=\frac12$, things are much simpler. Indeed, in this case, $\int_0^r \Delta_{r-u} g_{H_\eta} (u-s)\btheta_s d\ov \bW_s = \int_u^r \btheta_s d\ov \bW_s$, so 
		$(k_n\del)^{-1/2}\sum_{j=1}^{[t/\del]-\ell k_n} \xi^n_j= O_\P(\del^{-1/4}\del^{-1} (k_n\del)^{-1} k_n \del^2 \del^{1/2})= O_\P(\del^{1/4})$.
	\end{proof}

	\section{Proof of Propositions~\ref{prop:approx1} and \ref{prop:approx3}}\label{sec:approx13}

	\begin{proof}[Proof of Proposition \ref{prop:approx3}]
		The   proposition follows from Lemmas~\ref{lem:drift}--\ref{lem:Q3}.
	\end{proof}

	\begin{lemma}\label{lem:drift}
		Recall  \eqref{eq:5terms} and that  $y_t=\int_0^t \si_s dW_s$. Under Assumption~\ref{ass:CLT'}, we have   $(k_n\del)^{-1/2}(Z^{n,\ell}_3(t)-\wt Z^{n,\ell}_3(t))\limL0$ and $(k_n\del)^{-1/2}(Z^{\prime n,\ell}_3(t)- M^{\prime n,\ell}_3(t))\limL0$ as $n\to\infty$, where
		\begin{equation}\label{eq:wtZn3} \begin{split}
				\wt Z^{n,\ell}_3(t)&=	\frac{2(k_n\delta_n)^{-1-2H} }{k_{n}}  \sum_{i=1}^{[ t/\delta_{n} ]-(\ell+2)k_n+1 }  D^n_{1,i+\ell k_n}\\
				&\quad\times
				\int_{(i-1)\delta_{n}}^{(i-1+2k_{n})\delta_{n}} \chi(\tfrac{[s/\del]-i+1}{2k_n-1})(y_{s}-y_{[s/\del]\del})dy_s.\end{split}
		\end{equation}
	\end{lemma}
	\begin{proof} We only consider the approximation of $Z^{n,\ell}_3(t)$; the arguments for $Z^{\prime n,\ell}_3(t)$ are analogous.
		Using the equality $xy-x_0y_0=(x-x_0)y_0+x(y-y_0)$,	we can decompose the difference $Z^{n,\ell}_3(t)-\wt Z^{n,\ell}_3(t)=E^n_1(t)+E^n_2(t)+E^n_3(t)$, where
		\begin{align*}
			E^n_1(t)&=\begin{multlined}[t][0.85\textwidth] \frac{2	(k_n\delta_n)^{-1-2H} }{k_{n}}  \sum_{i=1}^{[ t/\delta_{n} ]-(\ell+2)k_n+1 }	D^n_{1,i+\ell k_n}\\
				\times 
				\int_{(i-1)\delta_{n}}^{(i-1+2k_{n})\delta_{n}} \chi(\tfrac{[s/\del]-i+1}{2k_n-1})\int_{[s/\del]\del}^s b_r drdy_s
				, \end{multlined} \\
			E^n_2(t)&= \begin{multlined}[t][0.85\textwidth]	\frac{2(k_n\delta_n)^{-1-2H} }{k_{n}}  \sum_{i=1}^{[ t/\delta_{n} ]-(\ell+2)k_n+1 }D^n_{1,i+\ell k_n}\\	\times  
				\int_{(i-1)\delta_{n}}^{(i-1+2k_{n})\delta_{n}} \chi(\tfrac{[s/\del]-i+1}{2k_n-1})(x_{s}-x_{[s/\del]\del})b_sds
				, \end{multlined} \\
			E^n_3(t)&= 	\begin{multlined}[t][0.85\textwidth] \frac{2(k_n\delta_n)^{-1-2H}}{k_{n}}  \sum_{i=1}^{[ t/\delta_{n} ]-(\ell+2)k_n+1 }D^n_{2,i+\ell k_n}	\\ \times  
				\int_{(i-1)\delta_{n}}^{(i-1+2k_{n})\delta_{n}} \chi(\tfrac{[s/\del]-i+1}{2k_n-1})
				(x_{s}-x_{[s/\del]\del}) dx_s
				.\end{multlined}
		\end{align*}
		The first term is the easiest to deal with. The $dy_s$-integral  is of order $\del(k_n\del)^{1/2}$, while $D^n_{1,i+\ell k_n}$ is of order $(k_n\del)^{1+H}$ by \eqref{eq:C1}. Hence, $$E^n_1(t)=O_\P(k_n\del)^{-2-2H}\del(k_n\del)^{1/2}(k_n\del)^{1+H}=o((k_n\del)^{1/2})$$ for any $\kappa\geq \frac{2H}{2H+1}$. 
		
		Next, consider $E^n_3(t)$ and denote the  $dx_{s}$-integral by $Y^n_i$.
		 Clearly, we have $\E[(Y^n_i)^2]^{1/2}\lec \del^{1/2}(k_n\del)^{1/2}$, uniformly in $n$ and $i$. Interchanging summation over $i$ with the integral defining $D^n_{2,i+\ell k_n}$ in \eqref{eq:5terms}, we have that 
		\[ E^n_3(t)= \frac{2(k_n\delta_n)^{-1-2H}}{k_{n}}  \int_{\ell k_n\del}^{([t/\del]-k_n)\del} \sum_{i=[s/\del]-(\ell+1)k_n+2}^{[s/\del]-\ell k_n+1} Y^n_i (A_{s+k_n\del}-A_s) ds. \]
		The sum ranges over $O(k_n)$ many terms only. Thus, by Lemma~\ref{lem:Ai},
		\[ \E\biggl[\sup_{t\in[0,T]} \lvert E^n_3(t)\rvert\biggr]\lec (k_n\delta_n)^{-1-2H}\del^{1/2} (k_n\del)^{1/2}(k_n\del)^{(2H+1/2)\wedge 1}\int_0^T (1+s^{-H}) ds.\]
		Distinguishing the two cases $H\leq \frac14$ and $H\in(\frac14,\frac12)$, one can verify that the last line is $o((k_n\del)^{1/2})$ for all $\kappa\in[\frac{2H}{2H+1},\frac12]$ and $H\in(0,\frac12)$. We postpone the analysis of $E^n_3(t)$ if $H=\frac12$ to the end of this proof.
		
		The term $E^n_2(t)$ is more complicated. Let us first try a power-counting argument as before: the  $ds$-integral is of order $\del^{1/2}k_n\del$, while $D^n_{1,i+\ell k_n}=O_\P((k_n\del)^{1+H})$, so $E^n_2(t)=O_\P((k_n\del)^{-2-2H}\del^{1/2}k_n\del(k_n\del)^{1+H})$, which as the reader can check, is $o_\P((k_n\del)^{1/2}$ if $\kappa>\frac{2H}{2H+1}$ but unfortunately only $O_\P((k_n\del)^{1/2})$ if $\kappa=\frac{2H}{2H+1}$. While this simple approach fails for the boundary case $\kappa=\frac{2H}{2H+1}$, it shows one important point: when trying to improve the bound, we are allowed to make any modifications that generate an asymptotically vanishing error (the speed can be arbitrarily slow). For instance, we may replace  $x$ by $y$ and, thanks to \eqref{eq:b-2}, $b_s$ by $b_{(i-1)\del}$ in the definition of $E^n_2(t)$, so that we only have to analyze
		\begin{equation}\label{eq:E2prime} 
			\begin{split}
				& \frac{2(k_n\delta_n)^{-1-2H}}{k_{n}}  \sum_{i=1}^{[ t/\delta_{n} ]-(\ell+2)k_n+1 } b_{(i-1)\del}	  
				\int_{(i-1)\delta_{n}}^{(i-1+2k_{n})\delta_{n}} \chi(\tfrac{[s/\del]-i+1}{2k_n-1})(y_{s}-y_{[s/\del]\del})ds
				\\
				&\quad \times \int_0^{(i-1+(\ell+2)k_n)\delta_n} \Delta^{2}_{k_n\del}G_H((i+\ell k_n-1)\del-r) (\eta_rdW_r+\wh\eta_r d\wh W_r).
			\end{split} \!\!\!
		\end{equation}
		Since $y_{s}-y_{[s/\del]\del}=\int_{[s/\del]\del}^s \si_r dW_r$, we can use the stochastic Fubini theorem to rewrite the $ds$-integral above as 
		$$ \int_{(i-1)\del}^{(i-1+2k_n)\del} \int_r^{(r+\del)\wedge (i-1+2k_n)\del} \chi(\tfrac{[s/\del]-i+1}{2k_n-1})(y_{s}-y_{[s/\del]\del})ds \si_r dW_r. $$
		We do not really need the explicit form of the new $ds$-integral, so let us denote it by  $\psi^n_r$ and only remark that $\E[\sup_{r\in[0,T]} \lvert\psi^n_r\rvert^p]^{1/p}\lec \del^{3/2}$ for all $p>0$. Using integration by parts, we can now write \eqref{eq:E2prime} as $E^n_{21}(t)+E^n_{22}(t)+E^n_{23}(t)$, where
		\begin{align*}
			E^n_{21}(t)	&= \begin{multlined}[t][0.85\textwidth](k_n\delta_n)^{-1-2H} \frac{2}{k_{n}}  \sum_{i=1}^{[ t/\delta_{n} ]-(\ell+2)k_n+1 } b_{(i-1)\del}	  \int_{(i-1)\del}^{(i-1+2k_n)\del}  \psi^n_r \si_r\\
				\times\int_0^{r} \Delta^{2}_{k_n\del}G_H((i+\ell k_n-1)\del-u) (\eta_udW_u+\wh\eta_u d\wh W_u) dW_r,\end{multlined} \\
			E^n_{22}(t)	&=\begin{multlined}[t][0.85\textwidth] (k_n\delta_n)^{-1-2H} \frac{2}{k_{n}}  \sum_{i=1}^{[ t/\delta_{n} ]-(\ell+2)k_n+1 } b_{(i-1)\del}\\
				\times	 \int_{(i-1)\del}^{(i-1+(\ell+2)k_n)\delta_n} \Delta^{2}_{k_n\del}G_H((i+\ell k_n-1)\del-r)\\
				\times \int_{(i-1)\del}^{r} \psi^n_u\si_udW_u(\eta_rdW_r+\wh\eta_r d\wh W_r),\end{multlined} \\
			E^n_{23}(t)	&= \begin{multlined}[t][0.85\textwidth](k_n\delta_n)^{-1-2H} \frac{2}{k_{n}}  \sum_{i=1}^{[ t/\delta_{n} ]-(\ell+2)k_n+1 } b_{(i-1)\del}	\\
				\times  \int_{(i-1)\del}^{(i-1+2k_n)\del} \Delta^{2}_{k_n\del}G_H((i+\ell k_n-1)\del-r) \psi^n_r \si_r\eta_rdr.\end{multlined}
		\end{align*}
		
		Since
		\begin{equation}\label{eq:intG} 
			\int_{(i-1)\del}^{(i-1+2k_n)\del} \lvert \Delta^{2}_{k_n\del}G_H((i+\ell k_n-1)\del-r) \rvert dr = (k_n\del)^{H+3/2} \int_{\ell-2}^\ell \lvert \Delta^2_1 G_H(r)\rvert dr, 
		\end{equation}
		we have that $E^n_{23}(t)=O_\P((k_n\del)^{-2-2H}(k_n\del)^{H+3/2}\del^{3/2})=o_\P((k_n\del)^{1/2})$ for all $\kappa\geq\frac{2H}{2H+1}$.
		For both $E^n_{21}(t)$ and $E^n_{22}(t)$, note that the $i$th term is $\calf_{(i-1+(\ell+2)k_n)\del}$-measurable with zero $\calf_{(i-1)\del}$-conditional expectation. Moreover, using \eqref{eq:C1}, we have that 
		each summand is of order $(k_n\del)^{1+H}(k_n\del)^{1/2}\del^{3/2}$. We can therefore apply a martingale size estimate (see \cite[Appendix~A]{CDL22}) to both terms and obtain
		\begin{equation}\label{eq:help2}
			\E\biggl[\sup_{t\in[0,T]} \lvert E^n_{21}(t)+E^n_{22}(t)\rvert\biggr]	\lec (k_n\del)^{-1-2H}k_n^{-1}(k_n/\del)^{1/2} (k_n\del)^{1+H}(k_n\del)^{1/2}\del^{3/2},\!\!
		\end{equation}
		which is $o((k_n\del)^{1/2})$.
		
		Lastly, let us come back to $E^n_3(t)$ if $H=\frac12$. As in the case of $E^n_2(t)$, bounding term by term leads to an $O_\P((k_n\del)^{1/2})$ estimate, which is just not enough at the considered rate. But we are allowed to modify $E^n_3(t)$ in the following way at no cost: we replace $\si$ (which appears in $y$) by $\si_{(i-1)\del}$ and $a_r$ (which appears in $A_{s+k_n\del}-A_s$, which in turn appears in $D^n_{2,i+\ell k_n}$) by $a_{(i-1)\del}$. Once these changes are made, the $i$th term in $E^n_3(t)$ will be $\calf_{(i-1+2k_n)\del}$-measurable with zero $\calf_{(i-1)\del}$-conditional expectation, so we can conclude by   a martingale size estimate.
	\end{proof}
	
	Next, using integration by parts, we have that 
	\begin{equation*}
		\wt Z^{n,\ell}_3(t)=M^{n,\ell}_{31}(t)+M^{n,\ell}_{32}(t) + Q^{n,\ell}_3(t),
	\end{equation*}
	where
	\begin{align*}
		Q^{n,\ell}_3(t)&=(k_n\delta_n)^{-1-2H} \frac{2}{k_{n}}  \sum_{i=1}^{[ t/\delta_{n} ]-(\ell+2)k_n+1 }  \int_{(i-1)\del}^{(i-1+2k_{n})\delta_{n}}  \Delta^{2}_{k_n\del} G_H((i-1+\ell k_n)\del-u)\\
		&\quad\times\chi(\tfrac{[u/\del]-i+1}{2k_n-1})(y_{u}-y_{[u/\del]\del}) \si_u\eta_u du.
	\end{align*}
	
	\begin{lemma}\label{lem:M3} Under Assumption~\ref{ass:CLT'}, if $\kappa>\frac{2H}{2H+1}$, then $M^{n,\ell}_{31}+M^{n,\ell}_{32} +M^{\prime n,\ell}_3 \limL0$.
	\end{lemma}
	\begin{proof}
		By \eqref{eq:C1}, the $i$th term in the summation in  $M^{n,\ell}_{31}(t)$, $M^{n,\ell}_{32}(t)$ and $M^{\prime n,\ell}_3(t)$ is of order $(k_n\del)^{1+H}(k_n\del)^{1/2}\del^{1/2}$. Therefore, 
		by a martingale size argument, very similarly to how we obtained \eqref{eq:help2}, 
		it follows that
		\begin{align*}
			&\E\biggl[\sup_{t\in[0,T]} \lvert M^{n,\ell}_{31}(t)+M^{n,\ell}_{32}(t)+M^{\prime n,\ell}_3(t)\rvert\biggr] \\
			&\qquad\lec (k_n\del)^{-1-2H}k_n^{-1}(k_n/\del)^{1/2} (k_n\del)^{1+H}(k_n\del)^{1/2}\del^{1/2},
		\end{align*}
		which is $o((k_n\del)^{1/2})$ if (and only if) $\kappa>\frac{2H}{2H+1}$.
	\end{proof}

	\begin{lemma}\label{lem:Q3}
		Under Assumption~\ref{ass:CLT'},
		we have    $(k_n\del)^{-1/2}(Q^{n,\ell}_3(t)-\cala^{n,\ell,k_n}_t)\limL0$ for any $\ell\geq2$, where $\cala^{n,\ell,k_n}_t$ is defined in \eqref{eq:bias}.
		In addition, we have that $\cala^{n,\ell,k_n}_t=O_\P(k_n^{-1/2-H})$. In particular, 
		if $\kappa>\frac{1}{2+2H}$, then $(k_n\del)^{-1/2}\cala^{n,\ell,k_n}\limL0$. The last condition is satisfied with $\kappa=\frac{2H}{2H+1}$ if and only if $H> \frac14(\sqrt{5}-1)\approx 0.3090$.
	\end{lemma}

	\begin{proof} 
		In a first step, we decompose $Q^{n,\ell}_3(t)=Q^{n, \ell}_{31}(t)+Q^{n, \ell}_{32}(t)$, where
		\begin{equation}\begin{split}
				Q^{n,\ell}_{31}(t)  &=\begin{multlined}[t][0.8\textwidth] 2(k_{n}\delta_{n})^{-1-2H} \frac{1}{k_{n}}
					\sum_{i=1}^{[ t/\delta_{n} ] - (\ell+2)k_{n}+1} 
					\sum_{j=0}^{ 2k_{n}-1}	\chi(\tfrac{j}{2k_n-1})\\
					\times\int_{ (i-1+j)\delta_{n}}^{(i+j)\delta_{n}}  
					\Delta^2_{k_n\del} G_H((i-1+\ell k_n)\delta_{n}-u) \\
					\times 
					(y_{u}-y_{(i+j-1)\delta_{n}})(\si_{u}	\eta_{u}-\si_{(i+j-1)\delta_{n}}\eta_{(i+j-1)\delta_{n}})  du,\end{multlined}\\
				Q^{n,\ell}_{32}(t)  &= \begin{multlined}[t][0.8\textwidth]2(k_{n}\delta_{n})^{-1-2H} \frac{1}{k_{n}}
					\sum_{i=1}^{[ t/\delta_{n} ] - (\ell+2)k_{n}+1} 
					\sum_{j=0}^{ 2k_{n}-1}	\chi(\tfrac{j}{2k_n-1})\\
					\times\int_{ (i-1+j)\delta_{n}}^{(i+j)\delta_{n}}  
					\Delta^2_{k_n\del} G_H((i-1+\ell k_n)\delta_{n}-u) \\
					\times 
					(y_{u}-y_{(i+j-1)\delta_{n}})\si_{(i-1+j)\delta_{n}}\eta_{(i-1+j)\delta_{n}}  du.\end{multlined}
			\end{split}
		\end{equation}
		Let us consider $Q^{n,\ell}_{32}(t)$ first and interchange the sums over $i$ and $j$. For every fixed $j$, we observe that the $i$th term is $\calf_{(i+j)\del}$-measurable with vanishing $\calf_{(i+j-1)\del}$-conditional expectation. Moreover, similarly to \eqref{eq:intG}, 
		\begin{align*}
			\int_{(i-1+j)\del}^{(i+j)\del} \lvert \Delta^{2}_{k_n\del}G_H((i+\ell k_n-1)\del-u) \rvert du& = (k_n\del)^{H+3/2} \int_{\ell-{\frac{j+1}{k_n}}}^{\ell-\frac j {k_n}} \lvert \Delta^2_1 G_H(u)\rvert du \\
			&\lec \frac{(k_n\del)^{H+3/2}}{k_n}. 
		\end{align*}  
		Therefore, for every $j$, the sum over $i$ is a martingale sum, which yields
		\[ \E\biggl[\sup_{t\in[0,T]}\lvert Q^{n,\ell}_{32}(t) \rvert\biggr]\lec (k_n\del)^{-1-2H}\frac{(k_n\del)^{H+3/2}}{k_n}\del^{-1/2}  \del^{1/2} = o((k_n\del)^{1/2}) \]
		for all $\kappa\geq\frac{2H}{2H+1}$.
		
		Consequently, we only have to consider $Q^{n,\ell}_{31}(t)$ further, which can be rewritten as
		\begin{align*}
			Q^{n,\ell}_{31}(t)	&= 2(k_{n}\delta_{n})^{-1/2-H} \frac{1}{k_{n}}
			\sum_{i=1}^{[ t/\delta_{n} ] - (\ell+2)k_{n}+1} 
			\int_{ (i-1)\delta_{n}}^{(i+2k_n-1)\delta_{n}}  \chi(\tfrac{[u/\del]-i+1}{2k_n-1})
			\\
			&\quad\times 	\Delta^2_{1} G_H(\tfrac{i-1-u/\del}{k_n}+\ell)
			(y_{u}-y_{[u/\del]\del})(\si_{u}	\eta_{u} -\si_{[u/\del]\del}\eta_{[u/\del]\del}) du\\
			&=2(k_{n}\delta_{n})^{-1/2-H} \frac{1}{k_{n}}
			\int_{0}^{([t/\del]-\ell k_n)\del} \sum_{i=([u/\del]-2k_n+2)\vee 1}^{([ u/\delta_{n} ]+1) \wedge ([t/\del]- (\ell+2)k_{n}+1)}  \chi(\tfrac{[u/\del]-i+1}{2k_n-1})
			\\
			&\quad\times\Delta^2_{1} G_H(\tfrac{i-1-u/\del}{k_n}+\ell)
			(y_{u}-y_{[u/\del]\del})(\si_{u}	\eta_{u} -\si_{[u/\del]\del}\eta_{[u/\del]\del}) du\\
			&=2(k_{n}\delta_{n})^{-1/2-H} \frac{1}{k_{n}}
			\int_{0}^{([t/\del]-\ell k_n)\del} \sum_{i=(1-2k_n)\vee (-[u/\del])}^{0 \wedge ([t/\del]-[u/\del]- (\ell+2)k_{n})}  \chi(\tfrac{-i}{2k_n-1})
			\\
			&\quad\times  \Delta^2_{1} G_H(\tfrac{i-\{u/\del\}}{k_n}+\ell)
			(y_{u}-y_{[u/\del]\del})(\si_{u}	\eta_{u} -\si_{[u/\del]\del}\eta_{[u/\del]\del}) du.
		\end{align*}
		Similarly to how we proved Lemma~\ref{lem:QV}, one can use \eqref{eq:Delta-G} to show that $(k_n\del)^{-1/2} (Q^{n,\ell}_{31}(t)-\wt Q^{n,\ell}_{31}(t)) \limL 0$, where
		\begin{align*}
			\wt Q^{n,\ell}_{31}(t)	&=2(k_{n}\delta_{n})^{-1/2-H} \frac{1}{k_{n}}
			\int_{0}^{([t/\del]-\ell k_n)\del} \sum_{i=0}^{2k_n-1}  \chi(\tfrac{i}{2k_n-1})\Delta^2_{1} G_H(\tfrac{-i-\{u/\del\}}{k_n}+\ell)
			\\
			&\quad\times  
			(y_{u}-y_{[u/\del]\del})(\si_{u}	\eta_{u} -\si_{[u/\del]\del}\eta_{[u/\del]\del}) du.
		\end{align*} 
		In fact, we can further change the upper bound of the integral and replace $\wt Q^{n,\ell}_{31}(t)$ by 
		\begin{equation}\label{eq:hat-Q}\begin{split}
				\wh Q^{n,\ell}_{31}(t)	&=2(k_{n}\delta_{n})^{-1/2-H} \frac{1}{k_{n}}
				\int_{0}^t \sum_{i=0}^{2k_n-1}   \chi(\tfrac{i}{2k_n-1})\Delta^2_{1} G_H(\tfrac{-i-\{u/\del\}}{k_n}+\ell)
				\\
				&\quad\times  
				(y_{u}-y_{[u/\del]\del})(\si_{u}	\eta_{u} -\si_{[u/\del]\del}\eta_{[u/\del]\del}) du.
		\end{split}\end{equation}
		Indeed, by \eqref{eq:Delta-G}, $\E[\sup_{t\in[0,T]}\lvert   \wt Q^{n,\ell}_{31}(t)-\wh Q^{n,\ell}_{31}(t)\rvert]\lec (k_n\del)^{1/2-H}\del^{1/2+H} =o((k_n\del)^{1/2})$.
		
		Now recall the definition of $\chi(t)$, which is $-1$ for $t<\frac12$ and $1$ for $t\geq\frac12$. Therefore,
		\begin{align*}
			\wh Q^{n,\ell}_{31}(t)	&=\begin{multlined}[t][0.85\textwidth]\frac{2(k_{n}\delta_{n})^{-1/2-H} }{k_{n}}
				\int_{0}^{t} \sum_{i=0}^{k_n-1} \Bigl\{\Delta^2_{1} G_H(\tfrac{-i-\{u/\del\}}{k_n}+\ell-1)
				\\
				-\Delta^2_{1} G_H(\tfrac{-i-\{u/\del\}}{k_n}+\ell)\Bigr\}
				(y_{u}-y_{[u/\del]\del})(\si_{u}	\eta_{u} -\si_{[u/\del]\del}\eta_{[u/\del]\del}) du\end{multlined}\\
			&=\begin{multlined}[t][0.85\textwidth] - \frac{2(k_{n}\delta_{n})^{-1/2-H}}{k_{n}}
				\int_{0}^{t} \sum_{i=0}^{k_n-1} \Delta^3_1 G_H (\ell-1-\tfrac{i+\{u/\del\}}{k_n}) \\
				\times (y_{u}-y_{[u/\del]\del})(\si_{u}	\eta_{u} -\si_{[u/\del]\del}\eta_{[u/\del]\del}) du,\end{multlined}
		\end{align*}
		which shows that $\wh Q^{n,\ell}_{31}(t)$ is nothing else but the bias term $\cala^{n,\ell,k_n}_t$. This establishes the first claim of the lemma. The second follows from \eqref{eq:bias} by observing that $\Delta^3_1 G_H$ is a bounded function (and, of course, that $y_u-y_{[u/\del]\del}$ and $\si_u\eta_u-\si_{[u/\del]\del}\eta_{[u/\del]\del}$ are of order $\del^{1/2}$ and $\del^H$, respectively). The last two assertions are obvious.
	\end{proof}
	
	\begin{proof}[Proof of Proposition \ref{prop:approx1}] 
		By \eqref{eq:Jn1}, we have that 
		\begin{align*}
			J^n_{1,i}	&= \frac{2}{k_n\del} \int_{(i-1)\del}^{(i+k_n-1)\del} (y_s-y_{[s/\del]\del}) dy_s+\frac{2}{k_n\del} \int_{(i-1)\del}^{(i+k_n-1)\del} (y_s-y_{[s/\del]\del}) b_sds\\
			&\quad+\frac{2}{k_n\del} \int_{(i-1)\del}^{(i+k_n-1)\del} \int_{[s/\del]\del}^s b_r dr dy_s,
		\end{align*}
		which implies
		\begin{align*}
			J^n_{1,i+k_n}-J^n_{1,i}	&=\frac{2}{k_n\del} \int_{(i-1)\del}^{(i+2k_n-1)\del} \chi(\tfrac{[s/\del]-i+1}{2k_n-1})(y_{s}-y_{[s/\del]\del}) dy_s \\
			&\quad+\frac{2}{k_n\del} \int_{(i-1)\del}^{(i+2k_n-1)\del} \chi(\tfrac{[s/\del]-i+1}{2k_n-1})(y_{s}-y_{[s/\del]\del}) b_sds\\
			&\quad +\frac{2}{k_n\del} \int_{(i-1)\del}^{(i+2k_n-1)\del} \chi(\tfrac{[s/\del]-i+1}{2k_n-1})\int_{[s/\del]\del}^s b_r dr dy_s.
		\end{align*}
		Clearly, the last term is $O_\P(\sqrt{\del/k_n})$, while $J^n_{1,i+k_n}-J^n_{1,i}=O_\P(\sqrt{1/k_n})$. Therefore, the contribution of the former to $Z^{n,\ell}_1(t)$ is $O_\P((k_n\del)^{1-2H}(k_n\del)^{-1}k_n^{-1/2}(\del/k_n)^{1/2})=O_\P(k_n^{-1-2H}\del^{1/2-2H})$, which, as the reader may verify, is $o_\P((k_n\del)^{1/2})$ for all $\kappa\in[\frac{2H}{2H+1},\frac12]$. Therefore, 
		\[ (k_n\del)^{-1/2}(Z^{n,\ell}_1(t)-M^{n,\ell}_1(t))\approx (k_n\del)^{-1/2}(F^n_1(t)+ F^n_2(t)), \]
		where
		\begin{align*}
			F^n_1(t)&=4(k_n\delta_n)^{-1-2H} \frac{1}{k_{n}}\sum_{i=1}^{[t/\del]-(\ell+2)k_n+1} \int_{(i-1)\del}^{(i+2k_n-1)\del} \chi(\tfrac{[s/\del]-i+1}{2k_n-1})(y_{s}-y_{[s/\del]\del}) dy_s   \\
			&\quad\times \int_{(i+\ell k_n-1)\del}^{(i+(\ell+2)k_n-1)\del} \chi(\tfrac{[s/\del]-i-\ell k_n+1}{2k_n-1})(y_{s}-y_{[s/\del]\del}) b_sds,\\
			F^n_2(t)&=4(k_n\delta_n)^{-1-2H} \frac{1}{k_{n}}\sum_{i=1}^{[t/\del]-(\ell+2)k_n+1}\int_{(i+\ell k_n-1)\del}^{(i+(\ell+2)k_n-1)\del} \chi(\tfrac{[s/\del]-i-\ell k_n+1}{2k_n-1}) \\
			&\quad\times(y_{s}-y_{[s/\del]\del}) dy_s  \int_{(i-1)\del}^{(i+2k_n-1)\del} \chi(\tfrac{[s/\del]-i+1}{2k_n-1})(y_{s}-y_{[s/\del]\del}) b_sds.
		\end{align*}
		
		Because $\ell\geq2$, the $i$th term in $F^n_2(t)$ is $\calf_{(i+(\ell+2)k_n-1)\del}$-measurable with a vanishing $\calf_{(i+\ell k_n-1)\del}$-conditional expectation. Thus, by a martingale size estimate (see \cite[Appendix~A]{CDL22}) and the bounds found in the previous paragraph, we obtain that
		\begin{align*}
			\E\biggl[\sup_{t\in[0,T]}\lvert F^n_2(t) \rvert\biggr] &\lec (k_n\del)^{-1-2H}k_n^{-1}\sqrt{k_n/\del}(k_n\del^2)^{1/2}k_n\del^{3/2}=(k_n\del)^{-2H}\del  \\
			&=o((k_n\del)^{1/2})
		\end{align*}  
		for all $\kappa\in[\frac{2H}{2H+1},\frac12]$. Regarding $F^n_1(t)$, observe that if we just applied a term-by-term size estimate, we would obtain $(k_n\del)^{-1/2}F^n_1(t)=o_\P(1)$ if $\kappa>\frac{2H}{2H+1}$ but only $O_\P(1)$ if  $\kappa=\frac{2H}{2H+1}$. To handle the latter case, note that we can replace $b_s$ in $F^n_1(t)$ by $b_{(i-1)\del}$ (by the preceding arguments and \eqref{eq:b-2}, the error is $o_\P((k_n\del)^{1/2})$). After doing so, the $i$th {\color{blue}term} in $F^n_1(t)$ will have a zero $\calf_{(i+\ell k_n-1)\del}$-conditional expectation, so applying another martingale size estimate yields $F^n_1(t)=o_\P((k_n\del)^{1/2})$.
		
		It remains to prove the last statement of the proposition. Because $\ell\geq2$, it is easy to see that the $i$th term in \eqref{eq:M1} is $\calf_{(i-1+(\ell+2)k_n)\del}$-measurable while having a zero $\calf_{(i-1+\ell k_n)\del}$-conditional mean. By yet another martingale size estimate, it follows that 
		\[ \E\biggl[\sup_{t\in[0,T]}\lvert M^{n,\ell}_1(t) \rvert\biggr] \lec (k_n\del)^{-1-2H}k_n^{-1}\sqrt{k_n/\del}k_n\del^2 =o((k_n\del)^{1/2})\]
		for   $\kappa>\frac{2H}{2H+1}$. 
	\end{proof}
	
	\section{Details for Section~\ref{sec:CLT2}}\label{sec:details}
	
	\begin{proof}[Proof of Proposition~\ref{prop:Mapprox}]
		Let us start with $M^{n,\ell,k_n}_2(t)$. Interchanging summation over $i$ with the $d\bW_r$-integral in \eqref{eq:M2} and breaking the latter into small pieces of length $\del$, we can rewrite
		$
		\del^{-(1-\kappa)/2} M^{n,\ell,k_n}_2(t)	=\sum_{j=1}^{[t/\del]} \wt\zeta^{n,j,\ell,k_n}_2, $
		where
		\begin{equation}\label{eq:zeta2}\begin{split}
				\wt \zeta^{n,j,\ell,k_n}_2	&=\frac{\del^{-(1-\kappa)/2}(k_n\del)^{-1-2H}}{k_n}\int_{(j-1)\del}^{j\delta_n}\int_0^r \sum_{i= ([r/\del]-(\ell+2)k_n+2)\vee1}^{[t/\del]-(\ell+2)k_n+1} \\
				&\quad\times\Bigl\{ \Delta^{2}_{k_n\del}G_H((i-1)\del-r) \Delta^{2}_{k_n\del}G_H((i+\ell k_n-1)\del-u)   \\
				&\qquad+\Delta^{2}_{k_n\del}G_H((i+\ell k_n-1)\del-r)  \Delta^{2}_{k_n\del}G_H((i-1)\del-u)\Bigr\} \bta_ud\bW_u \bta_rd\bW_r\\
				&=\del^{-(1-\kappa)/2}\int_{(j-1)\del}^{j\delta_n}\int_0^r \frac1{k_n} \Biggl( \sum_{i= (1-(\ell+2)k_n)\vee (-[r/\del])}^{[t/\del]-[r/\del]-(\ell+2)k_n }   \Delta^{2}_{1}G_H(\tfrac{i-\{r/\del\}}{k_n}) \\
				&\quad\times \Delta^{2}_{1} G_H(\tfrac{i-\{r/\del\}}{k_n}+\tfrac{r-u}{k_n\del}+\ell) +\sum_{i= (1-2k_n)\vee (\ell k_n-[r/\del])}^{[t/\del]-[r/\del]-2k_n}  \Delta^2_1 G_H(\tfrac{i-\{r/\del\}}{k_n})\\
				&\quad\times\Delta^2_1G_H(\tfrac{i-\{r/\del\}}{k_n}+\tfrac{r-u}{k_n\del}-\ell)\Biggr) \bta_ud\bW_u \bta_rd\bW_r.
			\end{split}\raisetag{-6.7\baselineskip}\end{equation}
		Let us bound the $p$th moment of $\wt\zeta^{n,j,\ell,k_n}_2$ for $p\geq2$ and draw some conclusions. 
		By the Burkholder--Davis--Gundy inequality and similar steps to \eqref{eq:step1} and \eqref{eq:step2}, we have that  
		\begin{align*}
			\E[\lvert\wt\zeta^{n,j,\ell,k_n}_2\rvert^p]	&\lec \del^{-(1-\kappa)p/2}\biggl( \int_{(j-1)\del}^{j\del} \int_0^r \biggl( \frac1{k_n} \sum_{i=1-2k_n}^\infty \Bigl\{\Delta^2_1 G_H(\tfrac{i-\{r/\del\}}{k_n})\\
			&\quad\times \Delta^2_1G_H(\tfrac{i-\{r/\del\}}{k_n}+\tfrac{r-u}{k_n\del}+\ell) +\Delta^2_1 G_H(\tfrac{i-\{r/\del\}}{k_n})\\
			&\quad\times \Delta^2_1G_H(\tfrac{i-\{r/\del\}}{k_n}+\tfrac{r-u}{k_n\del}-\ell)\Bigr\} \biggr)^2 dudr\biggr)^{p/2}.
		\end{align*}
		Changing $(r-u)/(k_n\del)$ to $u$ and noticing that $\frac1{k_n}\sum_{i=1-2k_n}^\infty$ is a Riemann sum, we obtain from \eqref{eq:Delta-G-2} that
		\begin{equation}\label{eq:help4} 
			\begin{split}
				\E[\lvert\wt\zeta^{n,j,\ell,k_n}_2\rvert^p]	&\lec\begin{multlined}[t][0.7\textwidth]  \biggl( \int_{(j-1)\del}^{j\del} \int_0^\infty \biggl(\int_{-2}^\infty \Bigl\{\Delta^2_1 G_H(v)\Delta^2_1 G_H(v+u+\ell) \\
					+\Delta^2_1 G_H( v)\Delta^2_1 G_H( v+u-\ell)\Bigr\} dv\biggr)^2 dudr\biggr)^{p/2}\end{multlined}\\
				&\lec \del^{p/2}.
			\end{split}
		\end{equation}
		Consequently, $\wt\zeta^{n,j,\ell,k_n}_2$ is of size $\del^{1/2}$, uniformly in $i$. Because \eqref{eq:M2} is a sum of martingale increments (note that $\E[\wt\zeta^{n,j,\ell,k_n}_2\mid \calf_{(j-1)\del}]=0$), it follows that \eqref{eq:M2} is $O_\P(1)$. This is, of course, expected because \eqref{eq:M2} is supposed to contribute to the CLT. But what this calculation also shows is that before we try to  find the limit of \eqref{eq:M2}, we can make any modifications that lead to an $o_\P(1)$ error. For example, we can replace $\bta_r$ by $\bta_{(j-1)\del}$ (this incurs an $O_\P(\del^H)$ error) and replace $\frac{1}{k_n}$ times 
		the two sums after second equality in \eqref{eq:zeta2} by
		\[\int_{-2}^\infty \Bigl\{\Delta^{2}_{1}G_H(v)\Delta^{2}_{1}G_H(v+\tfrac{r-u}{k_n\del}+\ell)  
		+ \Delta^2_1 G_H(v)\Delta^2_1G_H(v+\tfrac{r-u}{k_n\del}-\ell)\Bigr\} dv\]
		(for modifying the upper and lower bounds of the summation, see the discussion after \eqref{eq:step2};  for the integral approximation, the error is at most $k_n^{-1/2-H}$ because $\Delta^2_1 G_H$ is $(\frac12+H)$-H\"older continuous). We will make two more changes, after which we will arrive at $\zeta^{n,j,\ell,k_n}_2$, hence proving the second relation in \eqref{eq:approx}: first, we   {\color{blue}change} the boundaries of the $d\bW_u$-integral from $\int_0^r$ to $\int_{r-k_n\del^{1-\eps}}^r$, where $\eps>0$ is a small but fixed number. Similarly to \eqref{eq:help4}, one can show that the resulting error is $\del^{\eps(1-H)}$. And second, we replace $\bta_u$ first by $\bta_{r-k_n\del^{1-\eps}}$ and then by $\bta_{(j-1)\del}$, which leads to an $O_\P((k_n\del^{1-\eps})^H)$ error. 
		
		Similar arguments can be employed to show the other two approximations in \eqref{eq:approx}. Note that  thanks to  Proposition \ref{prop:approx1}  and   Lemma~\ref{lem:M3}, we only have to consider the case where   $\kappa=\frac{2H}{2H+1}$.  	 In order to show the first approximation in  \eqref{eq:approx},  we interchange summation and double integration in \eqref{eq:M1} and obtain 
		\begin{multline*}
			M_1^{n,\ell,k_n}(t)	=4(k_n\del)^{-1-2H}\frac1{k_n} \int_{\ell k_n}^{[t/\del]\del} \int_{([s/\del]-(\ell+2)k_n+1)\del\vee0}^{([s/\del]-(\ell-2)k_n)\del\wedge ([t/\del]-\ell k_n)\del}\\ \sum_{i=(2+[s/\del]-(\ell+2)k_n)\vee (2+[r/\del]-2k_n)\vee 1}^{([s/\del]-\ell k_n+1)\wedge ([r/\del]+1)\wedge ([t/\del]-(\ell+2)k_n+1)}  \chi(\tfrac{[s/\del]-i-\ell k_n+1}{2k_n-1}) \chi(\tfrac{[r/\del]-i+1}{2k_n-1})\\
			\times (y_{r}-y_{[r/\del]\del})\si_r dW_r (y_{s}-y_{[s/\del]\del})\si_sdW_s. 
		\end{multline*}
		We change $i+\ell k_n-1-[s/\del]$ to $i$ and, with similar arguments to those after \eqref{eq:help4}, omit the last $\wedge (\cdots)$ and $\vee (\cdots)$ in  the boundaries of both the $dW_r$-integral and the sum over $i$. As a result,
		\begin{multline*}
			\del^{-(1-\kappa)/2}M_1^{n,\ell,k_n}(t)	\approx 4 \del^{-(1-\kappa)/2}(k_n\del)^{-1-2H}\frac1{k_n} \int_{\ell k_n}^{[t/\del]\del} \int_{([s/\del]-(\ell+2)k_n+1)\del}^{([s/\del]-(\ell-2)k_n)\del}\\ \sum_{i=(1-2k_n)\vee ([r/\del]-[s/\del]+(\ell-2)k_n+ 1)}^{0\wedge ([r/\del]-[s/\del]+\ell k_n}  \chi(\tfrac{-i}{2k_n-1}) \chi(\tfrac{[r/\del]-[s/\del]-i+\ell k_n}{2k_n-1})\\
			\times (y_{r}-y_{[r/\del]\del})\si_r dW_r (y_{s}-y_{[s/\del]\del})\si_sdW_s. 
		\end{multline*}
		Because
		\[ \sum_{i=(1-2k_n)\vee (1-2k_n+m)}^{0\wedge m}  \chi(\tfrac{-i}{2k_n-1}) \chi(\tfrac{-i+m}{2k_n-1}) = 2k_n\xi(\tfrac{m}{2k_n}), \]
		the first approximation in \eqref{eq:approx} follows by replacing all $\si$'s by $\si_{[s/\del]\del}$.
		%


		Regarding the last approximation in  \eqref{eq:approx}, we have   analogously to \eqref{eq:zeta2}    that 
		\begin{equation}\label{eq:M3} \begin{split}
				\del^{-(1-\kappa)/2}M^{n,\ell,k_n}_{31}(t) &= \sum_{j=1}^{[t/\del]-\ell k_n} \wt\zeta^{n,j,\ell,k_n}_{31}, \quad 	\del^{-(1-\kappa)/2}M^{n,\ell,k_n}_{32}(t) = \sum_{j=1}^{[t/\del]}\wt \zeta^{n,j,\ell,k_n}_{32},\\
				\del^{-(1-\kappa)/2}M^{\prime n,\ell,k_n}_{3}(t) &= \sum_{j=1+\ell k_n}^{[t/\del]}\wt \zeta^{\prime n,j,\ell,k_n}_{3}, \end{split}
		\end{equation}
		where
		\begin{align*}
			\wt\zeta^{n,j,\ell,k_n}_{31}	&=\frac{2(k_n\del)^{-1/2-H}\del^{-(1-\kappa)/2}}{k_n}\int_{(j-1)\del}^{j\del} \int_0^s \sum_{i=(1-2k_n)\vee (-[s/\del])}^{0\wedge ([t/\del]-[s/\del]-(\ell+2)k_n)} \chi(\tfrac{-i}{2k_n-1}) \\
			&\quad\times\Delta^2_1 G_H(\tfrac{i-\{s/\del\}}{k_n}+\tfrac{s-r}{k_n}+\ell) \bta_r d\bW_r (y_s-y_{[s/\del]\del}) \si_sdW_s,\\
			\wt\zeta^{n,j,\ell,k_n}_{32}	&=\frac{2(k_n\del)^{-1/2-H}\del^{-(1-\kappa)/2}}{k_n}\int_{(j-1)\del}^{j\del} \int_{([r/\del]-(\ell+2)k_n+1)\del\vee0}^{r\wedge ([t/\del]-\ell k_n)\del}  (y_s-y_{[s/\del]\del}) \\
			&\quad\times\sum_{i=(1-2k_n)\vee([r/\del]-[s/\del]-(\ell+2)k_n+1) \vee(-[s/\del])}^{0\wedge ([t/\del]-[s/\del]-(\ell+2)k_n)} \chi(\tfrac{-i}{2k_n-1})\\
			&\quad\times\Delta^2_1 G_H(\tfrac{i-\{s/\del\}}{k_n}-\tfrac{r-s}{k_n}+\ell)  \si_sdW_s\bta_r d\bW_r,\\
			\wt\zeta^{\prime n,j,\ell,k_n}_{3}	&=\frac{2(k_n\del)^{-1/2-H}\del^{-(1-\kappa)/2}}{k_n}\int_{(j-1)\del}^{j\del} \int_0^{([s/\del]-(\ell-2)k_n)\del\wedge ([t/\del]-\ell k_n)\del} \si_s  \\
			&\quad\times(y_s-y_{[s/\del]\del})\sum_{i=(1-2k_n)\vee([r/\del]-[s/\del]+(\ell-2)k_n+1)\vee (\ell k_n-[s/\del])}^{0\wedge ([t/\del]-[s/\del]-2k_n)} \chi(\tfrac{-i}{2k_n-1})\\
			&\quad\times\Delta^2_1 G_H(\tfrac{i-\{s/\del\}}{k_n}+\tfrac{s-r}{k_n}-\ell) \bta_r d\bW_r dW_s.
		\end{align*}
		Since $\kappa=\frac{2H}{2H+1}$, it can be shown similarly to \eqref{eq:help4} and the subsequent paragraph that each of the three terms defined in \eqref{eq:M3} is of order $O_\P(1)$. Therefore, by the same type of modifications (i.e., discretization of $\bta$ and $\si$, dropping  $\wedge (\cdots)$ and $\vee (\cdots)$ in the summation over $i$, approximating sums by integrals, and restricting the $d\bW_r$-integral in $\wt\zeta^{n,j,\ell,k_n}_{31}$ and $\wt\zeta^{\prime n,j,\ell,k_n}_{31}$ to $r\geq s-k_n\del^{1-\eps}$), 
		we obtain $\del^{-(1-\kappa)/2}(M^{n,\ell,k_n}_{31\mid32}(t)-\bar M^{n,\ell,k_n}_{31\mid32}(t))\limL0$ and $\del^{-(1-\kappa)/2}(M^{\prime n,\ell,k_n}_{3}(t)-\bar M^{\prime n,\ell,k_n}_{3}(t))\limL0$, where $$\bar M^{n,\ell,k_n}_{31\mid32}(t)=\sum_{j=1}^{[t/\del]-\ell k_n} \bar \zeta^{n,j,\ell,k_n}_{31\mid32},\qquad \bar M^{\prime n,\ell,k_n}_{3}(t)=\sum_{j=1+\ell k_n}^{[t/\del]} \bar \zeta^{\prime n,j,\ell,k_n}_{3}$$ and 
		\begin{equation}\label{eq:wtzeta31} \begin{split}
				\bar \zeta^{n,j,\ell,k_n}_{31}&=2(k_n\del)^{-1/2-H}\del^{-(1-\kappa)/2}\int_{(j-1)\del}^{j\del} \si_{(j-1)\del}^2 \int_{s-k_n\del^{1-\eps}}^s \int_0^2 \chi(\tfrac{u}2)   \\
				&\quad\times\Delta^2_1 G_H(\tfrac{s-r}{k_n\del} + \ell -u) du\,\bta_{(j-1)\del} d\bW_r  (W_s-W_{(j-1)\del})dW_s,\\
				\bar \zeta^{n,j,\ell,k_n}_{32}&=2(k_n\del)^{-1/2-H}\del^{-(1-\kappa)/2}\int_{(j-1)\del}^{j\del}\! \si_{(j-1)\del}^2 \int_{([r/\del]-(\ell+2)k_n+1)\del}^r \!(W_s-W_{[s/\del]\del}) \\
				&\quad\times\int_0^{2-((\frac{r-s}{k_n\del}-\ell)\vee0)} \chi(\tfrac{u}2) \Delta^2_1 G_H( \ell-\tfrac{r-s}{k_n\del} -u) du dW_s\bta_{(j-1)\del} d\bW_r,\\
				\bar \zeta^{\prime n,j,\ell,k_n}_{3}&=2(k_n\del)^{-1/2-H}\del^{-(1-\kappa)/2}\int_{(j-1)\del}^{j\del} \si_{(j-1)\del}^2 \int_{s-k_n\del^{1-\eps}}^{([s/\del]-(\ell-2)k_n)\del}  (W_s-W_{(j-1)\del})  \\
				&\quad\times\int_0^{2-(( \ell-\frac{s-r}{k_n\del})\vee 0)} \chi(\tfrac u2) \Delta^2_1 G_H(\tfrac{s-r}{k_n\del} - \ell -u) du \bta_{(j-1)\del} d\bW_r dW_s.
			\end{split}\raisetag{-6\baselineskip}
		\end{equation}
		Let us make three observations: First, for any of the three terms in \eqref{eq:wtzeta31}, by a straightforward power-counting argument, if we restrict the inner integral to $((j-1)\del,s)$ or $((j-1)\del,r)$, respectively, the second moment of the resulting term will be of order $\del^{1+2H/(2H+1)}=o(\del)$, showing that the latter is asymptotically negligible (cf.\ \eqref{eq:help4} and the subsequent arguments). Second, by the definition of $\chi(t)$,
		$$ \int_0^2 \chi(\tfrac u2)\Delta^2_1 G_H(\tfrac{s-r}{k_n\del}+\ell-u)du=-\int_0^1 \Delta^3_1 G_H(\tfrac{s-r}{k_n\del} + \ell -u-1) du. $$
		And finally, because $\Delta^2_1 G_H (v)=0$ for $v\leq -2$, there is no harm in extending the $du$-integral in $\bar \zeta^{ n,j,\ell,k_n}_{32}$ and $\bar \zeta^{\prime n,j,\ell,k_n}_{3}$ up to the upper bound $2$. The aforementioned modifications turn $\ov \zeta^{n,j,\ell,k_n}_{32}$ into $\zeta^{n,j,\ell,k_n}_{32}$ and the sum $\ov \zeta^{n,j,\ell,k_n}_{31}+\ov \zeta^{\prime n,j,\ell,k_n}_{3}$ into $\zeta^{n,j,\ell,k_n}_{31}$, which establishes the last relation in \eqref{eq:approx}.
	\end{proof}
	
	\begin{proof}[Proof of Equation \eqref{eq:tight}]
		For any $\nu\in\{1,2,3\}$, we have seen in the proof of Proposition~\ref{prop:Mapprox} that $\E[\lvert\zeta^{n,j,\ell,k_n}_\nu\rvert^p]\lec \del^{p/2}$, uniformly in $j$. Setting $p=4$, we easily obtain that the left-hand side of \eqref{eq:tight} is $O_\P(\del)$.
	\end{proof}
	
	\begin{proof}[Proof of Equation~\eqref{eq:stable}]
		We only show \eqref{eq:stable} for $\nu=2$ as the arguments for $\nu=1$ and $\nu=3$ are similar. 
		Note that $\zeta^{n,j,\ell,k_n}_2$ can be decomposed into two parts, $\zeta^{n,j,\ell,k_n}_{21}$ and $\zeta^{n,j,\ell,k_n}_{22}$, which are defined in the same way as $\zeta^{n,j,\ell,k_n}_2$ in \eqref{eq:wtzeta2}, except that the $d\bW_u$-integral is restricted to $(r-k_n\del^{1-\eps},(j-1)\del)$ and $((j-1)\del,r)$, respectively. By definition, $\zeta^{n,j,\ell,k_n}_{22}$  belongs to the second Wiener chaos with respect to $\bW$, conditionally on $\calf_{(j-1)\del}$. Thus, $\E[\zeta^{n,j,\ell,k_n}_{22}(N_{j\del}-N_{(j-1)\del})\mid \calf_{(j-1)\del}]=0$ by the orthogonality of Wiener chaoses of different orders if $N\in\{W,\wh W\}$ and by the orthogonality of $N$ and $W$ otherwise. If $N$ is orthogonal to $W$, we also have $\E[\zeta^{n,j,\ell,k_n}_{21}(N_{j\del}-N_{(j-1)\del})\mid \calf_{(j-1)\del}]=0$, so let us assume that $N=W$ or $N=\wh W$. Since the two cases are completely analogous, we take $N=W$. Then
		\begin{multline*}
			\E[\zeta^{n,j,\ell,k_n}_{21}(N_{j\del}-N_{(j-1)\del})\mid \calf_{(j-1)\del}]\\	=\del^{-(1-\kappa)/2}\int_{(j-1)\del}^{j\delta_n}\int_{r-k_n\del^{1-\eps}}^{(j-1)\del}\int_{-2}^\infty \Bigl\{\Delta^{2}_{1}G_H(v)\Delta^{2}_{1}G_H(v+\tfrac{r-u}{k_n\del}+\ell)  \\
			+ \Delta^2_1 G_H(v)\Delta^2_1G_H(v+\tfrac{r-u}{k_n\del}-\ell)\Bigr\} dv \bta_{(j-1)\del}d\bW_u \eta_{(j-1)\del}dr.
		\end{multline*}
		Since taking conditional expectation is a contraction on $L^2$, this term is still of size $O_\P(\del)$. Consequently, for the purpose of showing \eqref{eq:stable}, we may replace $\eta_{(j-1)\del}$ and $\bta_{(j-1)\del}$ in the previous display by $\eta_{(j-1-k_n\del^{-\eps})\del}$ and $\bta_{(j-1-k_n\del^{-\eps})\del}$, respectively. Once we have done so, the resulting expression will be $\calf_{(j-1)\del}$-measurable with vanishing $\calf_{(j-1-k_n\del^{-\eps})\del}$-conditional expectation. Therefore, by   a martingale size estimate (see \cite[Appendix~A]{CDL22}), it follows that 
		\begin{equation}\label{eq:help} 
			\E\biggl[ \biggl\lvert\sum_{j=1}^{[t/\del]} \E[\zeta^{n,j,\ell,k_n}_{21}(N_{j\del}-N_{(j-1)\del})\mid\calf_{(j-1)\del}]\biggr\rvert\biggr] \lec (k_n\del^{-\eps-1})^{1/2}\del \to0,
		\end{equation}
		proving \eqref{eq:stable} for $\nu=2$.
	\end{proof}
	
	\begin{proof}[Proof of Equation~\eqref{eq:cov}]
		Again let us start with $\nu=2$. There is no loss of generality to restrict ourselves to $m=1$ and $m'=2$, in which case we simply write $k_n=k_n^{(1)}$ and $k'_n=k_n^{(2)}$. We want to find the limit of 
		\begin{equation}\label{eq:R22}
			R^{n,\ell_1,k_n,\ell_2,k'_n}_{22}(t)=\sum_{j=1}^{[t/\del]} \E[  \zeta^{n,j,\ell_1,k_n}_2 \zeta^{n,j,\ell_2,k'_n}_2\mid \calf_{(j-1)\del}],
		\end{equation}
		where $\ell_1,\ell_2\geq2$, $k_n\sim \theta_1 \del^{-\kappa}$ and $k'_n \sim \theta_2\del^{-\kappa}$. 
		Moreover, by the flexibility we have in the truncation of the $d\bW_u$-integral in \eqref{eq:wtzeta2}, we may and will assume that it runs from $r-k_n\del^{1-\eps}$ to $r$ for both $\zeta^{n,j,\ell,k_n}_1$ and $\zeta^{n,j,\ell,k_n}_2$.
		By It\^o's isometry, we  then have $R^{n,\ell_1,k_n,\ell_2,k'_n}_{22}(t)=\sum_{\iota=1}^3 R^{n,\ell_1,k_n,\ell_2,k'_n}_{22,\iota}(t)$ where
		\begin{align*}
			R^{n,\ell_1,k_n,\ell_2,k'_n}_{22,1}(t)	&=  \begin{multlined}[t][0.75\textwidth] \sum_{j=1}^{[t/\del]} \del^{-(1-\kappa)}\lvert \bta_{(j-1)\del}\rvert^4 \int_{(j-1)\del}^{j\delta_n}\int_{r-k_n\del^{1-\eps}}^r\int_{-2}^\infty \Bigl\{\Delta^{2}_{1}G_H(v)  \\
				\times\Delta^{2}_{1}G_H(v+\tfrac{r-u}{k_n\del}+\ell_1)	+ \Delta^2_1 G_H(v)\Delta^2_1G_H(v+\tfrac{r-u}{k_n\del}-\ell_1)\Bigr\} dv\end{multlined}\\
			&\quad \begin{multlined}[t][0.75\textwidth]\times\int_{-2}^\infty \Bigl\{\Delta^{2}_{1}G_H(w)\Delta^{2}_{1}G_H(w+\tfrac{r-u}{k'_n\del}+\ell_2)   	\\
				+ \Delta^2_1 G_H(w)\Delta^2_1G_H(w+\tfrac{r-u}{k'_n\del}-\ell_2)\Bigr\} dw dudr,\end{multlined}\\
			R^{n,\ell_1,k_n,\ell_2,k'_n}_{22,2}(t)&=\begin{multlined}[t][0.75\textwidth]\sum_{j=1}^{[t/\del]} 
				\del^{-(1-\kappa)}\lvert \bta_{(j-1)\del}\rvert^2 \int_{(j-1)\del}^{j\delta_n}\int_{r-k_n\del^{1-\eps}}^{(j-1)\del}\int_{-2}^\infty \Bigl\{\Delta^{2}_{1}G_H(v)  \\
				\times\Delta^{2}_{1}G_H(v+\tfrac{r-u}{k_n\del}+\ell_1)	+ \Delta^2_1 G_H(v)\Delta^2_1G_H(v+\tfrac{r-u}{k_n\del}-\ell_1)\Bigr\} dv\end{multlined}\\
			&\quad\begin{multlined}[t][0.75\textwidth] \times\int_{r-k_n\del^{1-\eps}}^{u}\int_{-2}^\infty \Bigl\{\Delta^{2}_{1}G_H(w)\Delta^{2}_{1}G_H(w+\tfrac{r-u'}{k'_n\del}+\ell_2) 	+ \Delta^2_1 G_H(w) \\
				\times\Delta^2_1G_H(w+\tfrac{r-u'}{k'_n\del}-\ell_2)\Bigr\} dw\bta_{(j-1)\del} d\bW_{u'}\bta_{(j-1)\del} d\bW_udr,\end{multlined}\\
			R^{n,\ell_1,k_n,\ell_2,k'_n}_{22,3}(t)&=\begin{multlined}[t][0.75\textwidth]\sum_{j=1}^{[t/\del]} 
				\del^{-(1-\kappa)}\lvert \bta_{(j-1)\del}\rvert^2 \int_{(j-1)\del}^{j\delta_n}\int_{r-k_n\del^{1-\eps}}^{(j-1)\del}\int_{-2}^\infty \Bigl\{\Delta^{2}_{1}G_H(v)  \\
				\times	\Delta^{2}_{1}G_H(v+\tfrac{r-u}{k'_n\del}+\ell_2)	+ \Delta^2_1 G_H(v)\Delta^2_1G_H(v+\tfrac{r-u}{k'_n\del}-\ell_2)\Bigr\} dv\end{multlined}\\ &\quad\begin{multlined}[t][0.75\textwidth]\times\int_{r-k_n\del^{1-\eps}}^{u}\int_{-2}^\infty \Bigl\{\Delta^{2}_{1}G_H(w)\Delta^{2}_{1}G_H(w+\tfrac{r-u'}{k_n\del}+\ell_1) + \Delta^2_1 G_H(w) \\
				\times\Delta^2_1G_H(w+\tfrac{r-u'}{k_n\del}-\ell_1)\Bigr\} dw\bta_{(j-1)\del} d\bW_{u'}\bta_{(j-1)\del} d\bW_udr.\end{multlined}
		\end{align*}
		Repeating the argument leading to \eqref{eq:help}	
		shows that $R^{n,\ell_1,k_n,\ell_2,k'_n}_{22,2}(t)$ and $R^{n,\ell_1,k_n,\ell_2,k'_n}_{22,3}(t)$ are $O_\P((k_n\del^{-\eps}/\del)^{1/2}\del)=O_\P((k_n\del^{1-\eps})^{1/2})=o_\P(1)$, and hence they do not contribute to the limit of \eqref{eq:R22}.
		So only $R^{n,\ell_1,k_n,\ell_2,k'_n}_{22,1}(t)$ is asymptotically relevant. 
		
		By a change of variables ($(r-u)/\del^{1-\kappa}$ to $u$),
		\begin{align*}
			R^{n,\ell_1,k_n,\ell_2,k'_n}_{22,1}(t)	&= \begin{multlined}[t][0.75\textwidth] \sum_{j=1}^{[t/\del]}  \lvert \bta_{(j-1)\del}\rvert^4 \int_{(j-1)\del}^{j\delta_n}\int_0^{k_n\del^{\kappa-\eps}}\int_{-2}^\infty \Bigl\{\Delta^{2}_{1}G_H(v) \\
				\times\Delta^{2}_{1}G_H(v+u\tfrac{\del^{-\kappa}}{k_n}+\ell_1) 	+ \Delta^2_1 G_H(v)\Delta^2_1G_H(v+u\tfrac{\del^{-\kappa}}{k_n}-\ell_1)\Bigr\} dv\end{multlined} \\
			&\quad\begin{multlined}[t][0.75\textwidth]\times\int_{-2}^\infty \Bigl\{\Delta^{2}_{1}G_H(w)  \Delta^{2}_{1}G_H(w+u\tfrac{\del^{-\kappa}}{k'_n}+\ell_2) \\
				+ \Delta^2_1 G_H(w)\Delta^2_1G_H(w+u\tfrac{\del^{-\kappa}}{k'_n}-\ell_2)\Bigr\} dw dudr,\end{multlined}
		\end{align*}
		so we obtain $R^{n,\ell_1,k_n,\ell_2,k'_n}_{22,1}(t)\sim \ga^{\ell_1,\theta_1,\ell_2,\theta_2}_{2}(H)  \Ga_2(t)$ once we establish
		\begin{equation}\label{eq:rho22} 
			\begin{split}
				\ga^{\ell_1,\theta_1,\ell_2,\theta_2}_{2}(H) &= \begin{multlined}[t][0.7\textwidth] \int_0^{\infty}\int_{-2}^\infty \Bigl\{\Delta^{2}_{1}G_H(v)\Delta^{2}_{1}G_H(v+u/\theta_1+\ell_1)\\
					+ \Delta^2_1 G_H(v)\Delta^2_1G_H(v+u/\theta_1-\ell_1)\Bigr\} dv\end{multlined} \\
				&\quad	\begin{multlined}[t][0.7\textwidth]	\times \int_{-2}^\infty \Bigl\{\Delta^{2}_{1}G_H(w)\Delta^{2}_{1}G_H(w+u/\theta_2+\ell_2) \\
					+ \Delta^2_1 G_H(w)\Delta^2_1G_H(w+u/\theta_2-\ell_2)\Bigr\} dw du.\end{multlined}
			\end{split} 
		\end{equation}

		By \eqref{eq:Phi} (and its extension to $\ell\in\R$ as shown in the proof), the right-hand side equals
		\begin{equation}\label{eq:interm2} 
			\begin{split}
				& \begin{multlined}[t][0.9\textwidth] \frac{1}{4(2H+1)^2(2H+2)^2} \int_0^{\infty} (\delta^4_1\lvert u/\theta_1+\ell_1\rvert^{2H+2} + \delta^4_1\lvert u/\theta_1-\ell_1\rvert^{2H+2})\\
					\times(\delta^4_1 \lvert u /\theta_2 +\ell_2\rvert^{2H+2}+\delta^4_1 \lvert u /\theta_2-\ell_2\rvert^{2H+2})  du\end{multlined}\\
				&\begin{multlined}[t][0.9\textwidth]\quad=\frac{1}{4(2H+1)^2(2H+2)^2} \biggl(\int_\R \delta^4_1\lvert u/\theta_1+\ell_1\rvert^{2H+2}\delta^4_1 \lvert u /\theta_2 +\ell_2\rvert^{2H+2} du\\
					+\int_\R \delta^4_1\lvert u/\theta_1+\ell_1\rvert^{2H+2}\delta^4_1 \lvert u /\theta_2 -\ell_2\rvert^{2H+2} du\biggr),\end{multlined}
			\end{split}
		\end{equation}
		where the second step follows by symmetry.
		By Parseval's identity,
		\begin{equation}\label{eq:interm} 
			\begin{split}
				&\int_\R \delta^4_1\lvert u/\theta_1+\ell_1\rvert^{2H+2}\delta^4_1 \lvert u /\theta_2 +\ell_2\rvert^{2H+2} du\\
				&\qquad=\frac{1}{(\theta_1\theta_2)^{2H+2}}\int_\R \delta^4_{\theta_1}\lvert u+\ell_1\theta_1\rvert^{2H+2}\delta^4_{\theta_2} \lvert u +\ell_2 \theta_2\rvert^{2H+2} du\\
				&\qquad=\begin{multlined}[t][0.8\textwidth]\frac{1}{2\pi(\theta_1\theta_2)^{2H+2}}\int_\R e^{i(\ell_1\theta_1-\ell_2\theta_2)\xi}\calf[\lvert x\rvert^{2H+2}](\xi)^2 \\
					\times(e^{\frac12 i \theta_1\xi}-e^{-\frac12i\theta_1\xi})^4(e^{\frac12 i \theta_2\xi}-e^{-\frac12i\theta_2\xi})^4 d\xi.\end{multlined}
			\end{split}
		\end{equation}
		The product $(e^{\frac12 i \theta_1\xi}-e^{-\frac12i\theta_1\xi})^4(e^{\frac12 i \theta_2\xi}-e^{-\frac12i\theta_2\xi})^4$ translates into $\delta^4_{\theta_1}\delta^4_{\theta_2}$ in the time domain. Together with \eqref{eq:abs}, this yields
		\begin{align*}
			&\int_\R \delta^4_1\lvert u/\theta_1+\ell_1\rvert^{2H+2}\delta^4_1 \lvert u /\theta_2 +\ell_2\rvert^{2H+2} du\\
			&\quad = \begin{multlined}[t][0.9\textwidth]\frac{2\cos^2(\pi(H+\frac32))\Ga(2H+3)^2}{\pi(\theta_1\theta_2)^{2H+2}}\int_\R e^{i(\ell_1\theta_1-\ell_2\theta_2)\xi} \lvert\xi\rvert^{-4H-6}\\
				\times(e^{\frac12 i \theta_1\xi}-e^{-\frac12i\theta_1\xi})^4(e^{\frac12 i \theta_2\xi}-e^{-\frac12i\theta_2\xi})^4d\xi\end{multlined}\\
			&\quad=\frac{4\cos^2(\pi(H+\frac32))\cos(\pi(2H+\frac52))\Ga(2H+3)^2\Ga(-5-4H)}{\pi(\theta_1\theta_2)^{2H+2}}  \delta^4_{\theta_1}\delta^4_{\theta_2}\lvert \ell_2\theta_2-\ell_1\theta_1\rvert^{4H+5},
		\end{align*}
		where the last step is valid for all $H\in(0,\frac12)\setminus\{\frac14\}$. 
		Inserting this into \eqref{eq:interm2} and simplifying the resulting expression, we finally obtain
		\eqref{eq:rho22}
		if $H\notin\{\frac14,\frac12\}$. To obtain the results for $H\in\{\frac14,\frac12\}$, it suffices by the dominated convergence theorem to let $H\to\frac14$ and $H\to\frac12$ in the  formula established for $H\in(0,\frac12)\setminus\{\frac14\}$. As there is no singularity at $H=\frac12$, this formula   continues to hold for $H=\frac12$. For $H=\frac14$, it suffices to note that 
		$(1-1/\cos(2\pi H))(H-\frac14)\to\frac{1}{2\pi}$  as $H\to\frac14$ and that 
		$$ \lim_{H\to\frac14} \frac{\delta^4_{\theta_1}\delta^4_{\theta_2}\lvert \ell_2\theta_2-\ell_1\theta_1\rvert^{4H+5}}{H-\frac14} = 4 \delta^4_{\theta_1}\delta^4_{\theta_2}[\lvert \ell_2\theta_2-\ell_1\theta_1\rvert^{6}\log\lvert \ell_2\theta_2-\ell_1\theta_1\rvert] $$
		by L'H\^opital's rule. 
		
		\smallskip
		Next, we consider $\nu=1$.  As in \eqref{eq:R22} we  want to find the limit of 
		\begin{eqnarray*}
			R^{n,\ell_1,k_n,\ell_2,k'_n}_{11}(t)=\sum_{j=1}^{[t/\del]} 
			\mE[ \zeta_{1}^{n,j,\ell_{1},k_{n} }\zeta_{1}^{n,j,\ell_{2},k_{n}' }  \mid \cf_{(j-1)\delta_{n}} ],
		\end{eqnarray*}
		where $\ell_1,\ell_2\geq2$, $k_n\sim \theta_1 \del^{-\kappa}$ and $k'_n \sim \theta_2\del^{-\kappa}$ with $\kappa=\frac{2H}{2H+1}$. 
		
		By It\^o's isometry, 
		\begin{align*}
			R^{n,\ell_1,k_n,\ell_2,k'_n}_{11}(t)  &=64\del^{\kappa-1}(\theta_1\theta_2)^{-1-2H}\del^{-2}\sum_{j=1}^{[t/\del]} \int_{(j-1)\del}^{j\del} \si_{(j-1)\del}^8 \\
			&\quad\times\begin{multlined}[t][0.75\textwidth] \E\biggl[ \int_{([s/\del]-(\ell_1+2)k_n+1)\del}^{([s/\del]-(\ell_1-2)k_n)\del}\xi(\tfrac{[r/\del]-[s/\del]+\ell_1 k_n}{2k_n})(W_{r}-W_{[r/\del]\del}) dW_r\\
				\times  \int_{([s/\del]-(\ell_2+2)k'_n+1)\del}^{([s/\del]-(\ell_2-2)k'_n)\del}\xi(\tfrac{[r/\del]-[s/\del]+\ell_2 k'_n}{2k'_n})(W_{r}-W_{[r/\del]\del}) dW_r\\
				\times  (W_{s}-W_{[s/\del]\del})^2 \mathrel{\Big|} \calf_{(j-1)\del} \biggr]  ds. \end{multlined}
		\end{align*}
		Further conditioning on $\calf_{[s/\del]\del}=\calf_{(j-1)\del}$, we can replace $(W_s-W_{[s/\del]\del})^2$ simply by $s-(j-1)\del$. Hereafter, we can further remove the boundaries of the two $dW_r$-integrals because $\xi(t)=0$ for $\lvert t\rvert >1$. Consequently,
		\begin{align*}
			R^{n,\ell_1,k_n,\ell_2,k'_n}_{11}(t)  &=64\del^{\kappa-1}(\theta_1\theta_2)^{-1-2H}\del^{-2}\sum_{j=1}^{[t/\del]} \si_{(j-1)\del}^8 \int_{(j-1)\del}^{j\del} (s-(j-1)\del)\\
			&\quad\times\int_\R \xi(\tfrac{[r/\del]-[s/\del]}{2k_n}+\tfrac{\ell_1}2)\xi(\tfrac{[r/\del]-[s/\del]}{2k'_n}+\tfrac{\ell_2}2)(r-[r/\del]\del)dr ds. 
		\end{align*}
		Changing $r-[s/\del]\del$ to $u$, we can write
		\begin{align*}
			R^{n,\ell_1,k_n,\ell_2,k'_n}_{11}(t)  &=64\del^{\kappa-1}(\theta_1\theta_2)^{-1-2H}\del^{-2}\sum_{j=1}^{[t/\del]} \si_{(j-1)\del}^8 \int_{(j-1)\del}^{j\del} (s-(j-1)\del)\\
			&\quad\times\int_\R \xi(\tfrac{[u/\del]}{2k_n}+\tfrac{\ell_1}2)\xi(\tfrac{[u/\del]}{2k'_n}+\tfrac{\ell_2}2)(u-[u/\del]\del)du ds\\
			&=64\del^{\kappa-1}(\theta_1\theta_2)^{-1-2H}\del^{-2}\sum_{j=1}^{[t/\del]} \si_{(j-1)\del}^8 \int_{(j-1)\del}^{j\del} (s-(j-1)\del)\\
			&\quad\times\sum_{i=-\infty}^\infty\int_{(i-1)\del}^{i\del} \xi(\tfrac{i-1}{2k_n}+\tfrac{\ell_1}2)\xi(\tfrac{i-1}{2k'_n}+\tfrac{\ell_2}2)(u-(i-1)\del)du ds. 
		\end{align*}
		Computing the $duds$-integrals and observing that $\del^{\kappa}\sum_{i=-\infty}^\infty$ and {\color{blue}$\del\sum_{j=1}^{[t/\del]}$} are  Riemann sums, we have that
		\begin{equation*}
			R^{n,\ell_1,k_n,\ell_2,k'_n}_{11}(t)  \sim16 (\theta_1\theta_2)^{-1-2H}\int_0^t \si_s^8 ds
			\int_\R \xi(\tfrac{v}{2\theta_1}+\tfrac{\ell_1}2)\xi(\tfrac{v}{2\theta_2}+\tfrac{\ell_2}2) dv.
		\end{equation*}
		
		Next, we realize that $\xi(t)$ is equal to  $-\frac14\delta^4_1\lvert x\rvert$ evaluated at $x=2t$. Thus,
		\begin{equation*}
			R^{n,\ell_1,k_n,\ell_2,k'_n}_{11}(t)  \sim  (\theta_1\theta_2)^{-2-2H}\Ga_1(t)
			\int_\R \delta^4_{\theta_1} \lvert v+\ell_1\theta_1\rvert \delta^4_{\theta_2} \lvert v+\ell_2\theta_2\rvert dv.
		\end{equation*}

		It remains to derive a closed-form expression for the integral.  By Parseval's identity and \eqref{eq:abs} (and a limit argument noting that $2\Ga(\al+1)\cos(\frac{\pi(\al+1)}2)\to \frac\pi6$ as $\al\to-4$), it is given by
		\begin{equation*}
			\frac{2}{\pi}\int_\R e^{i(\ell_1\theta_1-\ell_2\theta_2)\xi}\lvert \xi\rvert^{-4} (e^{\frac12 i \theta_1\xi}-e^{-\frac12i\theta_1\xi})^4(e^{\frac12 i \theta_2\xi}-e^{-\frac12i\theta_2\xi})^4 d\xi \\
			= \frac{1}{3}\delta^4_{\theta_1}\delta^4_{\theta_2}\lvert \ell_1\theta_1-\ell_2\theta_2\rvert^3,
		\end{equation*}
		which completes the proof of \eqref{eq:cov} for $\nu=1$.

		\medskip	
		Finally, let us consider $\nu=3$ and, as a first step,  note that 
		\begin{equation}\label{eq:R3132} 
			R^{n,\ell_1,k_n,\ell_2,k'_n}_{31,32}(t)=\sum_{j=1}^{[t/\del]} \E[  \zeta^{n,j,\ell_1,k_n}_{31} \zeta^{n,j,\ell_2,k'_n}_{32}\mid \calf_{(j-1)\del}]=0
		\end{equation}
		because $\E[W_s-W_{(j-1)\del}\mid \calf_{(j-1)\del}]=0$. Thus, it remains to find the   limits of 
		\begin{equation}\label{eq:R33}
			R^{n,\ell_1,k_n,\ell_2,k'_n}_{31|32,31|32}(t)=\sum_{j=1}^{[t/\del]} \E[  \zeta^{n,j,\ell_1,k_n}_{31|32} \zeta^{n,j,\ell_2,k'_n}_{31|32}\mid \calf_{(j-1)\del}].
		\end{equation}
		To this end, we define 
		\begin{equation}\label{eq:calg} 
			\calg_H(t)=\frac{K_H^{-1}}{(H+\frac12)(H+\frac32)} t^{H+3/2}_+,
		\end{equation}
		such that $\int_0^1 G_H(t-u)^{H+1/2} du = \calg_H(t)-\calg_H(t-1)$ for all $t\in\R$ and therefore,
		\begin{equation}\label{eq:int} 
			\int_0^1  \Delta^3_1 G_H(\tfrac{s-r}{k_n\del} + \ell -u-1) du = \delta_1^4 \calg_H(\tfrac{s-r}{k_n\del}+\ell)
		\end{equation}
		for all $\ell\in\R$. Analogously to the arguments between \eqref{eq:R22} and \eqref{eq:rho22}, it suffices to consider, instead of $R^{n,\ell_1,k_n,\ell_2,k'_n}_{31\mid 32,31\mid32}(t)$, the simpler terms
		\begin{equation}\label{eq:R3131}\begin{split}
				\wt	R^{n,\ell_1,k_n,\ell_2,k'_n}_{31,31}(t)	&= 4(k_n\del)^{-1/2-H}(k'_n\del)^{-1/2-H}\del^{-(1-\kappa)}\sum_{j=1}^{[t/\del]} \int_{(j-1)\del}^{j\del} \si_{(j-1)\del}^4\lvert\bta_{(j-1)\del}\rvert^2    \\
				&\quad\times\begin{multlined}[t][0.75\textwidth] \int_{s-k_n\del^{1-\eps}}^{(j-1)\del} (\delta_1^4 \calg_H(\tfrac{s-r}{k_n\del}+\ell_1)+\delta_1^4 \calg_H(\tfrac{s-r}{k_n\del}-\ell_1))\\
					\times (\delta_1^4 \calg_H(\tfrac{s-r}{k'_n\del}+\ell_2)+\delta_1^4 \calg_H(\tfrac{s-r}{k'_n\del}-\ell_2)) dr  (s-(j-1)\del)ds\end{multlined}
			\end{split}\raisetag{-3.5\baselineskip}\end{equation}
		and 
		\begin{equation}\label{eq:R3232}\begin{split}
				&\wt R^{n,\ell_1,k_n,\ell_2,k'_n}_{32,32}(t)	= 4(k_n\del)^{-1/2-H}(k'_n\del)^{-1/2-H}\del^{-(1-\kappa)}\sum_{j=1}^{[t/\del]} \int_{(j-1)\del}^{j\del} \si_{(j-1)\del}^4\lvert\bta_{(j-1)\del}\rvert^2    \\
				&\qquad\qquad\times \int_{r-((\ell_1+2)k_n\wedge(\ell_2+2) k'_n)\del}^{(j-1)\del} \delta_1^4 \calg_H(\ell_1-\tfrac{r-s}{k_n\del}) \delta_1^4 \calg_H(\ell_2-\tfrac{r-s}{k'_n\del})  (s-[s/\del]\del)ds dr.  
			\end{split}\raisetag{-2.75\baselineskip}\end{equation}
		In \eqref{eq:R3131},
		changing $(s-r)/\del^{1-\kappa}$ to $u$ and $s-(j-1)\del$ to $v$, we obtain
		\begin{align*}
			\wt	R^{n,\ell_1,k_n,\ell_2,k'_n}_{31,31}(t)		& =4 (k_n\del)^{-1/2-H}(k'_n\del)^{-1/2-H} \sum_{j=1}^{[t/\del]} \int_{0}^{\del} \si_{(j-1)\del}^4\lvert\bta_{(j-1)\del}\rvert^2    \\
			&\quad\times \begin{multlined}[t][0.75\textwidth]\int_{\frac{v}{\del^{1-\kappa}}}^{k_n\del^{\kappa-\eps}} (\delta_1^4 \calg_H(u\tfrac{\del^{-\kappa}}{k_n}+\ell_1)+\delta_1^4 \calg_H(u\tfrac{\del^{-\kappa}}{k_n}-\ell_1)) \\
				\times (\delta_1^4 \calg_H(u\tfrac{\del^{-\kappa}}{k'_n}+\ell_2)+\delta_1^4 \calg_H(u\tfrac{\del^{-\kappa}}{k'_n}-\ell_2)) du  \cdot vdv\end{multlined}\\
			&\sim \rho_{31,31}^{\ell_1,\theta_1,\ell_2,\theta_2} \Ga_3(t),
		\end{align*}
		where
		\begin{equation}\label{eq:rho3131}\begin{multlined}[0.9\textwidth]
				\rho_{31,31}^{\ell_1,\theta_1,\ell_2,\theta_2}	= \frac{2}{(\theta_1\theta_2)^{1/2+H}} \int_0^\infty  (\delta_1^4 \calg_H(u/\theta_1+\ell_1)+\delta_1^4 \calg_H(u/\theta_1-\ell_1)) \\
				\times(\delta_1^4 \calg_H(u/\theta_2+\ell_2)+\delta_1^4 \calg_H(u/\theta_2-\ell_2)) du.
		\end{multlined}\end{equation}
		Similarly, changing $(r-s)/\del^{1-\kappa}$ to $u$, we derive
		\begin{align*}
			\wt	R^{n,\ell_1,k_n,\ell_2,k'_n}_{32,32}(t)		& =4 (k_n\del)^{-1/2-H}(k'_n\del)^{-1/2-H} \sum_{j=1}^{[t/\del]} \int_{(j-1)\del}^{j\del} \si_{(j-1)\del}^4\lvert\bta_{(j-1)\del}\rvert^2    \\
			&\quad\times \begin{multlined}[t][0.75\textwidth]\int_{\frac{r-(j-1)\del}{\del^{1-\kappa}}}^{\frac{(\ell_1+2)k_n\wedge(\ell_2+2) k'_n}{\del^{-\kappa}}} (\delta_1^4 \calg_H(\ell_1-u\tfrac{\del^{-\kappa}}{k_n})+\delta_1^4 \calg_H(\ell_2-u\tfrac{\del^{-\kappa}}{k'_n}))\\
				\times (r-u\del^{1-\kappa}-[r/\del-u\del^{-\kappa}]\del) du  dr\end{multlined}\\
			&\sim 4 (k_n\del)^{-1/2-H}(k'_n\del)^{-1/2-H} \sum_{j=1}^{[t/\del]}  \si_{(j-1)\del}^4\lvert\bta_{(j-1)\del}\rvert^2    \\
			&\quad\times \begin{multlined}[t][0.75\textwidth]\int_{0}^{\frac{(\ell_1+2)k_n\wedge(\ell_2+2) k'_n}{\del^{-\kappa}}} (\delta_1^4 \calg_H(\ell_1-u\tfrac{\del^{-\kappa}}{k_n})+\delta_1^4 \calg_H(\ell_2-u\tfrac{\del^{-\kappa}}{k'_n}))\\
				\times \int_{(j-1)\del}^{j\del}(r-u\del^{1-\kappa}-[r/\del-u\del^{-\kappa}]\del)   drdu.\end{multlined}
		\end{align*}
		Note that the last $dr$-integral equals $\int_0^{\del} r dr= \frac12\del^2$, so that 
		\[ 	\wt R^{n,\ell_1,k_n,\ell_2,k'_n}_{32,32}(t)	\sim \rho_{32,32}^{\ell_1,\theta_1,\ell_2,\theta_2} \int_0^t \si^4_s\lvert \bta_s\rvert^2 ds, \]
		where
		\begin{equation}\label{eq:rho3232} 
			\rho_{32,32}^{\ell_1,\theta_1,\ell_2,\theta_2}	= \frac{2}{(\theta_1\theta_2)^{1/2+H}} \int_0^{(\ell_1+2)\theta_1\wedge(\ell_2+2)\theta_2}  \delta_1^4 \calg_H(\ell_1-u/\theta_1)\delta_1^4 \calg_H(\ell_2-u/\theta_2) du.
		\end{equation}
		Using the fact that $\delta_1^4 \calg_H(t)=0$ for $t\leq -2$, we can extend the previous integral up to $+\infty$, which shows that
		\begin{multline*} \rho_{31,31}^{\ell_1,\theta_1,\ell_2,\theta_2}+\rho_{32,32}^{\ell_1,\theta_1,\ell_2,\theta_2}= \frac{2}{(\theta_1\theta_2)^{1/2+H}} \int_\R  (\delta_1^4 \calg_H(u/\theta_1+\ell_1)+\delta_1^4 \calg_H(u/\theta_1-\ell_1)) \\
			\times(\delta_1^4 \calg_H(u/\theta_2+\ell_2)+\delta_1^4 \calg_H(u/\theta_2-\ell_2)) du. \end{multline*}
		
		We want to show that this is exactly $\ga^{\ell_1,\theta_1,\ell_2,\theta_2}_3(H)$, which would then finish the proof of \eqref{eq:cov}.
		Switching to the Fourier domain, we use \eqref{eq:fml},  \eqref{eq:abs}, \eqref{eq:fml-prod} and \eqref{eq:gH} to obtain
		\begin{align*}
			&\int_\R  \delta_1^4 \calg_H(u/\theta_1+\ell_1)	\delta_1^4 \calg_H(u/\theta_2+\ell_2) du\\ &~=\frac{K_H^{-2}}{(\theta_1\theta_2)^{H+3/2}(H+\frac12)^2(H+\frac32)^2} \int_\R \delta_{\theta_1}^4 (u+\ell_1\theta_1)^{H+3/2}_+	\delta_{\theta_2}^4  (u+\ell_2\theta_2)^{H+3/2}_+ du \\
			&~=\begin{multlined}[t][0.95\textwidth]\frac{K_H^{-2}\Ga(H+\frac12)^2}{2\pi(\theta_1\theta_2)^{H+3/2}} \int_\R e^{i\xi(\ell_1\theta_1-\ell_2\theta_2)}e^{-i\pi(H+5/2)/2}(\xi-i0)^{-H-5/2}\\
				\times e^{i\pi(H+5/2)/2}(\xi+i0)^{-H-5/2}(e^{\frac12 i \theta_1\xi}-e^{-\frac12i\theta_1\xi})^4(e^{\frac12 i \theta_2\xi}-e^{-\frac12i\theta_2\xi})^4 d\xi\end{multlined}\\
			&~= \frac{\sin(\pi H)\Ga(2H+1)}{2\pi(\theta_1\theta_2)^{H+3/2}} \int_\R e^{i\xi(\ell_1\theta_1-\ell_2\theta_2)} \lvert \xi\rvert^{-2H-5}(e^{\frac12 i \theta_1\xi}-e^{-\frac12i\theta_1\xi})^4(e^{\frac12 i \theta_2\xi}-e^{-\frac12i\theta_2\xi})^4 d\xi\\
			&~=\frac{\sin(\pi H)\Ga(2H+1)\Ga(-2H-4)\cos(\pi(H+2))}{\pi}(\theta_1\theta_2)^{-H-3/2} \delta^4_{\theta_1}\delta^4_{\theta_2} \lvert \ell_2\theta_2-\ell_1\theta_1 \rvert^{2H+4}
		\end{align*}
		for $H\in(0,\frac12)$.
		The last fraction is equal to $-1/(32(H+\frac12)(H+1)(H+\frac32)(H+2))$, which shows that $	\rho_{31,31}^{\ell_1,\theta_1,\ell_2,\theta_2}+\rho_{32,32}^{\ell_1,\theta_1,\ell_2,\theta_2}=\ga^{\ell_1,\theta_1,\ell_2,\theta_2}_3(H)$
		for $H\in(0,\frac12)$. As before, the expression for $H=\frac12$ can be obtained by letting $H\to\frac12$, and since there is no singularity at $H=\frac12$ in the formula defining $\ga^{\ell_1,\theta_1,\ell_2,\theta_2}_3(H)$, it remains valid for $H=\frac12$.
	\end{proof}

	\begin{proof}[Proof of Equation~\eqref{eq:cov2}]
		Let us start by showing that
		\[ R^{n,\ell_1,k_n,\ell_2,k'_n}_{2,31|32}(t)=\sum_{j=1}^{[t/\del]} \E[  \zeta^{n,j,\ell_1,k_n}_2   \zeta^{n,j,\ell_2,k'_n}_{31|32} \mid \calf_{(j-1)\del}] \stackrel{\P}{\longrightarrow}0.\]
		By \eqref{eq:wtzeta2} and \eqref{eq:wtzeta3} and It\^o's isometry, we have 
		\begin{align*}
			R^{n,\ell_1,k_n,\ell_2,k'_n}_{2,31}(t)	&= -2\sum_{j=1}^{[t/\del]} (k'_n\del)^{-1/2-H}\del^{-(1-\kappa)} \int_{(j-1)\del}^{j\del} \si^2_{(j-1)\del}\eta^2_{(j-1)\del}\\
			&\quad\times\begin{multlined}[t][0.75\textwidth]\int_{s-k'_n\del^{1-\eps}}^{(j-1)\del} \int_0^1 \Bigl\{ \Delta^3_1 G_H(\tfrac{s-r}{k'_n\del} + \ell_2 -u-1)  \\
				+\Delta^3_1 G_H(\tfrac{s-r}{k'_n\del} - \ell_2 -u-1)\Bigr\}du \bta_{(j-1)\del} d\bW_r\end{multlined} \\
			&\quad\times\begin{multlined}[t][0.75\textwidth] \int_{(j-1)\del}^s\int_{-2}^\infty \Bigl\{\Delta^{2}_{1}G_H(v)\Delta^{2}_{1}G_H(v+\tfrac{s-w}{k_n\del}+\ell_1)  \\
				+ \Delta^2_1 G_H(v)\Delta^2_1G_H(v+\tfrac{s-w}{k_n\del}-\ell_1)\Bigr\} dv dw ds.\end{multlined}
		\end{align*}
		For each $j$, we know from the analysis of $R^{n,\ell_1,k_n,\ell_2,k'_n}_{22}$ and $R^{n,\ell_1,k_n,\ell_2,k'_n}_{31,31}$ that $\zeta^{n,j,\ell_1,k_n}_2$ and $\zeta^{n,j,\ell_2,k'_n}_{31}$ are of order $O_\P(\del^{1/2})$, uniformly in $i$. Therefore, we are free to modify terms in the previous display as long as it leads to an asymptotically vanishing error. For example, for any fixed $j$, we may replace $\si_{(j-1)\del}^2\eta^2_{(j-1)\del}$ and $\bta_{(j-1)\del}$ by $\si_{(j-1-k'_n\del^{-\eps})\del}^2\eta^2_{(j-1-k'_n\del^{-\eps})\del}$ and $\bta_{(j-1-k'_n\del^{-\eps})\del}$, respectively. Once we have done so, the resulting term, for fixed $j$, will be $\calf_{(j-1)\del}$-measurable with vanishing $\calf_{(j-1-k'_n\del^{-\eps})\del}$-conditional expectation. Thus, by a martingale size estimate (see \cite[Appendix~A]{CDL22}), the sum over $j$ will be of magnitude $O_\P(\del^{-(\kappa+\eps+1)/2} \del^{1/2}\del^{1/2}) =o_\P(1)$, proving $R^{n,\ell_1,k_n,\ell_2,k'_n}_{2,31}\approx 0$.
		
		The reasoning for $R^{n,\ell_1,k_n,\ell_2,k'_n}_{2,32}$ is similar. Again by It\^o's isometry,
		\begin{align*}
			R^{n,\ell_1,k_n,\ell_2,k'_n}_{2,32}(t)	&= -2\sum_{j=1}^{[t/\del]} (k'_n\del)^{-1/2-H}\del^{-(1-\kappa)} \int_{(j-1)\del}^{j\del} \si^2_{(j-1)\del}\lvert\bta_{(j-1)\del}\rvert^2 \\
			&\quad\times\int_{r-(\ell_2+1)k'_n\del}^{(j-1)\del} \int_0^1   \Delta^3_1 G_H(\ell_2-\tfrac{r-s}{k'_n\del} -u-1)    du (W_s-W_{[s/\del]\del})dW_s  \\
			&\quad\times \begin{multlined}[t][0.75\textwidth]\int_{r-k_n\del^{1-\eps}}^{(j-1)\del}\int_{-2}^\infty \Delta^{2}_{1}G_H(v)\Bigl\{\Delta^{2}_{1}G_H(v+\tfrac{r-w}{k_n\del}+\ell_1) \\
				+  \Delta^2_1G_H(v+\tfrac{r-w}{k_n\del}-\ell_1)\Bigr\} dv \bta_{(j-1)\del} d\bW_w dr.\end{multlined}
		\end{align*}
		We can now use integration by parts to expand the product of the $dW_s$-integral and the $d\bW_w$-integral. As in the analysis of $R^{n,\ell_1,k_n,\ell_2,k'_n}_{2,31}$ above, the martingale terms can be shown to be negligible. So only the quadratic variation part remains and 
		\begin{align*}
			&R^{n,\ell_1,k_n,\ell_2,k'_n}_{2,32}(t)	\approx -2\sum_{j=1}^{[t/\del]} (k'_n\del)^{-1/2-H}\del^{-(1-\kappa)} \int_{(j-1)\del}^{j\del} \si^2_{(j-1)\del}\lvert\bta_{(j-1)\del}\rvert^2 \eta_{(j-1)\del} \\
			&\quad\times\int_{r-(\ell_2+1)k'_n\del}^{(j-1)\del} \int_0^1   \Delta^3_1 G_H(\ell_2-\tfrac{r-s}{k'_n\del} -u-1)    du (W_s-W_{[s/\del]\del}) \\
			&\quad\times \int_{-2}^\infty \Delta^{2}_{1}G_H(v)\Bigl\{\Delta^{2}_{1}G_H(v+\tfrac{r-s}{k_n\del}+\ell_1)  +  \Delta^2_1G_H(v+\tfrac{r-s}{k_n\del}-\ell_1)\Bigr\} dv ds dr.
		\end{align*}
		Now we apply the same trick as before: we first shift the   index of $\si^2_{(j-1)\del}\lvert\bta_{(j-1)\del}\rvert^2 \eta_{(j-1)\del}$ to $(j-1-(\ell_2+1)k'_n)\del$ and then realize that the conditional expectation of the resulting expression given $\calf_{(j-1-(\ell_2+1)k'_n)\del}$ is zero. Thus, by another martingale size estimate, we obtain $R^{n,\ell_1,k_n,\ell_2,k'_n}_{2,32}\approx 0$. Since the proof of 
		\begin{eqnarray*}
			R^{n,\ell_1,k_n,\ell_2,k'_n}_{12\mid 1,31}(t)=\sum_{j=1}^{[t/\del]} \E[  \zeta^{n,j,\ell_1,k_n}_1   \zeta^{n,j,\ell_2,k'_n}_{2\mid 31} \mid \calf_{(j-1)\del}] \stackrel{\P}{\longrightarrow}0
		\end{eqnarray*}
		is very similar, we omit the details and leave it to the reader. Lastly, by It\^o's isometry, we have
		\begin{equation*}
			R^{n,\ell_1,k_n,\ell_2,k'_n}_{1,32}(t) =\sum_{j=1}^{[t/\del]} \E[  \zeta^{n,j,\ell_1,k_n}_1   \zeta^{n,j,\ell_2,k'_n}_{32} \mid \calf_{(j-1)\del}]\equiv 0. \qedhere
		\end{equation*}
	\end{proof}
	
\end{appendix}

\begin{acks}[Acknowledgments]
 The authors would like to thank the Associate Editor and two referees for their careful reading of the paper and for their constructive comments, which led to significant improvements of the paper. The authors would also like to thank Mikko Pakkanen for sharing the code from the paper \cite{BLP17}.
\end{acks}

\begin{funding}
	Yanghui Liu is supported by the PSC-CUNY Award 64353-00 52. Mathieu Rosenbaum and Gr\'egoire Szymanski gratefully acknowledge the financial support of the \'Ecole Polytechnique chairs {\it Deep Finance and Statistics} and {\it Machine Learning and Systematic Methods}. 
\end{funding}

\bibliographystyle{imsart-number2}
\bibliography{RoughVol.bib}

\end{document}